\documentclass{amsart}[11pt]

\usepackage[margin=1in]{geometry}

\usepackage{amsmath,amsthm,amssymb, amsfonts,mathtools}
\usepackage{latexsym}
\usepackage{url}
\usepackage{xcolor}
\usepackage{graphicx}
\usepackage{float}
\usepackage{caption}
\usepackage{subcaption}
\usepackage{verbatim}
\usepackage{enumitem}
\usepackage{amsrefs}
\usepackage{array}
\usepackage{booktabs}
\usepackage{color}
\usepackage[utf8]{inputenc}
\usepackage{tcolorbox}
\usepackage{bbm}
\usepackage{commath}
\usepackage{enumitem}
\usepackage{hyperref}
\usepackage{bm}

\newcommand{\R}{\mathbb{R}}
\newcommand{\E}{\mathbb{E}}
\newcommand{\F}{\mathcal{F}}
\newcommand{\Prob}{\mathbb{P}}
\newcommand{\1}{\mathbbm{1}}
\newcommand{\bb}[1]{\mathbb{#1}}
\newcommand{\br}[1]{\lbrace #1 \rbrace}

\newcommand{\toinf}{\rightarrow\infty}
\newcommand{\tto}{\rightarrow}

\newcommand{\W}{\Omega}
\newcommand{\mc}[1]{\mathcal{ #1}}

\DefineSimpleKey{bib}{arxiveprint}
\DefineSimpleKey{bib}{arxivid}
\DefineSimpleKey{bib}{arxivclass}
\makeatletter
\catcode`\'=11  \def\arxiveprint{
    \resolve@inner{\bib@arxiveprint}
  }
  \def\bib@arxiveprint#1{
    \begingroup
        #1\relax
        \bib@resolve@xrefs
        \bib@field@patches
        \bib'setup
        \let\PrintPrimary\@empty
        {
          \IfEmptyBibField{arxivid}{\url{https://arxiv.org/}}
          {
            \href{https://arxiv.org/abs/\bib'arxivid}{\nolinkurl{arXiv:\bib'arxivid}}
            \IfEmptyBibField{arxivclass}{}{~\nolinkurl{[\bib'arxivclass]}}          }
        }\bib'transition
        \setbib@@
    \endgroup
  }
\catcode`\'=12\makeatother
\BibSpec{arxiv}{
    +{}  {\PrintAuthors}                {author}
    +{,} { \textit}                     {title}
    +{.} { }                            {part}
    +{:} { \textit}                     {subtitle}
    +{,} { \PrintContributions}         {contribution}
    +{.} { \PrintPartials}              {partial}
    +{,} { }                            {journal}
    +{}  { \textbf}                     {volume}
    +{}  { \PrintDatePV}                {date}
    +{,} { \issuetext}                  {number}
    +{,} { \eprintpages}                {pages}
    +{,} { }                            {status}
    +{,} { \PrintDOI}                   {doi}
    +{,} { available at \eprint}        {eprint}
    +{,} { available at \arxiveprint}   {arxiveprint}
    +{}  { \parenthesize}               {language}
    +{}  { \PrintTranslation}           {translation}
    +{;} { \PrintReprint}               {reprint}
    +{.} { }                            {note}
    +{.} {}                             {transition}
    +{}  {\SentenceSpace \PrintReviews} {review}
}

\newtheorem{thm}{Theorem}[section]
\newtheorem{theo}[thm]{Theorem}
\newtheorem{corollary}[thm]{Corollary}

\newtheorem{proposition}[thm]{Proposition}

\newtheorem{defi}[thm]{Definition}

\theoremstyle{remark}
\newtheorem{remark}[thm]{Remark}
\newtheorem{example}[thm]{Example}

\newtheorem{lemma}[thm]{Lemma}

\author{Z.W. Bezemek, and K. Spiliopoulos}
\address{Boston University, Department of Mathematics and Statistics\\ 111 Cummington Mall, Boston, MA 02215, USA}
\email[Zachary William Bezemek]{bezemek@bu.edu}
\email[Konstantinos Spiliopoulos]{kspiliop@bu.edu}
\thanks{This work has been partially supported by the National Science Foundation (DMS 2107856) and Simons Foundation Award  672441}
\title{Rate of homogenization for fully-coupled McKean-Vlasov SDEs}
\date{\today}

\allowdisplaybreaks
\begin{document}

\begin{abstract}
We consider a  fully-coupled slow-fast system of McKean-Vlasov SDEs with full dependence on the slow and fast component and on the law of the slow component and derive convergence rates to its homogenized limit.  We do not make periodicity assumptions, but we impose conditions on the fast motion to guarantee ergodicity. In the course of the proof we obtain related ergodic theorems and we gain results on the regularity of Poisson type of equations and of the associated Cauchy-Problem on the Wasserstein space that are of independent interest.
\end{abstract}
\maketitle
\section{Introduction}

The goal of this paper is to study the behavior as $\epsilon\rightarrow 0$ of the system of slow-fast McKean-Vlasov SDEs
\begin{align}\label{eq:slow-fastMcKeanVlasov}
X^{\epsilon,\nu}_t &= \eta+\int_0^t\biggl[\frac{1}{\epsilon}b(X^{\epsilon,\nu}_s,Y^{\epsilon,\nu}_s,\mc{L}(X^{\epsilon,\nu}_s))+ c(X^{\epsilon,\nu}_s,Y^{\epsilon,\nu}_s,\mc{L}(X^{\epsilon,\nu}_s)) \biggr]ds + \int_0^t\sigma(X^{\epsilon,\nu}_s,Y^{\epsilon,\nu}_s,\mc{L}(X^{\epsilon,\nu}_s))dW_s\\
Y^{\epsilon,\nu}_t & = \zeta+\int_0^t\frac{1}{\epsilon}\biggl[\frac{1}{\epsilon}f(X^{\epsilon,\nu}_s,Y^{\epsilon,\nu}_s,\mc{L}(X^{\epsilon,\nu}_s))+ g(X^{\epsilon,\nu}_s,Y^{\epsilon,\nu}_s,\mc{L}(X^{\epsilon,\nu}_s)) \biggr]dt \nonumber\\
&+ \frac{1}{\epsilon}\int_0^t\biggl[\tau_1(X^{\epsilon,\nu}_s,Y^{\epsilon,\nu}_s,\mc{L}(X^{\epsilon,\nu}_s))dW_s+\int_0^t\tau_2(X^{\epsilon,\nu}_s,Y^{\epsilon,\nu}_s,\mc{L}(X^{\epsilon,\nu}_s))dB_s\biggr].\nonumber
\end{align}

Equation (\ref{eq:slow-fastMcKeanVlasov}) is defined on a filtered probability space $(\W,\F,\Prob,\br{\F_t})$ with $\br{\F_t}$ satisfying the usual conditions, where $b,c,f,g,:\R^d\times\R^d\times\mc{P}_2(\R^d)\tto \R^d$ $\sigma,\tau_1,\tau_2:\R^d\times\R^d\times\mc{P}_2(\R^d)\tto \R^{d\times m}$, $B_t,W_t$ are independent standard $m$-dimensional $\br{\F_t}$-Brownian motions, $\eta\in L^2(\W,\F_0,\Prob;\R^d)$ with $\eta\sim \nu$, $\zeta\in L^p(\W,\F_0,\Prob;\R^d)$ for all $p>0$, and $(\eta,\zeta)$ is independent of $(W,B)$. Here and throughout $\mc{P}_2(\R^d)$ denotes the space of probability measures on $\R^d$ with finite second moment, equipped with the 2-Wasserstein metric (see Appendix \ref{Appendix:LionsDifferentiation}).

 Note the superscript $\nu\in\mc{P}_2(\R^d)$ on $(X^{\epsilon,\nu},Y^{\epsilon,\nu})$, which is parameterizing the distribution of $X^{\epsilon,\nu}_0=\eta$. This parameterization is important for the formulation of solutions to the Cauchy-Problem on Wasserstein Space \eqref{eq:W2CauchyProblem}, which is employed to prove our main result, Theorem \ref{theo:mckeanvlasovaveraging}. See Section \ref{sec:onthecauchyproblem} and in particular Remark \ref{remark:McKeanVlasovFlowProperty} for further discussion of this choice of parameterization and its importance.

The theory of averaging for diffusion process with coefficients that do not depend on the law of the solution itself, i.e., that do not depend on $\mc{L}(X^{\epsilon,\nu})$, is a classical result by now and had been studied under different assumptions and settings, see for example \cites{Bensoussan,BorkarGaitsgory,PV1,PV2,PS,RocknerFullyCoupled,Spiliopoulos2014Fluctuations} to name a few. Existing averaging results for slow-fast McKean-Vlasov SDEs can be found  in \cites{RocknerMcKeanVlasov,RocknerHolderContinuous,HLL,KSS,XLLM}. In particular,  only systems where $L^{2}$ convergence rates can be found, possible for instance when $b=0$ and $\sigma(x,y,\mu)=\sigma(x,\mu)$, have been considered in the McKean-Vlasov setting, leaving the fully-coupled cases unsolved.  Even for standard diffusions (where the coefficients do not depend on the law of the solution), the only existing result for rates of convergence in distribution can be found in Theorem 2.3 of \cite{RocknerFullyCoupled}. We close this literature review mentioning the very recent preprints \cites{QW,LWX,HLS}, dealing with various aspects of averaging for McKean-Vlasov SDE systems, that appeared on arXiv after this paper had appeared on arXiv and was submitted to the journal.

Let $a(x,y,\mu)=\frac{1}{2}[\tau_1\tau_1^\top+\tau_2\tau_2^\top](x,y,\mu)$. For $x\in \R^d,\mu\in \mc{P}_2(\R^d)$, we define the differential operator $\mc{L}_{x,\mu}$ acting on $\phi \in C_b^2(\R^d)$ by

\begin{align}\label{eq:frozengeneratormold}
\mc{L}_{x,\mu}\phi(y) = f(x,y,\mu)\cdot\nabla \phi(y)+a(x,y,\mu):\nabla^2\phi(y).
\end{align}

This is the generator obtained from considering the $O(1/\epsilon^2)$ terms from the generator of $(X^{\epsilon,\nu}_t,Y^{\epsilon,\nu}_t)$ and ``freezing'' the terms associated to the slow process at fixed $x\in\R^d$ and $\mu\in\mc{P}_2(\R^d)$.

Under assumptions \ref{assumption:uniformellipticity} and \ref{assumption:retractiontomean} below, by \cite{PV1} Proposition 1 (see also \cite{Veretennikov1987}), there exists a $\pi(\cdot;x,\mu)$ which is the unique probability measure solving the distributional equation
\begin{align}\label{eq:invariantmeasureold}
\mc{L}_{x,\mu}^*\pi=0.
\end{align}
Moreover, all moments of $\pi$ are bounded uniformly in $x$ and $\mu$.

For $\pi$ as in Equation \eqref{eq:invariantmeasureold}, under the standard centering condition
\begin{align}\label{eq:centeringconditionold}
\int_{\R^d} b(x,y,\mu)\pi(dy;x,\mu)=0,\forall x\in\R^d,\mu\in\mc{P}_2(\R^d),
\end{align}
and other technical assumption to be stated later on (see Assumptions \ref{assumption:uniformellipticity}-\ref{assumption:centeringcondition}),
by Lemma \ref{lemma:Ganguly1DCellProblemResult} we may also consider $\Phi$ the unique classical solution to:
\begin{align}\label{eq:cellproblemold}
\mc{L}_{x,\mu}\Phi_k(x,y,\mu) &= -b_k(x,y,\mu),k\in \br{1,...,d} \\
\int_{\R^d}\Phi(x,y,\mu)\pi(dy;x,\mu)&=0. \nonumber
\end{align}

Define:
\begin{align}\label{eq:limitingcoefficients}
\gamma(x,y,\mu)& \coloneqq \gamma_1(x,y,\mu)+c(x,y,\mu)\\
\gamma_1(x,y,\mu)&\coloneqq \partial_x\Phi(x,y,\mu)b(x,y,\mu)+\partial_y\Phi(x,y,\mu)g(x,y,\mu)+\tau_1(x,y,\mu)\sigma^\top(x,y,\mu):\partial_x\partial_y\Phi(x,y,\mu) \nonumber \\
D(x,y,\mu) & \coloneqq D_1(x,y,\mu)+D^\top_1(x,y,\mu)+\frac{1}{2}\sigma(x,y,\mu)\sigma^\top(x,y,\mu) \nonumber\\
D_1(x,y,\mu)& \coloneqq \frac{1}{2}[b(x,y,\mu)\otimes \Phi(x,y,\mu)+\partial_y\Phi(x,y,\mu)\tau_1(x,y,\mu)\sigma^\top(x,y,\mu)] \nonumber
\end{align}
and
\begin{align}\label{eq:averagedlimitingcoefficients}
\bar{\gamma}(x,\mu) &\coloneqq \biggl[\int_{\R^d} \gamma(x,y,\mu) \pi(dy;x,\mu)\biggr]\\
\bar{D}(x,\mu) &\coloneqq\biggl[\int_{\R^d} D(x,y,\mu) \pi(dy;x,\mu)\biggr]\nonumber.
\end{align}
Here
\begin{align*}
\tau_1\sigma^\top :\partial_x\partial_y\Phi(x,y,\mu) \coloneqq [\tau_1\sigma^\top :\partial_x\partial_y\Phi_1(x,y,\mu),...,\tau_1\sigma^\top :\partial_x\partial_y\Phi_d(x,y,\mu)]^\top.
\end{align*}

In this paper, we will establish a rate of convergence of $\mc{L}(X^{\epsilon,\nu}_t)$ to $\mc{L}(X^\nu_t)$ in terms of sufficiently smooth test functions on the space $\mc{P}_2(\R^d)$, where $X^\nu_t$ satisfies the averaged McKean-Vlasov SDE:
\begin{align}\label{eq:averagedMcKeanVlasov}
X^\nu_t &= \eta^2+\int_0^t \bar{\gamma}(X^{\nu}_s,\mc{L}(X^{\nu}_s))ds+\int_0^t\sqrt{2}\bar{D}^{1/2}(X^{\nu}_s,\mc{L}(X^{\nu}_s))dW^2_s.
\end{align}

Here the equation is posed on a possibly different filtered probability space satisfying the usual conditions and supporting a $d$-dimensional Brownian motion $W^2$, $\eta^2$ is random variable on this new probability space independent from $W^2$ and equal in distribution to $X^{\epsilon,\nu}_0=\eta$ - that is $\eta^2\sim \nu \in\mc{P}_2(\R^d)$ - and once again the superscript $\nu$ in the notation $X^\nu$ is parameterizing the distribution of the initial condition. In addition, $\bar{D}^{1/2}(x,\mu)$ is the unique positive semi-definite matrix such that $\bar{D}^{1/2}(x,\mu)\bar{D}^{1/2}(x,\mu)=\bar{D}(x,\mu)$. Note that while $\bar{D}$ is symmetric, it is not clear a priori that it is positive semi-definite. Thus we make the following remark:

\begin{remark}\label{remark:alternateformofdiffusion}
One can find that the diffusion coefficient $\bar{D}$ can be written in the alternative form
\begin{align}\label{eq:alternativediffusion}
\bar{D}(x,\mu) & = \frac{1}{2}\int_{\R^d} \left(\partial_y\Phi(x,y,\mu)\tau_2(x,y,\mu)\tau^\top_2(x,y,\mu)[\partial_y\Phi(x,y,\mu)]^\top\right. \\
&\left.+ [\sigma(x,y,\mu)+\partial_y\Phi(x,y,\mu)\tau_1(x,y,\mu)][\sigma(x,y,\mu)+\partial_y\Phi(x,y,\mu)\tau_1(x,y,\mu)]^\top\right)\pi(dy;x,\mu),\nonumber
\end{align}
and hence is both symmetric and positive semi-definite. This reduces to showing
\begin{align*}
&\int_{\R^d}\left(b(x,y,\mu)\otimes \Phi(x,y,\mu)+\Phi(x,y,\mu)\otimes b(x,y,\mu)\right)\pi(dy;x,\mu)\\
&=2\int_{\R^d}\partial_y\Phi(x,y,\mu)a(x,y,\mu)[\partial_y\Phi(x,y,\mu)]^\top\pi(dy;x,\mu),
\end{align*}
which can be seen via an integration by parts argument as per \cite{PS} Remarks 11.4/11.5.
\end{remark}

Further than simply establishing this new averaging principle, we also establish a rate of convergence in distribution. Our main result is Theorem \ref{theo:mckeanvlasovaveraging}, while Corollary \ref{corollary:lineartestfunction} specializes the main result in the important case of convergence for linear test functionals of the law. We make our results concrete in Section \ref{SS:Example} for a class of Aggregation-Diffusion equations where we also note that the effect of the multiple scales is to decrease the magnitude of the effective interaction potential in all directions.  When using classical methods, such as the martingale problem, to show convergence in distribution of stochastic process, rates of convergence can only possibly be found after serious added effort. An added difficulty of the general setting studied in this paper is that $\Phi$ solving (\ref{eq:cellproblemold}) depends on the measure parameter $\mu$, and thus terms involving derivatives of $\Phi$ with respect to $\mu$ appear in the prelimit expression for functions of the slow process and its law; see Proposition \ref{prop:purpleterm2} (and analogously Proposition \ref{prop:newpurpleterm}). These terms are handled using a novel coupling argument and an extended Poisson equation (\ref{eq:doublecorrectorproblem}) (a doubled corrector problem).

To our knowledge, the only result providing rates of convergence in distribution in the fully-coupled setting for standard diffusion processes (which do not depend on their law, i.e. functional derivatives with respect to $\mu$ and terms that we have to deal with in Propositions \ref{prop:purpleterm2} and \ref{prop:newpurpleterm} do not appear there), are found in \cite{RocknerFullyCoupled} as Theorem 2.3. The insight provided by the proof of that Theorem is to write the difference $\E\biggl[\phi(X^\epsilon_t)- \phi(X_t) \biggr] $ in terms of the solution to the Cauchy problem (backward Kolmogorov equation) associated to the limiting system (here $X^\epsilon$ and $X$ are solutions of a standard SDEs, not McKean-Vlasov SDEs).

Though not stated explicitly in \cite{RocknerFullyCoupled}, the constant such that $\sup_{t\in[0,T]}\E\biggl[\phi(X^\epsilon_t)- \phi(X_t) \biggr] \leq C\epsilon$ (considering here Regime 4 and $\vartheta=1$ in Theorem 2.3) can be seen to depend linearly on the norm of the test function $\phi$ in an appropriate function space. Viewing $\mc{L}(X^\epsilon_t)$ and $\mc{L}(X_t)$ as elements of the dual of this space, such estimates on the ``operator norm'' of these probability measures are key in establishing tightness for fluctuation processes which establish a functional CLT related to the propagation of chaos for standard McKean-Vlasov SDEs \cites{HM,LossFromDefault,DLR,KX,FM}. This is, in fact, a key source of inspiration for this work, as it provides a needed rate of convergence in distribution for an intermediary Slow-Fast McKean-Vlasov process to its averaged limit in the proof of tightness of the fluctuations process for which a large deviations principle for is established in \cite{BS_MDP2022}.

As we will see, it is greatly beneficial in the McKean-Vlasov setting to consider, in lieu of the standard backward Kolmogorov equation associated to the averaged dynamics \eqref{eq:averagedMcKeanVlasov}, the associated Cauchy-Problem on Wasserstein space:
\begin{align}\label{eq:W2CauchyProblem}
\dot{U}(t,\mu)& = \int_{\R^d}\bar{\gamma}(z,\mu)\cdot\partial_\mu U(t,\mu)[z] + \bar{D}(z,\mu):\partial_z\partial_\mu U(t,\mu)[z]\mu(dz),t\in (0,\tau],\mu\in\mc{P}_2\\
U(0,\mu)& = G(\mu).\nonumber
\end{align}

The derivatives in the measure argument in the above equation are in the sense of Lions \cite{NotesMFG}. For the reader's convenience, we have included in Appendix \ref{Appendix:LionsDifferentiation} a brief review on differentiation of functions on spaces of measures. For a more comprehensive exposition on this, we refer the interested reader to \cite{CD} Chapter 5.

Such equations were originally studied in \cite{BLPR}. They have been used to study propagation of chaos rates in \cite{DF2}, and a related but different PDE on Wasserstein space originally posed in \cite{CCD} is used in the study Mean Field Games. We will use the recent result of \cite{CST} Theorem 2.15, which extends regularity of solutions to \eqref{eq:W2CauchyProblem} beyond the two derivatives usually needed for these applications (stated here as Lemma \ref{lem:cauchyproblemregularity}).

To our knowledge, this is the first application of the Cauchy problem on Wasserstein space used to establish rates of convergence of the law of one McKean-Vlasov SDE to another. One benefit of our proof method via the use of the Cauchy Problem on Wasserstein space is that it allows for non-linear test functions on the space of measures, so that a Corollary of our proof method is somewhat of an extension of the current results on rates of averaging for Fully-Coupled standard diffusions - see Remark \ref{remark:nomeasuredependence}. For more discussion of the reasoning behind and benefits of using solutions of Equation \eqref{eq:W2CauchyProblem} over solutions to the standard backward Kolmogorov equation, see Section \ref{sec:onthecauchyproblem}.

The rest of the paper is organized as follows. In Section \ref{subsec:notationandtopology} we go over notation and the assumptions that hold throughout the paper. Section \ref{subsection:statementofmainresults} contains our main results together with illustrative examples. In Section \ref{section:mckeanvlasovergodictheorems} we present ergodic theorems relevant to the behavior of the system (\ref{eq:slow-fastMcKeanVlasov}) as $\epsilon\rightarrow 0$. Section \ref{sec:onthecauchyproblem} discusses in detail the Cauchy Problem \eqref{eq:W2CauchyProblem}. The proof of Theorem \ref{theo:mckeanvlasovaveraging} is in Section \ref{section:proofofmainresults}. Conclusions and a discussion on future work is the content of Section \ref{S:ConclusionsFutureWork}. The Appendix contains a number of technical results used in the paper. In particular, Appendix \ref{sec:regularityofthecellproblem} contains regularity results on the Poisson equations studied in this paper. In Appendix \ref{Appendix:LionsDifferentiation}, we recall some notation and terminology associated to differentiation of functions on spaces of measures.

\section{Notation and Assumptions}\label{subsec:notationandtopology}
Let $\bm{X}$ and $\bm{Y}$ be a Polish spaces, and $(\tilde\W,\tilde\F,\mu)$ be a measure space. We will denote by $\mc{P}(\bm{X})$ the space of probability measures on $\bm{X}$ with the topology of weak convergence, $\mc{P}_2(\bm{X})\subset \mc{P}(\bm{X})$ the space of square integrable probability measures on $\bm{X}$ with the 2-Wasserstein metric (see Definition \ref{def:lionderivative}), $\mc{B}(\bm{X})$ the Borel $\sigma$-field of $\bm{X}$, $C(\bm{X};\bm{Y})$ the space of continuous functions from $\bm{X}$ to $\bm{Y}$, $C_b(\bm{X})$ the space of bounded, continuous functions from $\bm{X}$ to $\R$ with norm $\norm{\psi}_\infty\coloneqq \sup_{x\in\bm{X}}|\psi(x)|$, and $L^p(\tilde\W,\tilde\F,\mu;\R^d)$ the space of $p$-integrable functions on $(\tilde\W,\tilde\F,\mu)$ with values in $\R^d$ (where if $\tilde\W =\bm{X}$ and no $\sigma$-algebra is provided we assume it is $\mc{B}(\bm{X})$). $C^k_b(\R^d)$ for $k\in\bb{N}$ will note the space of functions with $k$ continuous and bounded derivatives on $\R^d$, with norm $\norm{\psi}_{C^k_b(\R^d)} = \sum_{j=0}^k\norm{\nabla^j\psi}_\infty$, and $C^{1,k}_b([0,T]\times\R^d)$ will denote continuous functions $\psi$ on $[0,T]\times\R^d$ with a continuous, bounded time derivative on $(0,T)$, denoted $\dot{\psi}$, such that $\norm{\psi}_{C^{1,k}_b([0,T]\times\R^d)}\coloneqq \sup_{t\in[0,T],x\in\R^d} |\dot{\psi}(t,x)|+\sup_{t\in[0,T]}\norm{\psi(t,\cdot)}_{C^k_b(\R^d)}<\infty$. $C^k_{b,L}(\R^d)\subset C^k_b(\R^d)$ is the space of functions in $C^k_b(\R^d)$ such that all $k$ derivatives are Lipschitz continuous. For $\phi \in L^1(\bm{X},\mu),\mu\in\mc{P}(\R^d)$ we define $\langle \mu,\phi\rangle \coloneqq \int_{\bm{X}}\phi(x)\mu(dx)$. For $a,b\in \R$, we will denote $a\wedge b = \min\br{a,b}$. $C$ will be used for a constant which may change from line to line throughout, and when there are parameters $a_1,...,a_n$ which $C$ depends on in an important manner, will will denote this dependence by $C(a_1,...,a_n)$. For all function spaces, the codomain is assumed to be $\R$ unless otherwise denoted.

In the proof of the Theorem \ref{theo:mckeanvlasovaveraging}, we will be making use of regularity of solutions to a Cauchy Problem on Wasserstein space (see Equation \eqref{eq:W2CauchyProblem}), for which we will need to establish existence, uniqueness, and regularity of some number of derivatives. In doing so, will be controlling many mixed derivatives of functions in the Lions sense and in the standard sense, it will be useful for us to borrow the multi-index notation proposed in \cite{CM} and employed in \cite{CST}. We will also need to ensure that the derivatives of $\Phi$ which appear in the definition the limiting coefficients in Equation \eqref{eq:limitingcoefficients} are well-defined and integrable against $\pi$. We thus extend the multi-index notation from the aforementioned papers to track specific collections of mixed partial derivatives, and to capture needed assumptions of local H\"older continuity and polynomial growth in $y$.

\begin{defi}\label{def:multiindexnotation}
Let $n,l,k$ be non-negative integers and $\bm{\beta}=(\beta_1,...,\beta_n)$ be an $n-$dimensional vector of non-negative integers. We call any ordered tuple of the form $(n,l,\mc{\beta})$ a \textbf{multi-index}. For a function $G:\R^j\times\mc{P}_2(\R^d)\tto \R^k$, we will denote for a multi-index $(n,l,\mc{\beta})$
\begin{align*}
D^{(n,l,\bm{\beta})}G(x,\mu)[z_1,...,z_n] = \partial_{z_1}^{\beta_1}...\partial_{z_n}^{\beta_n}\partial_x^l \partial_\mu^n G(x,\mu)[z_1,...,z_n]
\end{align*}
if this derivative is well defined. As noted in the Remark \ref{remark:thirdLionsDerivative}, for such a derivative to be well defined we require for it to be jointly continuous in $x,\mu,z_1,...,z_n$ where the topology used in the measure component is that of $\mc{P}_2(\R^d)$.
\end{defi}
\begin{defi}\label{def:completemultiindex}
For $\bm{\zeta}$ a collection of multi-indices of the form $(n,l,\bm{\beta})\in\bb{N}\times\bb{N}\times\bb{N}^n$, we will call $\bm{\zeta}$ a \textbf{complete} collection of multi-indices if for any $(n,l,\bm{\beta})\in \bm{\zeta}$, $\br{(k,j,\bm{\alpha}(k)):\bm{\alpha}(k)=(\alpha_1,...,\alpha_k),\alpha_p \leq \beta(k)_p, \forall \bm{\beta}(k)=(\beta(k)_1,...,\beta(k)_k)\in \binom{\bm{\beta}}{k},k\leq n,j\leq l}\subset \bm{\zeta}$. Here for a vector of positive integers $\beta=(\beta_1,...,\beta_n)$ and $k\in\bb{N},k\leq n$, we are using the notation $\binom{\bm{\beta}}{k}$ to represent the set of size $\binom{n}{k}$ containing all the $k$-dimensional vectors of positive integers which can be obtained from removing $n-k$ entries from $\bm{\beta}$.
\end{defi}
\begin{remark}\label{remark:oncompletecollectiondefinition}
Definition \ref{def:completemultiindex} is essentially enforcing that if collection of multi-indices contains a multi-index representing some mixed derivative in $(x,\mu,z)$ as per Definition \ref{def:multiindexnotation}, then it also contains all lower-order mixed derivatives of the same type. For instance, let $\bm{\zeta}$ be the collection of multi-indices containing $(2,0,(1,1))$ (corresponding to $\partial_{z_1}\partial_{z_2}\partial^2_\mu G(x,\mu)[z_1,z_2]$).  Then, in order to be complete, $\bm{\zeta}$ must also contain $(2,0,(1,0)),(2,0,(0,1)),(2,0,0),(1,0,1),(1,0,0),$ and $(0,0,0)$ (corresponding to $\partial_{z_1}\partial^2_\mu G(x,\mu)[z_1,z_2]$, $\partial_{z_2}\partial^2_\mu G(x,\mu)[z_1,z_2]$, $\partial^2_\mu G(x,\mu)[z_1,z_2]$,$\partial_{z}\partial_\mu G(x,\mu)[z]$, $\partial_\mu G(x,\mu)[z],$ and $G(x,\mu)$ respectively). This is a technical requirement used in order to state the results in Appendix \ref{sec:regularityofthecellproblem} in a way that allows the inductive arguments used therein to go through.
\end{remark}

Using this multi-index notation, it will be useful for us to define some spaces regarding regularity of functions in regard to these mixed derivatives. We thus make the following modifications to Definition 2.13 in \cite{CST}:
\begin{defi}\label{def:lionsderivativeclasses}
For $\bm{\zeta}$ a collection of multi-indices of the form $(n,l,\bm{\beta})\in\bb{N}\times\bb{N}\times\bb{N}^n$ and $k,j,d\in\bb{N}$ (or $k\in\bb{N}\times\bb{N}$ to denote matrix-valued functions), we define $\mc{M}_b^{\bm{\zeta}}(\R^j\times\mc{P}_2(\R^d);\R^{k})$ to be the class of functions $G:\R^j\times\mc{P}_2(\R^d)\tto \R^{k}$ such that $D^{(n,l,\bm{\beta})}G(x,\mu)[z_1,...,z_n]$ exists and satisfies
\begin{align}
\label{eq:LionsClassBoundedness} \norm{G}_{\mc{M}_b^{\bm{\zeta}}(\R^j\times\mc{P}_2(\R^d);\R^{k})}\coloneqq \sup_{(n,l,\bm{\beta})\in \bm{\zeta}}\sup_{x\in\R^j,z_1,...,z_n\in\R^d,\mu\in\mc{P}_2(\R^d)}|D^{(n,l,\bm{\beta})}G(x,\mu)[z_1,...,z_n]|&\leq C .
\end{align}
We denote the class of functions $G\in \mc{M}_b^{\bm{\zeta}}(\R^j\times\mc{P}_2(\R^d);\R^{k})$ such that:
\begin{align}
\label{eq:LionsClassLipschitz}|D^{(n,l,\bm{\beta})}G(x,\mu)[z_1,...,z_n]- D^{(n,l,\bm{\beta})}G(x',\mu')[z_1',...,z_n']|&\leq C_L\biggl(|x-x'|+\sum_{i=1}^N|z_i-z'_i|+\bb{W}_2(\mu,\mu') \biggr)
\end{align}
for all $(n,l,\bm{\beta})\in \bm{\zeta}$ and $x,x'\in\R^j,z_1,...,z_n,z_1',...,z_n'\in\R^d,\mu,\mu'\in\mc{P}_2(\R^d)$ by $\mc{M}_{b,L}^{\bm{\zeta}}(\R^j\times\mc{P}_2(\R^d);\R^{k})$. We define $\mc{M}_b^{\bm{\zeta}}(\mc{P}_2(\R^d);\R^{k})$ and $\mc{M}_{b,L}^{\bm{\zeta}}(\mc{P}_2(\R^d);\R^{k})$ analogously, where instead here $\bm{\zeta}$ a collection of multi-indices of the form $(n,\bm{\beta})\in\bb{N}\times\bb{N}^n$, and we take the $l=0$ in the above multi-index notation for the derivatives.

We will also make use of the class of functions $\mc{M}_{p}^{\bm{\zeta}}(\R^j\times\R^j\times \mc{P}_2(\R^d);\R^{k})$ which contains $G:\R^j\times\R^j\times\mc{P}_2(\R^d)\tto \R^{k}$ such that $G(\cdot,y,\cdot)\in \mc{M}_{b}^{\bm{\zeta}}(\R^j\times \mc{P}_2(\R^d);\R^{k})$ for all $y\in\R^j$, for each multi-index $(n,l,\bm{\beta})\in\bm{\zeta}$, there exists $m\in \bb{N}$ and $C>0$ such that:
\begin{align}\label{eq:newqnotation}
\sup_{x\in\R^j,z_1,...,z_n\in\R^d,\mu\in\mc{P}_2(\R^d)}|D^{(n,l,\bm{\beta})}G(x,y,\mu)[z_1,...,z_n]|&\leq C(1+|y|^{m}),
\end{align}
and there exists $\theta\in (0,1]$ such that for all $(n,l,\bm{\beta})\in \bm{\zeta}$ and $y,y',x,x'\in\R^j,z_1,...,z_n,z_1',...,z_n'\in\R^d,\mu,\mu'\in\mc{P}_2(\R^d)$:
\begin{align}\label{eq:LionsClassYLipschitz}
&|D^{(n,l,\bm{\beta})}G(x,y,\mu)[z_1,...,z_n]- D^{(n,l,\bm{\beta})}G(x',y',\mu')[z_1',...,z_n']|\\
&\leq C\biggl(1\wedge |y-y'|^\theta+|x-x'|+\sum_{i=1}^N|z_i-z'_i|+\bb{W}_2(\mu,\mu') \biggr)\biggl(1+|y|^m+|y'|^m\biggr).\nonumber
{}\end{align}

We also define $\mc{M}_b^{\bm{\zeta}}([0,T]\times\R^j\times\mc{P}_2(\R^d);\R^{k})$ to be the class of functions $G:[0,T]\times\R^j\times \mc{P}_2(\R^d)\tto \R^{k}$ such that $G(\cdot,x,\mu)$ is continuously differentiable on $(0,T)$ for all $x\in\R^j,\mu\in\mc{P}_2(\R^d)$ with time derivative denoted by $\dot{G}(t,x,\mu)$, $G(t,\cdot,\cdot) \in \mc{M}_b^{\bm{\zeta}}(\R^j\times\mc{P}_2(\R^d);\R^{k})$ for all $t\in [0,T]$, with \eqref{eq:LionsClassBoundedness} holding uniformly in $t$, and $G,\dot{G},$ and all derivatives involved in the definition of $\mc{M}_b^{\bm{\zeta}}(\R^j\times\mc{P}_2(\R^d);\R^{k})$ are jointly continuous in time, measure, and space. We define for $G\in \mc{M}_b^{\bm{\zeta}}([0,T]\times\R^j\times\mc{P}_2(\R^d);\R^{k})$
\begin{align*}
\norm{G}_{\mc{M}_b^{\bm{\zeta}}([0,T]\times\R^j\times\mc{P}_2(\R^d);\R^{k})}\coloneqq \sup_{t\in[0,T]}\norm{G(t,\cdot)}_{\mc{M}_b^{\bm{\zeta}}(\R^j\times\mc{P}_2(\R^d);\R^{k})}+\sup_{t\in[0,T],x\in\R^j,\mu\in\mc{P}_2(\R^d)}|\dot{G}(t,x,\mu)|.
\end{align*}
We denote the class of functions $G\in \mc{M}_b^{\bm{\zeta}}([0,T]\times\R^j\times\mc{P}_2(\R^d);\R^{k})$ such that \eqref{eq:LionsClassLipschitz} holds uniformly in $t$ by $\mc{M}_{b,L}^{\bm{\zeta}}([0,T]\times\R^j\times\mc{P}_2(\R^d);\R^{k})$. Again, we define $\mc{M}_b^{\bm{\zeta}}([0,T]\times\mc{P}_2(\R^d);\R^{k})$ and $\mc{M}_{b,L}^{\bm{\zeta}}([0,T]\times\mc{P}_2(\R^d);\R^{k})$ analogously.\end{defi}

It will be useful do define the following complete collections of multi-indices in the sense of Definitions \ref{def:multiindexnotation} and \ref{def:completemultiindex}:
\begin{align}\label{eq:collectionsofmultiindices}
\hat{\bm{\zeta}} &\coloneqq \br{(0,j_1,0),(1,j_2,j_3),(2,j_4,(j_5,j_6)),(3,0,(j_7,0,0)):j_1\in\br{0,1,...4},j_2+j_3\leq 4,j_4+j_5+j_6\leq 2,j_7=0,1 }\\
\hat{\bm{\zeta}}_1 &\coloneqq \br{(0,j_1,0),(1,j_2,j_3),(2,j_4,(j_5,j_6)),(3,j_7,(j_8,0,0))\nonumber\\
&\hspace{4.6cm}:j_1\in\br{0,1,...5},j_3\leq 4,j_2+j_3\leq 5,j_5+j_6\leq 2,j_4+j_5+j_6\leq 3,j_7+j_8\leq 1 }\nonumber.
\end{align}

Next, we introduce the main assumptions to be used throughout this paper. For a detailed explanation of the use of these assumptions, and in particular of \ref{assumption:regularityofcoefficientsnew}, which is stated in terms of the collections of multi-indices from Equation \eqref{eq:collectionsofmultiindices}, see Remark \ref{remark:ontheassumptions2}.
\begin{enumerate}[label={A\arabic*)}]
\item \label{assumption:uniformellipticity} There exist $\lambda_-,\lambda_+>0$ such that $0<\lambda_-\leq \frac{z^\top[\tau_1\tau_1^\top+\tau_2\tau_2^\top](x,y,\mu)z}{|z|^2}\leq \lambda_+<\infty$, $\forall x,y,z\in\R^d,z\neq 0,\mu\in\mc{P}_2(\R^d)$ and $\tau_1,\tau_2$ are bounded, have two uniformly bounded derivatives in $y$, and $\tau_1,\tau_2$ and both these derivatives are H\"older continuous in $y$ uniformly in $(x,\mu)$.
\item \label{assumption:retractiontomean}There exists constants $C,\beta>0$ independent of $x,y,\mu$ such that:
\begin{align}\label{eq:fdecayimplication}
f(x,y,\mu)\cdot y\leq -\beta |y|^2 +C,\forall x,y\in\R^d,\mu\in\mc{P}_2(\R^d),
{}\end{align}
$f$ grows at most linearly in $|y|$, $f$ has two uniformly bounded derivatives in $y$, and $f$ and both these derivatives are H\"older continuous in $y$ uniformly in $(x,\mu)$.
\item \label{assumption:centeringcondition} As is standard in averaging results, we assume also the centering condition. Namely, we assume that for $\pi$ as in Equation \eqref{eq:invariantmeasureold}, the condition (\ref{eq:centeringconditionold}) holds.
\end{enumerate}

In terms of regularity of coefficients, we assume the following assumptions on the coefficients of \eqref{eq:slow-fastMcKeanVlasov}:
\begin{enumerate}[label={A\arabic*)}]\addtocounter{enumi}{3}
\item \label{assumption:strongexistence} For $F=b,c,\sigma,f,g,\tau_1,$ or $\tau_2$, $F$ is globally Lipschitz continuous in $(x,y,\mu)$. That is, there exists $C>0$ such that for all $x,x',y,y'\in\R^d$ and $\mu,\mu'\in\mc{P}_2(\R^d)$:
\begin{align}\label{eq:globallyLipschitz}
|F(x,y,\mu)-F(x',y',\mu')|\leq C(|x-x'|+|y-y'|+\bb{W}_2(\mu,\mu')).
\end{align}
Moreover, $g$ is uniformly bounded.
\item \label{assumption:regularityofcoefficientsnew} $g,c\in \mc{M}^{\hat{\bm{\zeta}}}_p(\R^d\times\R^d\times \mc{P}_2(\R^d);\R^d),\sigma,\tau_1\in \mc{M}^{\hat{\bm{\zeta}}}_p(\R^d\times\R^d\times \mc{P}_2(\R^d);\R^{d\times m}),f,b\in \mc{M}^{\hat{\bm{\zeta}}_1}_p(\R^d\times\R^d\times \mc{P}_2(\R^d);\R^d),$ and $a\in \mc{M}^{\hat{\bm{\zeta}}_1}_p(\R^d\times\R^d\times \mc{P}_2(\R^d);\R^{d\times d})$.
\item \label{assumption:uniformellipticityDbar} For $\bar{D}$ as in Equation \eqref{eq:averagedlimitingcoefficients}, there exists $\bar{\lambda}_->0$ such that $0<\bar{\lambda}_-\leq \frac{z^\top\bar{D}(x,\mu)z}{|z|^2}$, $\forall x,z\in\R^d,z\neq 0,\mu\in\mc{P}_2(\R^d)$.
\end{enumerate}
\begin{remark}\label{remark:ontheassumptions1}
Note that under Assumption \ref{assumption:strongexistence}, for each choice of square integrable initial conditions $(\eta,\zeta)$ and each $\epsilon>0$, there exists a unique solution $\br{(X^{\epsilon,\nu}_t,Y^{\epsilon,\nu}_t),t\geq 0}$ to the system \eqref{eq:slow-fastMcKeanVlasov}  such that $\mc{L}(X^{\epsilon,\nu}_t)\in \mc{P}_2(\R^d)$ for each $t\geq 0$. See, e.g. \cite{Wang} Theorem 2.1 and Section 6.1 in \cite{RocknerMcKeanVlasov}. There are weaker assumptions under which existence and uniqueness for McKean-Vlasov SDEs have been established in the recent literature which may replace the global Lipschitz assumption \eqref{eq:globallyLipschitz} (see, e.g. \cites{HSS,Mehri,HWSingular,HWY,Ren}), but we chose the simplest of these in order to clearly illustrate our results.

Moreover, under Assumptions \ref{assumption:uniformellipticity},\ref{assumption:retractiontomean}, \ref{assumption:centeringcondition}, \ref{assumption:regularityofcoefficientsnew}, and \ref{assumption:uniformellipticityDbar}, Proposition \ref{proposition:allneededregularity} yields that $\bar{\gamma},\bar{D}^{1/2}$ are globally Lipschitz in $(x,\mu)$, so that the same result implies existence and uniqueness of solutions to the limiting averaged dynamics \eqref{eq:averagedMcKeanVlasov}.
\end{remark}
\begin{remark}\label{remark:ontheassumptions2}
The high amount of regularity imposed on the coefficients and their derivatives in Assumption \ref{assumption:regularityofcoefficientsnew} is needed to establish the analogous regularity of the averaged coefficients appearing in the limiting Equation \eqref{eq:averagedMcKeanVlasov}. This regularity of the averaged coefficients is needed for Lemma \ref{lem:cauchyproblemregularity}, which provides bounds on the derivatives of the solution to the Cauchy Problem on Wasserstein space \eqref{eq:W2CauchyProblem} which appear in the proof of Theorem \ref{theo:mckeanvlasovaveraging}-that is, those which are contained in $\dot{\bm{\zeta}}$ as defined in Theorem \ref{theo:mckeanvlasovaveraging}. Unpacking the multi-index notation, we need that $U$ has bounded Lions derivatives up to order 3, that $\partial_\mu U(\mu)[z]$ has 3 bounded derivatives in $z$, that $\partial^2_\mu U(\mu)[z_1,z_2]$ has bounded second order derivatives in $z_1,z_2$, and that $\partial^3_\mu U(\mu)[z_1,z_2,z_3]$ has bounded first order derivatives in $z_1,z_2,z_3$ (by symmetry we can just assume this in one of the auxiliary variables).

In Lemma \ref{lem:cauchyproblemregularity}, we use the results of \cite{BLPR} as extended in \cite{CST}, where the proof method is via a ``variational approach'' that requires Lipschitz continuity in $x,\mu$ of many derivatives of the coefficients of the PDE \eqref{eq:W2CauchyProblem} in order to establish existence and uniqueness of variational equations related to the derivatives of the associated process \eqref{eq:averagedMcKeanVlasov} in its initial conditions. The way the result is stated in \cite{CST}, the requirement that $U\in \mc{M}^{\dot{\bm{\zeta}}}_b(\mc{P}_2(\R^d))$ would be summarized as $U$ having all derivatives in $(\mu,z)$ of order 4 bounded, and the sufficient requirement on the coefficients in $\bar{\gamma},\bar{D}$ appearing in Equation \eqref{eq:W2CauchyProblem} would be that $\bar{\gamma},\bar{D}^{1/2}$ have all derivatives in $(x,\mu,z)$ of order 4 bounded. However, we don't require $U$ be 4 times differentiable in $\mu$, so we can require slightly less regularity of $\bar{\gamma},\bar{D}^{1/2}$, resulting in the collection of multi-indices $\hat{\bm{\zeta}}$. The assumption that some $F\in \mc{M}^{\hat{\bm{\zeta}}}_p(\R^d\times\R^d\times \mc{P}_2(\R^d);\R^{k})$, as in \ref{assumption:regularityofcoefficientsnew}, is requiring regularity of $F$, $\partial_x^{j_1} F$ for $j_1=1,2,3,4$, $\partial^{j_3}_z\partial^{j_2}_x\partial_\mu F$ for $j_2+j_3\leq 4,\partial^{j_5}_{z_1}\partial^{j_6}_{z_2}\partial^{j_4}_x\partial^2_\mu F$ for $j_4+j_5+j_6\leq 2$, $\partial^3_\mu F$, and $\partial_{z_1}\partial^3_\mu F$. Note the lack of requirement of 4-times differentiability in $\mu$. Also note that this assumption does not impose regularity in the $y$ argument, other than the joint H\"older continuity imposed by Equation \eqref{eq:LionsClassYLipschitz} in Definition \ref{def:lionsderivativeclasses}. The collection of multi-indices $\hat{\bm{\zeta}}_1$ is the result of adding one more derivative in $x$ to the derivatives represented by $\hat{\bm{\zeta}}$. This is needed to ensure the terms $\partial_x\Phi$ and $\partial_x\partial_y\Phi$ appearing in $\bar{\gamma}$ have the $\hat{\bm{\zeta}}$-regularity required for Lemma \ref{lem:cauchyproblemregularity}. See Proposition \ref{proposition:allneededregularity} for how the regularity of the prelimit coefficients assumed in \ref{assumption:regularityofcoefficientsnew} implies the required regularity of the limiting coefficients.

It is likely that the regularity of the limiting, and hence prelimit, coefficients and of the initial condition $G$ that is required in order to establish boundedness of these derivatives of solutions to \eqref{eq:W2CauchyProblem} can be weakened to, e.g. H\"older continuity of some lesser number of derivatives, as in the case for the standard Cauchy problem \cites{EidelmanBook,Friedman,LSU}. Partial results in this direction via Malliavin calculus techniques \cite{CM} and an infinite-dimensional parametrix method \cites{DF1,DF2} already exist. Extending these results to the higher number of derivatives needed in Theorem \ref{theo:mckeanvlasovaveraging} is an interesting avenue of future research, and is beyond the scope of this paper.

It is also worth noting that, with the exception of the boundedness of $g,\tau_1,$ and $\tau_2$ as well as the dissipativity assumption \ref{assumption:retractiontomean}, which are used in Lemma \ref{lemma:barYuniformbound}, all of the other imposed regularity in the Assumptions \ref{assumption:uniformellipticity}-\ref{assumption:uniformellipticityDbar} are only needed to provide (weak) existence and uniqueness of the prelimit and limiting system (\eqref{eq:slow-fastMcKeanVlasov} and \eqref{eq:averagedMcKeanVlasov} respectively), the aforementioned sufficient regularity for Lemma \ref{lem:cauchyproblemregularity}, and the needed regularity and existence/uniqueness of the auxiliary Poisson Equations used in Section \ref{section:mckeanvlasovergodictheorems}. These are simply sufficient conditions, and are be no means necessary, and these properties can also be proved on a case-by-case basis. Observe that the only assumption imposed on the limiting coefficients here is \ref{assumption:uniformellipticityDbar}, which by the representation provided in Equation \eqref{eq:alternativediffusion} will hold in most situations.
\end{remark}

\section{Main Results and examples}\label{subsection:statementofmainresults}
We are now ready to state our main results:
\begin{theo}\label{theo:mckeanvlasovaveraging}
Assume \ref{assumption:uniformellipticity}-\ref{assumption:uniformellipticityDbar}. Define
\begin{align*}
\dot{\bm{\zeta}}\coloneqq \br{(0,0),(1,j_1),(2,(j_2,j_3)),(3,(0,0,j_4)):j_1\in\br{0,1,2,3},j_2+j_3\leq 2,j_4=0,1}.
\end{align*}
Then for any $G\in \mc{M}_{b,L}^{\dot{\bm{\zeta}}}(\mc{P}_2(\R^d);\R)$, $\nu\in\mc{P}_2(\R^d)$, and $T>0$, there is $C(T)$ independent of $G$ and $\nu$ such that for $\epsilon\in(0,1]$:
\begin{align*}
\sup_{s\in [0,T]}\biggl|G(\mc{L}(X^{\epsilon,\nu}_s))-G(\mc{L}(X^\nu_s))\biggr|&\leq \epsilon C(T)\norm{G}_{\mc{M}_b^{\dot{\bm{\zeta}}}(\mc{P}_2(\R^d);\R)}.
\end{align*}
\end{theo}
\begin{proof}
The proof is found in Section \ref{section:proofofmainresults}.
\end{proof}
\begin{corollary}\label{corollary:lineartestfunction}
In the setup of Theorem \ref{theo:mckeanvlasovaveraging}, for any $\phi\in C^4_{b,L}(\R^d)$ and $\nu\in\mc{P}_2(\R^d)$, we have there is $C(T)$ independent of $\phi$ and $\nu$ such that:
\begin{align*}
\sup_{s\in [0,T]}\biggl|\E[\phi(X^{\epsilon,\nu}_s)]-\E[\phi(X^\nu_s)]\biggr|&\leq \epsilon C(T)|\phi|_{C_b^4(\R^d)}.
\end{align*}
\end{corollary}
\begin{proof}
Considering $G$ of the form $G_\phi(\mu) = \langle \mu,\phi\rangle$ for $\phi\in C^4_{b,L}(\R^d)$ , $\partial_\mu G_\phi(\mu)[z] = \nabla\phi(z)$ (see, e.g. \cite{CD} Section 5.2.2 Example 1). Thus, $G_\phi\in\mc{M}_b^{\dot{\bm{\zeta}}}(\mc{P}_2(\R^d);\R)$, and $\norm{G_\phi}_{\mc{M}_b^{\dot{\bm{\zeta}}}(\mc{P}_2(\R^d);\R)}\leq |\phi|_{C_b^4(\R^d)}.$ The result then follows immediately via an application of Theorem \ref{theo:mckeanvlasovaveraging}.
\end{proof}

\begin{remark}\label{remark:nomeasuredependence}
Theorem \ref{theo:mckeanvlasovaveraging} holds in the situation where the coefficients in Equation \eqref{eq:slow-fastMcKeanVlasov} are independent of $\mu$. In this setting, Assumption \ref{assumption:regularityofcoefficientsnew} is imposing that 4 derivatives of $g,\sigma,\tau_1,c$ in $x$ and 5 derivatives of $f,a,b$ in $x$ grow at most polynomially in $y$ and are jointly Lipschitz/Locally H\"older continuous in $(x,y)$ in the sense of Equation \eqref{eq:LionsClassLipschitz}. This thus extends the results of \cite{RocknerFullyCoupled} Theorem 2.3 Regime 4 with $\vartheta =1$ from test functions of the form $G_\phi$ as in Corollary \ref{corollary:lineartestfunction} to the more general class of (possibly non-linear) test functions in $\mc{M}_{b,L}^{\dot{\bm{\zeta}}}(\mc{P}_2(\R^d);\R)$.
\end{remark}
\subsection{Examples: A Class of Aggregation-Diffusion Equations}\label{SS:Example}
A common form for interacting particle systems which are widely used in many settings such as in biology, ecology, social sciences, economics, molecular dynamics, and in study of spatially homogeneous granular media (see \cites{MT,Garnier1,BCCP,KCBFL} and the references therein) is:
\begin{align}\label{eq:nomultilangevin}
dX^{i,N}_t &= -\nabla V(X^{i,N}_t)dt - \frac{1}{N}\sum_{j=1}^N \nabla W(X^{i,N}_t-X^{j,N}_t)dt +\sigma dW^i_t,\quad X^{i,N}_0=\eta^i,
\end{align}
where $V:\R^d\tto\R$ is a sufficiently smooth confining potential, $W:\R^d\tto \R$ is a sufficiently smooth interaction potential, $\sigma>0$, and $W^i$ are iid $d-$dimensional Brownian motions. The class of systems \eqref{eq:nomultilangevin} contains the system in the seminal paper \cite{Dawson}, where many mathematical aspects of a model for cooperative behavior in a bi-stable confining potential with attraction to the mean are explored.

Under sufficient regularity on $V$ and $W$ and exchangeability assumptions on the initial conditions, as $N\toinf$, the particles $X^{i,N}$ become IID and equal in distribution to $\tilde{X}$ obeying the dynamics of the Aggregation-Diffusion Equation
\begin{align}\label{eq:aggregationdiffusion}
d\tilde{X}_t & = -\nabla V(\tilde{X}_t)dt- \bar{\E}[\nabla W(x-\bar{X}_t)]|_{x=\tilde{X}_t}dt + \sigma dW_t,\quad \tilde{X}_0=\eta,
\end{align}
where here we are denoting by $\bar{X}_t$ an independent copy of $\tilde{X}_t$ on another probability space $(\bar{\W},\bar{\F},\bar{\Prob})$, and by $\bar{\E}$ the expectation on that space, and the distribution of $\eta$ is determined by that of the $\eta^i$'s.

Inspired by the pervasiveness of such systems in the literature, we consider the fully-coupled slow-fast system of Aggregation-Diffusions Equations:

\begin{align}\label{eq:slowfastLangevin}
dX^{\epsilon,\nu}_t &= -[\nabla V_1(X^{\epsilon,\nu}_t)+\frac{1}{\epsilon}\nabla V_2(Y^{\epsilon,\nu}_t)]dt - \bar{\E}[\nabla W_1(x-\bar{X}^{\epsilon,\nu}_t)]|_{x=X^{\epsilon,\nu}_t}dt +\sigma dW_t\\
dY^{\epsilon,\nu}_t & = -\frac{1}{\epsilon}[\nabla V_3(X^{\epsilon,\nu}_t)+\frac{1}{\epsilon}\nabla V_4(Y^{\epsilon,\nu}_t)]dt - \frac{1}{\epsilon}\bar{\E}[\nabla W_2(x-\bar{X}^{\epsilon,\nu}_t)]|_{x=X^{\epsilon,\nu}_t}dt + \frac{1}{\epsilon}\tau_1 dW_t+\frac{1}{\epsilon}\tau_2 dB_t\nonumber\\
(X^{\epsilon,\nu}_0,Y^{\epsilon,\nu}_0)& = (\eta,\zeta),\nonumber
\end{align}
where $V_k:\R^d\tto \R,k=1,...,4$, $W_1,W_2:\R^d\tto \R$, $\bar{X}^{\epsilon,\nu},\bar\E$ as in Equation \eqref{eq:aggregationdiffusion}, and the rest of the setup is as in Equation \eqref{eq:slow-fastMcKeanVlasov} with $d=m$.

This is falls into the class of systems \eqref{eq:slow-fastMcKeanVlasov} with
\begin{align*}
f(x,y,\mu) &= -\nabla V_4(y)\\
b(x,y,\mu) & = -\nabla V_2(y)\\
c(x,y,\mu) &= -\nabla V_1(x) - \langle \mu, \nabla W_1(x-\cdot)\rangle \\
g(x,y,\mu) &= -\nabla V_3(x) - \langle \mu, \nabla W_2(x-\cdot)\rangle  \\
\sigma(x,y,\mu)&\equiv  \sigma I \\
\tau_1(x,y,\mu)&\equiv  \tau_1 I \\
\tau_2(x,y,\mu)&\equiv \tau_2I .
\end{align*}
\begin{example}\label{example:slowfastaggdiff}
Consider the system \eqref{eq:slowfastLangevin}.

Suppose $\alpha\coloneqq \frac{1}{2}[\tau_1^2+\tau_2^2]>0$, there exists $C,\beta>0$ such that $\nabla V_4(y)\cdot y +C \geq \beta |y|^2$, $\nabla V_4$ grows at most linearly in $y$ and has two locally H\"older continuous, bounded derivatives, $\nabla V_2$ is Lipschitz continuous, and $V_1,V_3\in C_{b,L}^5(\R^d), W_1,W_2\in C_{b,L}^6(\R^d)$.

Moreover, assume
\begin{align*}
\int_{\R^d}\nabla V_2(y)\exp(-V_4(y)/\alpha)dy=0.
\end{align*}
Then for any $G\in \mc{M}_{b,L}^{\dot{\bm{\zeta}}}(\mc{P}_2(\R^d);\R)$, $\nu\in\mc{P}_2(\R^d)$, and $T>0$:
\begin{align*}
\sup_{s\in [0,T]}\biggl|G(\mc{L}(X^{\epsilon,\nu}_s))-G(\mc{L}(X^\nu_s))\biggr|&\leq \epsilon C(T)\norm{G}_{\mc{M}_b^{\dot{\bm{\zeta}}}(\mc{P}_2(\R^d);\R)},
\end{align*}
where here $X^\nu_t$ satisfies:
\begin{align}\label{eq:averagedaggregationdiffusioneqn}
dX^\nu_t &= -[\alpha_1\nabla V_3(X^\nu_t)+ \nabla V_1(X^\nu_t)]dt -\bar{\E}[\alpha_1\nabla W_2(x-\bar{X}^\nu_t)+\nabla W_1(x-\bar{X}^\nu_t)]|_{x=X^\nu_t}dt\\
&+[\sigma^2I+2\alpha\alpha_2  +\sigma\tau_1[\alpha_1+\alpha_1^\top]]^{1/2}dW^2_t\nonumber\\
X^\nu_0& = \eta^2\sim\nu \nonumber\\
\alpha_1&=Z^{-1}\int_{\R^d} \nabla\Phi(y)\exp\biggl(\frac{-V_4(y)}{\alpha}\biggr)dy\nonumber\\
\alpha_2 &= Z^{-1}\int_{\R^d} \nabla\Phi(y)(\nabla\Phi)^\top(y)\exp\biggl(\frac{-V_4(y)}{\alpha}\biggr)dy\nonumber\\
Z& = \int_{\R^d}\exp\biggl(\frac{-V_4(y)}{\alpha}\biggr)dy.\nonumber
\end{align}
\end{example}
\begin{proof}

In this setting the invariant measure $\pi$ from Equation \eqref{eq:invariantmeasureold} admits a density of the form
\begin{align*}
\pi(y)& = Z^{-1}\exp\biggl(\frac{-V_4(y)}{\alpha}\biggr).
\end{align*}
Thus Assumptions \ref{assumption:uniformellipticity}-\ref{assumption:centeringcondition} hold by supposition. In addition, noting that as per, e.g. \cite{CD} Section 5.2.2 Example 1, $\partial_\mu c(x,y,\mu)[z]= -\nabla^2 W_1(x-z)$ and similarly for $g$, we can see that the current assumptions also imply Assumptions \ref{assumption:strongexistence} and \ref{assumption:regularityofcoefficientsnew}. Here for the Lipschitz continuity of $c$ and $g$ in $\bb{W}_2$, we use the boundedness of their Lions derivatives and Remark 5.27 in \cite{CD}. Note that Assumption \ref{assumption:uniformellipticityDbar} is only being used in Proposition \ref{proposition:allneededregularity} in order to obtain that $\bar{D}^{1/2}\in \mc{M}_{b,L}^{\hat{\bm{\zeta}}}(\R^d\times\mc{P}_2(\R^d);\R^{d\times d})$ knowing $\bar{D}\in\mc{M}_{b,L}^{\hat{\bm{\zeta}}}(\R^d\times\mc{P}_2(\R^d);\R^{d\times d}),$ but here this fact is trivial since both are constant, and hence we need not worry about uniform positive-definiteness of $\bar{D}$. Thus, this is a direct application of Theorem \ref{theo:mckeanvlasovaveraging}.
\end{proof}

An interesting subcase of the above example is when $V_2=V_4=: Q$ is 1-periodic in all directions, $V_1=V_3=: V,W_1=W_2=: W,\sigma=\tau_1>0,\tau_2=0$, so that $Y^{\epsilon,\nu}_t$ has the same dynamics as $X^{\epsilon,\nu}_t$ but order $1/\epsilon$ faster. In this situation, the dynamics of $Y^{\epsilon,\nu}$ are confined to the torus, so we need not worry about its integrability (in particular, we may drop Equation \eqref{eq:fdecayimplication} from Assumption \ref{assumption:retractiontomean}). This corresponds to the standard Aggregation-Diffusion Equation \eqref{eq:nomultilangevin}, but where we replace the confining potential $V$ with a rough potential $V^\epsilon(x) = V(x)+Q(x/\epsilon)$. Note that in this situation the centering condition \ref{assumption:centeringcondition} automatically holds.

In the case that we have a separable fluctuating part, that is $Q(y_1,...,y_d)=Q_1(y_1)+...+Q_d(y_d)$, the limiting equation \eqref{eq:averagedaggregationdiffusioneqn} becomes more tractable, as everything is explicitly computable. It reads:
\begin{align}\label{eq:averagedaggdifperioduccase}
dX^\nu_t &= -\Theta\nabla V(X^\nu_t)dt - \bar{\E}[\Theta \nabla W(x-\bar{X}^\nu_t)]|_{x=X^\nu_t}dt +\sigma\Theta^{1/2}dW^2_t\\
X^\nu_0 &=\eta^2\sim\nu\nonumber\\
\Theta &\coloneqq \text{diag}\biggl[Z_1^{-1}\hat{Z}_1^{-1},..., Z_d^{-1}\hat{Z}_d^{-1}\biggr]\nonumber\\
Z_k&\coloneqq \int_0^1 \exp(-2Q_k(y)/\sigma^2)dy,\hat{Z}_k\coloneqq \int_0^1 \exp(2Q_k(y)/\sigma^2)dy,k=1,...,d\nonumber.
\end{align}

Observe that $Z_k^{-1}\hat{Z}_k^{-1}\in (0,1]$ for $k=1,...,d$. Thus, the effect of averaging is not only that the effective diffusivity of the aggregation-diffusion equation is decreased and the magnitude of the effective confining potential is decreased in all directions (as is well understood- see e.g. \cites{Zwanzig,DP}), but also that the magnitude of the effective interaction potential is decreased in all directions. This is remarkable considering the fact that, considering the Aggregation-Diffusion Equation \eqref{eq:nomultilangevin} as the limit of the particle system \eqref{eq:aggregationdiffusion}, the addition of multiscale structure through modifying $V$ to $V^\epsilon$ is a priori only effecting the motion of each particle, not their means of interaction. Note that this is not a byproduct of the fact that we consider the limit as $\epsilon\downarrow 0$ after $N\toinf$, as the limits have been shown to commute \cite{BS}. See Figure \ref{fig:effectivepotentials} for an example of this rescaling of both the confining and interaction potentials in the setting of the seminal paper \cite{Dawson}, where the confining potential is Curie-Weiss and the interaction potential is quadratic. Note that in practice the interaction and diffusion potentials must be mollified so that they are bounded as $|x|\toinf$ to fit into the regime of Example \ref{example:slowfastaggdiff}.
\begin{figure}[t]
\centering
\includegraphics[width=12cm]{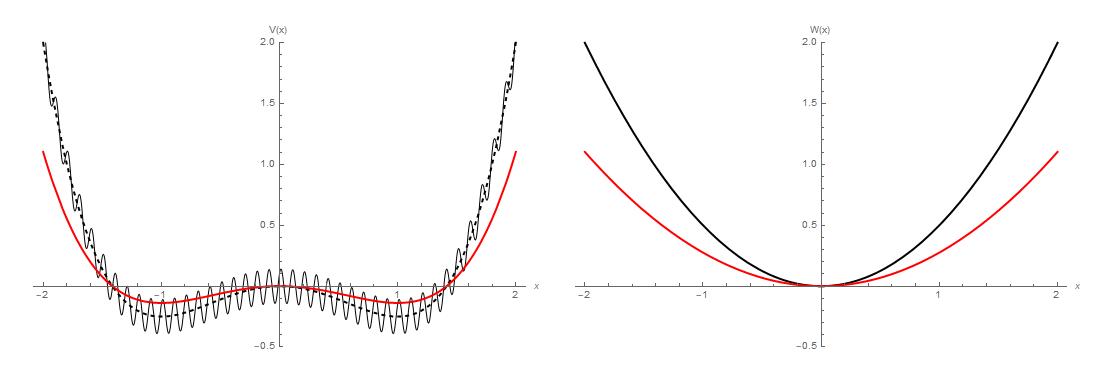}
\caption{Left: A rough confining potential $V^\epsilon(x) = \frac{x^4}{4}-\frac{x^2}{2}+.1[\cos(2\pi x/\epsilon)+\sin(2\pi x/\epsilon)]$ with $\epsilon=.1$ (black) overlaid on its non-rough counterpart (dashed). The corresponding effective confining potential after averaging is given in red.
Right: A prelimit interaction potential $W(x)=\frac{1}{2}x^2$ (black) and the corresponding effective interaction potential after averaging (red). The calculations are done with $\sigma=.5$.}
\label{fig:effectivepotentials}
\end{figure}

\section{Ergodic-Type Theorems for Fully-Coupled McKean-Vlasov SDEs}\label{section:mckeanvlasovergodictheorems}
In this subsection, we use the method of auxiliary Poisson equations to derive rates of averaging in the form of Ergodic-Type Theorems for the Slow-Fast McKean-Vlasov System \eqref{eq:slow-fastMcKeanVlasov}. This method is standard for averaging of diffusions. In particular, the analogous result to Proposition \ref{prop:fluctuationestimateparticles2} is necessary to see what the form of the limiting equation should be, and an analogous result is commonplace in the literature on averaging fully-coupled standard diffusions - see e.g. \cite{PV2} Theorem 4 and \cite{RocknerFullyCoupled} Lemma 4.4. A major difference here is, since we consider the Cauchy-Problem on Wasserstein Space \eqref{eq:W2CauchyProblem}, the test-function $\psi$'s domain is not only time and space, but also a measure component. When in some sense rates of convergence of the slow-fast system to the limiting averaged system must be established (such as when fluctuations are being considered), an analogous result to Proposition \ref{prop:llntypefluctuationestimate2} is often employed - see e.g. \cite{Spiliopoulos2014Fluctuations} Lemma 4.1,\cite{MS} Lemma B.5, and \cite{RocknerFullyCoupled} Lemma 4.2. For an analogous result in the context of McKean-Vlasov SDEs, see the proof of Theorem 2.5 in \cite{RocknerMcKeanVlasov}, although there they do not consider a Fully-Coupled system, so there is no need for the inclusion of the test function $\psi$.

It is worth noting that the terms being controlled in Propositions \ref{prop:purpleterm2} and \ref{prop:newpurpleterm} below are unique to slow-fast McKean-Vlasov SDEs (or their weakly interacting particle system counterparts), and thus do not appear in the one-particle setting. Thus the ``doubled Poisson equation'' construction (see Equations \eqref{eq:doublecorrectorproblem} and \eqref{eq:tildechi}) and the proof of Propositions \ref{prop:purpleterm2} and \ref{prop:newpurpleterm} are novel to this paper and its inspiring work \cite{BS_MDP2022}.
We begin with the following Lemma, which provides a necessary bound on the moments of the fast process needed for the Ergodic-Type Theorems:
\begin{lemma}\label{lemma:barYuniformbound}Assume \ref{assumption:uniformellipticity}, \ref{assumption:retractiontomean}, and \ref{assumption:strongexistence}. Then any $p\in\bb{N}$:\begin{align*}
\sup_{\epsilon\in(0,1]}\sup_{t\in [0,T]}\E\biggl[|Y^{\epsilon,\nu}_t|^{2p}\biggr]\leq C(p)+\E[|\zeta|^{2p}].
\end{align*}
\end{lemma}
\begin{proof}
We first note that, by It\^o's formula (suppressing the arguments of the coefficients for readability):
\begin{align*}
\E\biggl[ |Y^{\epsilon,\nu}_t|^{2p}\biggr] & = \E[|\zeta|^{2p}] + 2p\int_0^t \E\biggl[ \frac{1}{\epsilon^2}\biggl(f\cdot Y^{\epsilon,\nu}_s|Y^{\epsilon,\nu}_s|^2+(p-1)[\tau_1^2+\tau_2^2]:Y^{\epsilon,\nu}_s\otimes Y^{\epsilon,\nu}_s\biggr)|Y^{\epsilon,\nu}_s|^{2p-4}\biggr]ds \\
&+ \frac{p}{\epsilon^2}\int_0^t \E\biggl[\text{Tr}(\tau_1^2+\tau_2^2)|Y^{\epsilon,\nu}_s|^{2p-2} \biggr]ds+ \frac{2p}{\epsilon}\int_0^t \E\biggl[g\cdot Y^{\epsilon,\nu}_s|Y^{\epsilon,\nu}_s|^{2p-2}\biggr]ds \\
&+\frac{2p}{\epsilon}\E\biggl[\int_0^t |Y^{\epsilon,\nu}_s|^{2p-2}(Y^{\epsilon,\nu}_s)^\top \tau_1 dW_s+\int_0^t |Y^{\epsilon,\nu}_s|^{2p-2}(Y^{\epsilon,\nu}_s)^\top \tau_2 dB_s\biggr]\\
& =  \E[|\zeta|^{2p}] + 2p\int_0^t \E\biggl[ \frac{1}{\epsilon^2}\biggl(f\cdot Y^{\epsilon,\nu}_s|Y^{\epsilon,\nu}_s|^2+(p-1)[\tau_1^2+\tau_2^2]:Y^{\epsilon,\nu}_s\otimes Y^{\epsilon,\nu}_s\biggr)|Y^{\epsilon,\nu}_s|^{2p-4}\biggr]ds \\
&+ \frac{p}{\epsilon^2}\int_0^t \E\biggl[\text{Tr}(\tau_1^2+\tau_2^2)|Y^{\epsilon,\nu}_s|^{2p-2} \biggr]ds+ \frac{2p}{\epsilon}\int_0^t \E\biggl[g\cdot Y^{\epsilon,\nu}_s|Y^{\epsilon,\nu}_s|^{2p-2}\biggr]ds,
\end{align*}
using that by the boundedness of $\tau_1,\tau_2,g$ and linear growth of $f$ that for each $\epsilon>0$ and $m\in\bb{N}$, $\E\biggl[ \int_0^T |Y^{\epsilon,\nu}_s|^m ds\biggr]\leq C(\epsilon,m,T)$.
Continuing, we have by Equation \eqref{eq:fdecayimplication} and boundedness of $\tau_1,\tau_2,g$:
\begin{align*}
\frac{d}{dt}\E\biggl[ |Y^{\epsilon,\nu}_t|^{2p}\biggr] & \leq  -\frac{2p\beta}{\epsilon^2} \E\biggl[|Y^{\epsilon,\nu}_t|^{2p}\biggr] +  \frac{C(p)}{\epsilon^2}\E\biggl[|Y^{\epsilon,\nu}_t|^{2p-2}\biggr] + \frac{C(p)}{\epsilon} \E\biggl[|Y^{\epsilon,\nu}_t|^{2p-1}\biggr]\\
&\leq -\frac{2p\beta}{\epsilon^2}\E\biggl[|Y^{\epsilon,\nu}_t|^{2p}\biggr] +\frac{p\beta}{2\epsilon^2} \E\biggl[|Y^{\epsilon,\nu}_t|^{2p}\biggr] + \frac{C(p)}{\epsilon^{2}}+\frac{C(p)}{\epsilon}\E\biggl[|Y^{\epsilon,\nu}_t|^{2p-1}\biggr].
\end{align*}
Here we wrote for $\eta>0$:
\begin{align*}
\frac{C(p)}{\epsilon^2}|Y^{\epsilon,\nu}_s|^{2p-2} =\frac{1}{\epsilon^2}\frac{C(p)}{\eta}\eta|Y^{\epsilon,\nu}_s|^{2p-2} \leq \frac{1}{\epsilon^2}\frac{1}{p}\biggl(\frac{C(p)}{\eta}\biggr)^{p} + \frac{1}{\epsilon^2}\frac{2p-2}{2p}\biggl(\eta|Y^{\epsilon,\nu}_s|^{2p-2}\biggr)^{2p/(2p-2)}\end{align*}
using Young's inequality, and let  $\eta = \biggl(\frac{p^2\beta}{2(p-1)} \biggr)^{(2p-2)/(2p)}$.

Similarly, writing
\begin{align*}
\frac{C(p)}{\epsilon}|Y^{\epsilon,\nu}_s|^{2p-1} = \frac{C(p)}{\epsilon^{1-\alpha}\eta}\frac{\eta|Y^{\epsilon,\nu}_s|^{2p-1}}{\epsilon^\alpha}\leq \frac{1}{2p}\biggl(\frac{C(p)}{\epsilon^{1-\alpha}\eta}\biggr)^{2p} + \frac{2p-1}{2p}\biggl(\frac{\eta|Y^{\epsilon,\nu}_s|^{2p-1}}{\epsilon^\alpha}\biggr)^{2p/(2p-1)}
\end{align*}
and letting $\alpha = \frac{2p-1}{p}$ and $\eta = \biggl(\frac{p^2\beta}{2p-1}\biggr)^{(2p-1)/(2p)}$, we get
\begin{align*}
\frac{d}{dt}\E\biggl[ |Y^{\epsilon,\nu}_t|^{2p}\biggr] & \leq -\frac{2p\beta}{\epsilon^2}\E\biggl[|Y^{\epsilon,\nu}_t|^{2p}\biggr] +\frac{p\beta}{2\epsilon^2}\E\biggl[|Y^{\epsilon,\nu}_t|^{2p}\biggr] + \frac{C(p)}{\epsilon^{2}}+\frac{p\beta}{2\epsilon^2}\E\biggl[|Y^{\epsilon,\nu}_t|^{2p}\biggr]+\frac{C(p)}{\epsilon^{2-2p}}\\
& =  -\frac{p\beta}{\epsilon^2}\E\biggl[|Y^{\epsilon,\nu}_t|^{2p}\biggr] + \frac{C(p)}{\epsilon^{2}}+\frac{C(p)}{\epsilon^{2-2p}}.
\end{align*}

Now, recalling that if $g'(s)\leq -\gamma g(s)+C,\forall s\in[0,t],$ then $g(t)\leq C\int_0^t e^{-\gamma(t-s)}ds +e^{-\gamma t}g(0)$, we have
\begin{align*}
\E\biggl[ |Y^{\epsilon,\nu}_t|^{2p}\biggr] & \leq C(p)[\frac{1}{\epsilon^2}+\frac{1}{\epsilon^{2-2p}}] \exp(-\frac{p\beta}{\epsilon^2}t)\int_0^t \exp(\frac{p\beta}{\epsilon^2}s)ds + \E[|\zeta|^{2p}]\exp(-\frac{p\beta}{\epsilon^2}t)\\
& =  C(p)[\frac{1}{\epsilon^2}+\frac{1}{\epsilon^{2-2p}}] \frac{\epsilon^2}{p\beta}\exp(-\frac{p\beta}{\epsilon^2}t)[\exp(\frac{p\beta}{\epsilon^2}t)-1]+ \E[|\zeta|^{2p}]\exp(-\frac{p\beta}{\epsilon^2}t)\\
&\leq C(p)[1+\epsilon^{2p}] +\E[|\zeta|^{2p}]\\
&\leq C(p)+\E[|\zeta|^{2p}]
\end{align*}\
since $p>0$ and $\epsilon\in (0,1]$. Since the bound is uniform in $t$, we are done.
\end{proof}
We now provide the aforementioned Ergodic-Type Theorems and their proofs:
\begin{proposition}\label{prop:fluctuationestimateparticles2}
Consider $\psi\in \mc{M}_b^{\tilde{\bm{\zeta}}}([0,T]\times\R^d\times\mc{P}_2(\R^d);\R^d),$  where here
\begin{align*}
\tilde{\bm{\zeta}}&\coloneqq \br{(0,j_1,0),(1,0,j_2):j_1=0,1,2,j_2=0,1}.
\end{align*} \
Assume \ref{assumption:uniformellipticity} - \ref{assumption:regularityofcoefficientsnew}. Then for any $t\in[0,T]$ and $\epsilon\in (0,1]$:\begin{align*}
&\biggl|\E\biggl[\int_0^t \frac{1}{\epsilon}b(X^{\epsilon,\nu}_s,Y^{\epsilon,\nu}_s,\mc{L}(X^{\epsilon,\nu}_s))\cdot\psi(s,X^{\epsilon,\nu}_s,\mc{L}(X^{\epsilon,\nu}_s))ds - \int_0^t  \gamma_1(X^{\epsilon,\nu}_s,Y^{\epsilon,\nu}_s,\mc{L}(X^{\epsilon,\nu}_s))\cdot\psi(s,X^{\epsilon,\nu}_s,\mc{L}(X^{\epsilon,\nu}_s))\\
&+2D_1(X^{\epsilon,\nu}_s,Y^{\epsilon,\nu}_s,\mc{L}(X^{\epsilon,\nu}_s)):\partial_x\psi(s,X^{\epsilon,\nu}_s,\mc{L}(X^{\epsilon,\nu}_s))ds\\
&-\int_0^t \int_{\R^{2d}}\psi^\top(s,X^{\epsilon,\nu}_s,\mc{L}(X^{\epsilon,\nu}_s))\partial_\mu \Phi(X^{\epsilon,\nu}_s,Y^{\epsilon,\nu}_s,\mc{L}(X^{\epsilon,\nu}_s))[x]b(x,y,\mc{L}(X^{\epsilon,\nu}_s))\mc{L}(X^{\epsilon,\nu}_s,Y^{\epsilon,\nu}_s)(dx,dy)ds\\
&-\int_0^t \int_{\R^{2d}}\Phi^\top(X^{\epsilon,\nu}_s,Y^{\epsilon,\nu}_s,\mc{L}(X^{\epsilon,\nu}_s))\partial_\mu \psi(s,X^{\epsilon,\nu}_s,\mc{L}(X^{\epsilon,\nu}_s))[x]b(x,y,\mc{L}(X^{\epsilon,\nu}_s))\mc{L}(X^{\epsilon,\nu}_s,Y^{\epsilon,\nu}_s)(dx,dy)ds\biggr]\biggr|\\
&\leq C\epsilon(1+t)\norm{\psi}_{\mc{M}_b^{\tilde{\bm{\zeta}}}([0,T]\times\R^d\times\mc{P}_2(\R^d);\R^d)}.
\end{align*}
\end{proposition}
\begin{proof}
First we note that we may apply It\^o's formula for measure-dependent functions (Proposition 5.102 in \cite{CD}, see also \cite{RocknerMcKeanVlasov} Section 4.2 for a similar computation) to $\Phi(X^{\epsilon,\nu}_t,Y^{\epsilon,\nu}_t,\mc{L}(X^{\epsilon,\nu}_t))\cdot\psi(t,X^{\epsilon,\nu}_t,\mc{L}(X^{\epsilon,\nu}_t))$. We have more than enough differentiability of $\Phi$ in $x,y,$ and $\mu$ to apply It\^o's formula as per Proposition \ref{proposition:allneededregularity} (see also Remark \ref{remark:sufficientconditionsforItos}). Thus, we have:
\begin{align*}
&\int_0^t \frac{1}{\epsilon}b(X^{\epsilon,\nu}_s,Y^{\epsilon,\nu}_s,\mc{L}(X^{\epsilon,\nu}_s))\cdot\psi(s,X^{\epsilon,\nu}_s,\mc{L}(X^{\epsilon,\nu}_s))ds - \int_0^t  \gamma_1(X^{\epsilon,\nu}_s,Y^{\epsilon,\nu}_s,\mc{L}(X^{\epsilon,\nu}_s))\cdot\psi(s,X^{\epsilon,\nu}_s,\mc{L}(X^{\epsilon,\nu}_s))\\
&+2D_1(X^{\epsilon,\nu}_s,Y^{\epsilon,\nu}_s,\mc{L}(X^{\epsilon,\nu}_s)):\partial_x\psi(s,X^{\epsilon,\nu}_s,\mc{L}(X^{\epsilon,\nu}_s))ds\\
&-\int_0^t \int_{\R^{2d}}\psi^\top(s,X^{\epsilon,\nu}_s,\mc{L}(X^{\epsilon,\nu}_s))\partial_\mu \Phi(X^{\epsilon,\nu}_s,Y^{\epsilon,\nu}_s,\mc{L}(X^{\epsilon,\nu}_s))[x]b(x,y,\mc{L}(X^{\epsilon,\nu}_s))\mc{L}(X^{\epsilon,\nu}_s,Y^{\epsilon,\nu}_s)(dx,dy)ds\\
&-\int_0^t \int_{\R^{2d}}\Phi^\top(X^{\epsilon,\nu}_s,Y^{\epsilon,\nu}_s,\mc{L}(X^{\epsilon,\nu}_s))\partial_\mu \psi(s,X^{\epsilon,\nu}_s,\mc{L}(X^{\epsilon,\nu}_s))[x]b(x,y,\mc{L}(X^{\epsilon,\nu}_s))\mc{L}(X^{\epsilon,\nu}_s,Y^{\epsilon,\nu}_s)(dx,dy)ds \\
&= \sum_{k=1}^5 A^\epsilon_k(t),
\end{align*}
where
\begin{align*}
A^\epsilon_1(t)& = \epsilon\biggl[\Phi(X^{\epsilon,\nu}_0,Y^{\epsilon,\nu}_0,\mc{L}(X^{\epsilon,\nu}_0))\cdot\psi(0,X^{\epsilon,\nu}_0,\mc{L}(X^{\epsilon,\nu}_0))-\Phi(X^{\epsilon,\nu}_t,Y^{\epsilon,\nu}_t,\mc{L}(X^{\epsilon,\nu}_t))\cdot\psi(t,X^{\epsilon,\nu}_t,\mc{L}(X^{\epsilon,\nu}_t)) \biggr]\\
A^\epsilon_2(t)& = \epsilon \int_0^t \Phi(X^{\epsilon,\nu}_s,Y^{\epsilon,\nu}_s,\mc{L}(X^{\epsilon,\nu}_s))\cdot\dot{\psi}(s,X^{\epsilon,\nu}_s,\mc{L}(X^{\epsilon,\nu}_s))\\
&+\psi^\top(s,X^{\epsilon,\nu}_s,\mc{L}(X^{\epsilon,\nu}_s))\partial_x\Phi(X^{\epsilon,\nu}_s,Y^{\epsilon,\nu}_s,\mc{L}(X^{\epsilon,\nu}_s))c(X^{\epsilon,\nu}_s,Y^{\epsilon,\nu}_s,\mc{L}(X^{\epsilon,\nu}_s))\\
&+\Phi^\top(X^{\epsilon,\nu}_s,Y^{\epsilon,\nu}_s,\mc{L}(X^{\epsilon,\nu}_s))\partial_x\psi(s,X^{\epsilon,\nu}_s,\mc{L}(X^{\epsilon,\nu}_s)) c(X^{\epsilon,\nu}_s,Y^{\epsilon,\nu}_s,\mc{L}(X^{\epsilon,\nu}_s)) \\
&+\frac{1}{2}\biggl[\psi^\top(s,X^{\epsilon,\nu}_s,\mc{L}(X^{\epsilon,\nu}_s))\sigma\sigma^\top(X^{\epsilon,\nu}_s,Y^{\epsilon,\nu}_s,\mc{L}(X^{\epsilon,\nu}_s):\partial^2_x\Phi(X^{\epsilon,\nu}_s,Y^{\epsilon,\nu}_s,\mc{L}(X^{\epsilon,\nu}_s) \\
&+2\sigma\sigma^\top(X^{\epsilon,\nu}_s,Y^{\epsilon,\nu}_s,\mc{L}(X^{\epsilon,\nu}_s): [(\partial_x\Phi)^\top(X^{\epsilon,\nu}_s,Y^{\epsilon,\nu}_s,\mc{L}(X^{\epsilon,\nu}_s))\partial_x\psi(s,X^{\epsilon,\nu}_s,\mc{L}(X^{\epsilon,\nu}_s))]\\
&+\Phi^\top(X^{\epsilon,\nu}_s,Y^{\epsilon,\nu}_s,\mc{L}(X^{\epsilon,\nu}_s))\sigma\sigma^\top (X^{\epsilon,\nu}_s,Y^{\epsilon,\nu}_s,\mc{L}(X^{\epsilon,\nu}_s)):\partial^2_x\psi(s,X^{\epsilon,\nu}_s,\mc{L}(X^{\epsilon,\nu}_s)) \biggr]ds\\
A^\epsilon_3(t)& = \epsilon \int_0^t \int_{\R^{2d}}\psi^\top(s,X^{\epsilon,\nu}_s,\mc{L}(X^{\epsilon,\nu}_s)) \partial_\mu \Phi(X^{\epsilon,\nu}_s,Y^{\epsilon,\nu}_s,\mc{L}(X^{\epsilon,\nu}_s))[x]c(x,y,\mc{L}(X^{\epsilon,\nu}_s))\\
&+\Phi^\top(X^{\epsilon,\nu}_s,Y^{\epsilon,\nu}_s,\mc{L}(X^{\epsilon,\nu}_s))\partial_\mu\psi(s,X^{\epsilon,\nu}_s,\mc{L}(X^{\epsilon,\nu}_s))[x] c(x,y,\mc{L}(X^{\epsilon,\nu}_s)) \\
&+  \frac{1}{2}\biggl[\psi^\top(s,X^{\epsilon,\nu}_s,\mc{L}(X^{\epsilon,\nu}_s))\sigma\sigma^\top(x,y,\mc{L}(X^{\epsilon,\nu}_s)):\partial_z\partial_\mu \Phi(X^{\epsilon,\nu}_s,Y^{\epsilon,\nu}_s,\mc{L}(X^{\epsilon,\nu}_s))[x]\\
&+\Phi^\top(X^{\epsilon,\nu}_s,Y^{\epsilon,\nu}_s,\mc{L}(X^{\epsilon,\nu}_s))\sigma\sigma^\top(x,y,\mc{L}(X^{\epsilon,\nu}_s)):\partial_z\partial_\mu\psi(s,X^{\epsilon,\nu}_s,\mc{L}(X^{\epsilon,\nu}_s))[x]\biggr]\mc{L}(X^{\epsilon,\nu}_s,Y^{\epsilon,\nu}_s)(dx,dy)ds\\
A^\epsilon_4(t)& = \epsilon \int_0^t \biggl[\psi^\top(s,X^{\epsilon,\nu}_s,\mc{L}(X^{\epsilon,\nu}_s))\partial_x\Phi(X^{\epsilon,\nu}_s,Y^{\epsilon,\nu}_s,\mc{L}(X^{\epsilon,\nu}_s))+\Phi^\top(X^{\epsilon,\nu}_s,Y^{\epsilon,\nu}_s,\mc{L}(X^{\epsilon,\nu}_s))\partial_x\psi(s,X^{\epsilon,\nu}_s,\mc{L}(X^{\epsilon,\nu}_s)) \biggr]\\
&\hspace{8cm}\sigma(X^{\epsilon,\nu}_s,Y^{\epsilon,\nu}_s,\mc{L}(X^{\epsilon,\nu}_s))dW_s\\
&+\int_0^t\psi^\top(s,X^{\epsilon,\nu}_s,\mc{L}(X^{\epsilon,\nu}_s))\partial_y\Phi(X^{\epsilon,\nu}_s,Y^{\epsilon,\nu}_s,\mc{L}(X^{\epsilon,\nu}_s))\tau_1(X^{\epsilon,\nu}_s,Y^{\epsilon,\nu}_s,\mc{L}(X^{\epsilon,\nu}_s))dW_s \\
&+ \int_0^t \psi^\top(s,X^{\epsilon,\nu}_s,\mc{L}(X^{\epsilon,\nu}_s)) \partial_y\Phi(X^{\epsilon,\nu}_s,Y^{\epsilon,\nu}_s,\mc{L}(X^{\epsilon,\nu}_s))\tau_2(X^{\epsilon,\nu}_s,Y^{\epsilon,\nu}_s,\mc{L}(X^{\epsilon,\nu}_s))dB_s.\\
\end{align*}
Via the same Proposition \ref{proposition:allneededregularity} and Assumption \ref{assumption:regularityofcoefficientsnew}, all the coefficients, as well as $\Phi$ and its derivatives, which appear in $A^\epsilon_k,k=1,2,3,4$ grow at most polynomially in $y$ uniformly in their other arguments.
Thus, by Lemma \ref{lemma:barYuniformbound}, the martingale terms collected in $A^\epsilon_4(t)$ vanish in expectation, and there exists $m\in\bb{N}$ such that:\begin{align*}
&\biggl|\E\biggl[A^\epsilon_1(t)\biggr]\biggr|\leq\E\biggl[|A^\epsilon_1(t)|\biggr]\\
&\quad\leq \epsilon \sup_{t\in[0,T]}\E\biggl[|\Phi(X^{\epsilon,\nu}_t,Y^{\epsilon,\nu}_t,\mc{L}(X^{\epsilon,\nu}_t))|\biggr]\norm{\psi}_\infty \\
&\quad\leq \epsilon C\biggl(1+\sup_{t\in[0,T]}\E\biggl[|Y^{\epsilon,\nu}_t|^{2m}\biggr]\biggr)\norm{\psi}_\infty \\
&\quad\leq \epsilon C \norm{\psi}_\infty\\
&\biggl|\E\biggl[A^\epsilon_2(t)\biggr]\biggr|\leq \E\biggl[|A^\epsilon_2(t)|\biggr]\\
&\quad\leq \epsilon C t \sup_{s\in [0,T]}\E\biggl[|\Phi(X^{\epsilon,\nu}_s,Y^{\epsilon,\nu}_s,\mc{L}(X^{\epsilon,\nu}_s))|+|\partial_x\Phi(X^{\epsilon,\nu}_s,Y^{\epsilon,\nu}_s,\mc{L}(X^{\epsilon,\nu}_s))||c(X^{\epsilon,\nu}_s,Y^{\epsilon,\nu}_s,\mc{L}(X^{\epsilon,\nu}_s))|\\
&\quad+|\Phi(X^{\epsilon,\nu}_s,Y^{\epsilon,\nu}_s,\mc{L}(X^{\epsilon,\nu}_s))||c(X^{\epsilon,\nu}_s,Y^{\epsilon,\nu}_s,\mc{L}(X^{\epsilon,\nu}_s))| \\
&\quad+|\sigma\sigma^\top(X^{\epsilon,\nu}_s,Y^{\epsilon,\nu}_s,\mc{L}(X^{\epsilon,\nu}_s)||\partial^2_x\Phi(X^{\epsilon,\nu}_s,Y^{\epsilon,\nu}_s,\mc{L}(X^{\epsilon,\nu}_s)| \\
&\quad+|\sigma\sigma^\top(X^{\epsilon,\nu}_s,Y^{\epsilon,\nu}_s,\mc{L}(X^{\epsilon,\nu}_s)|| \partial_x\Phi(X^{\epsilon,\nu}_s,Y^{\epsilon,\nu}_s,\mc{L}(X^{\epsilon,\nu}_s))|\\
&\quad+|\Phi(X^{\epsilon,\nu}_s,Y^{\epsilon,\nu}_s,\mc{L}(X^{\epsilon,\nu}_s))||\sigma\sigma^\top (X^{\epsilon,\nu}_s,Y^{\epsilon,\nu}_s,\mc{L}(X^{\epsilon,\nu}_s))| \biggr]\sup_{\mu\in\mc{P}_2(\R^d)}\norm{\psi(\cdot,\cdot,\mu)}_{C^{1,2}_b([0,T];\R^d)}\\
&\quad\leq  \epsilon C t\biggl(1+\sup_{t\in[0,T]}\E\biggl[|Y^{\epsilon,\nu}_t|^{2m}\biggr]\biggr)\sup_{\mu\in\mc{P}_2(\R^d)}\norm{\psi(\cdot,\cdot,\mu)}_{C^{1,2}_b([0,T];\R^d)}\\
&\quad\leq \epsilon C t\sup_{\mu\in\mc{P}_2(\R^d)}\norm{\psi(\cdot,\cdot,\mu)}_{C^{1,2}_b([0,T];\R^d)}\\
&\biggl|\E\biggl[A^\epsilon_3(t)\biggr]\biggr|\leq \E\biggl[|A^\epsilon_3(t)|\biggr]\\
&\quad\leq \epsilon C t\sup_{s\in [0,T]}\E\biggl[\int_{\R^{2d}} |\partial_\mu \Phi(X^{\epsilon,\nu}_s,Y^{\epsilon,\nu}_s,\mc{L}(X^{\epsilon,\nu}_s))[x]||c(x,y,\mc{L}(X^{\epsilon,\nu}_s))|\\
&\quad+|\Phi(X^{\epsilon,\nu}_s,Y^{\epsilon,\nu}_s,\mc{L}(X^{\epsilon,\nu}_s))||c(x,y,\mc{L}(X^{\epsilon,\nu}_s))| \\
&\quad+  |\sigma\sigma^\top(x,y,\mc{L}(X^{\epsilon,\nu}_s))||\partial_z\partial_\mu \Phi(X^{\epsilon,\nu}_s,Y^{\epsilon,\nu}_s,\mc{L}(X^{\epsilon,\nu}_s))[x]|\\
&\quad+|\Phi(X^{\epsilon,\nu}_s,Y^{\epsilon,\nu}_s,\mc{L}(X^{\epsilon,\nu}_s))||\sigma\sigma^\top(x,y,\mc{L}(X^{\epsilon,\nu}_s))|\mc{L}(X^{\epsilon,\nu}_s,Y^{\epsilon,\nu}_s)(dx,dy)\biggr](\norm{\psi}_\infty+\norm{\partial_\mu\psi}_\infty+\norm{\partial_z\partial_\mu\psi}_\infty)\\
&\quad\leq \epsilon C t\sup_{s\in [0,T]}\biggl\lbrace\E\biggl[\sup_{x\in\R^d}|\partial_\mu \Phi(X^{\epsilon,\nu}_s,Y^{\epsilon,\nu}_s,\mc{L}(X^{\epsilon,\nu}_s))[x]|^2+|\Phi(X^{\epsilon,\nu}_s,Y^{\epsilon,\nu}_s,\mc{L}(X^{\epsilon,\nu}_s))|^2 \\
&\hspace{8cm}+ \sup_{x\in\R^d}|\partial_z\partial_\mu \Phi(X^{\epsilon,\nu}_s,Y^{\epsilon,\nu}_s,\mc{L}(X^{\epsilon,\nu}_s))[x]|^2\biggr]\\
&\quad+\int_{\R^{2d}}|c(x,y,\mc{L}(X^{\epsilon,\nu}_s))|^2+|\sigma\sigma^\top(x,y,\mc{L}(X^{\epsilon,\nu}_s))|^2 \mc{L}(X^{\epsilon,\nu}_s,Y^{\epsilon,\nu}_s)(dx,dy)\biggr\rbrace(\norm{\psi}_\infty+\norm{\partial_\mu\psi}_\infty+\norm{\partial_z\partial_\mu\psi}_\infty)\\
&\quad\leq \epsilon C t\biggl(1+\sup_{s\in [0,T]}\biggl\lbrace\E\biggl[|Y^{\epsilon,\nu}_s|^{2m}\biggr]+\int_{\R^{2d}}|y|^{2m} \mc{L}(X^{\epsilon,\nu}_s,Y^{\epsilon,\nu}_s)(dx,dy)\biggr\rbrace\biggr)(\norm{\psi}_\infty+\norm{\partial_\mu\psi}_\infty+\norm{\partial_z\partial_\mu\psi}_\infty)\\
& \quad= \epsilon C t\biggl(1+\sup_{s\in [0,T]}\E\biggl[|Y^{\epsilon,\nu}_s|^{2m}\biggr]\biggr)(\norm{\psi}_\infty+\norm{\partial_\mu\psi}_\infty+\norm{\partial_z\partial_\mu\psi}_\infty)\\
&\quad\leq \epsilon C t(\norm{\psi}_\infty+\norm{\partial_\mu\psi}_\infty+\norm{\partial_z\partial_\mu\psi}_\infty).
\end{align*}
Combining the above bounds, we get the desired result.
\end{proof}

\begin{proposition}\label{prop:purpleterm2}
In the setup of Proposition \ref{prop:fluctuationestimateparticles2}, for any $t\in[0,T]$ and $\epsilon\in (0,1]$:
\begin{align*}
&\biggl|\E\biggl[\int_0^t  \int_{\R^{2d}}\psi^\top(s,X^{\epsilon,\nu}_s,\mc{L}(X^{\epsilon,\nu}_s))\partial_\mu \Phi(X^{\epsilon,\nu}_s,Y^{\epsilon,\nu}_s,\mc{L}(X^{\epsilon,\nu}_s))[x]b(x,y,\mc{L}(X^{\epsilon,\nu}_s))\mc{L}(X^{\epsilon,\nu}_s,Y^{\epsilon,\nu}_s)(dx,dy)ds\biggr]\biggr|\\
&\leq C\epsilon(1+t)\norm{\psi}_{\mc{M}_b^{\tilde{\bm{\zeta}}}([0,T]\times\R^d\times\mc{P}_2(\R^d);\R^d)}.
\end{align*}

\end{proposition}
\begin{proof}
To see this, we note for all $t\in [0,T]$:
\begin{align*}
&\E\biggl[\int_0^t \int_{\R^{2d}} \psi^\top(s,X^{\epsilon,\nu}_s,\mc{L}(X^{\epsilon,\nu}_s))\partial_{\mu}\Phi(X^{\epsilon,\nu}_s,Y^{\epsilon,\nu}_s,\mc{L}(X^{\epsilon,\nu}_s))[x]b(x,y,\mc{L}(X^{\epsilon,\nu}_s))\mc{L}(X^{\epsilon,\nu}_s,Y^{\epsilon,\nu}_s)(dx,dy)ds \biggr]\\
&=  \E\biggl[ \int_0^t \psi^\top(s,X^{\epsilon,\nu}_s,\mc{L}(X^{\epsilon,\nu}_s))\partial_{\mu}\Phi(X^{\epsilon,\nu}_s,Y^{\epsilon,\nu}_s,\mc{L}(X^{\epsilon,\nu}_s))[\hat{X}^{\epsilon}_s]b(\hat{X}^{\epsilon}_s,\hat{Y}^{\epsilon}_s,\mc{L}(X^{\epsilon,\nu}_s)) ds \biggr]
\end{align*}
where $(\hat{X}^{\epsilon},\hat{Y}^{\epsilon})$ is an independent copy of $(X^{\epsilon,\nu},Y^{\epsilon,\nu})$. Thus it is enough to show that the expectation of $\int_0^t \psi^\top(s,X^{\epsilon,\nu}_s,\mc{L}(X^{\epsilon,\nu}_s))\partial_{\mu}\Phi(X^{\epsilon,\nu}_s,Y^{\epsilon,\nu}_s,\mc{L}(X^{\epsilon,\nu}_s))[\hat{X}^{\epsilon}_s] b(\hat{X}^{\epsilon}_s,\hat{Y}^{\epsilon}_s,\mc{L}(X^{\epsilon,\nu}_s)) ds$ is $O(\epsilon)$, where
\begin{align*}
dX^{\epsilon,\nu}_t &= \biggl[\frac{1}{\epsilon}b(X^{\epsilon,\nu}_t,Y^{\epsilon,\nu}_t,\mc{L}(X^{\epsilon,\nu}_t))+ c(X^{\epsilon,\nu}_t,Y^{\epsilon,\nu}_t,\mc{L}(X^{\epsilon,\nu}_t)) \biggr]dt + \sigma(X^{\epsilon,\nu}_t,Y^{\epsilon,\nu}_t,\mc{L}(X^{\epsilon,\nu}_t))dW_t\\
dY^{\epsilon,\nu}_t & = \frac{1}{\epsilon}\biggl[\frac{1}{\epsilon}f(X^{\epsilon,\nu}_t,Y^{\epsilon,\nu}_t,\mc{L}(X^{\epsilon,\nu}_t))+ g(X^{\epsilon,\nu}_t,Y^{\epsilon,\nu}_t,\mc{L}(X^{\epsilon,\nu}_t)) \biggr]dt \nonumber\\
&+ \frac{1}{\epsilon}\biggl[\tau_1(X^{\epsilon,\nu}_t,Y^{\epsilon,\nu}_t,\mc{L}(X^{\epsilon,\nu}_t))dW_t+\tau_2(X^{\epsilon,\nu}_t,Y^{\epsilon,\nu}_t,\mc{L}(X^{\epsilon,\nu}_t))dB_t\biggr]\\
d\hat{X}^{\epsilon}_t &= \biggl[\frac{1}{\epsilon}b(\hat{X}^{\epsilon}_t,\hat{Y}^{\epsilon}_t,\mc{L}(X^{\epsilon,\nu}_t))+ c(\hat{X}^{\epsilon}_t,\hat{Y}^{\epsilon}_t,\mc{L}(X^{\epsilon,\nu}_t)) \biggr]dt + \sigma(\hat{X}^{\epsilon}_t,\hat{Y}^{\epsilon}_t,\mc{L}(X^{\epsilon,\nu}_t))d\hat{W}_t\\
d\hat{Y}^{\epsilon}_t & = \frac{1}{\epsilon}\biggl[\frac{1}{\epsilon}f(\hat{X}^{\epsilon}_t,\hat{Y}^{\epsilon}_t,\mc{L}(X^{\epsilon,\nu}_t))+ g(\hat{X}^{\epsilon}_t,\hat{Y}^{\epsilon}_t,\mc{L}(X^{\epsilon,\nu}_t)) \biggr]dt \nonumber\\
&+ \frac{1}{\epsilon}\biggl[\tau_1(\hat{X}^{\epsilon}_t,\hat{Y}^{\epsilon}_t,\mc{L}(X^{\epsilon,\nu}_t))d\hat{W}_t+\tau_2(\hat{X}^{\epsilon}_t,\hat{Y}^{\epsilon}_t,\mc{L}(X^{\epsilon,\nu}_t))d\hat{B}_t\biggr]\\
(X^{\epsilon,\nu}_0,Y^{\epsilon,\nu}_0,\hat{X}^\epsilon_0,\hat{Y}^\epsilon_0)&=(\eta,\zeta,\hat{\eta},\hat{\zeta})
\end{align*}
and $W,B,\hat{W},\hat{B}$ are independent $m-$ dimensional standard Brownian motions, $(\eta,\zeta)$ and $(\hat{\eta},\hat{\zeta})$ are IID with $(\eta,\zeta)$ from Equation \eqref{eq:slow-fastMcKeanVlasov}, and $(\eta,\zeta,\hat{\eta},\hat{\zeta})$ is independent of $(W,B,\hat{W},\hat{B})$.

Recall the operator $\mc{L}_{x,\mu}$ from Equation \eqref{eq:frozengeneratormold}. For fixed $x\in\R,\mu\in\mc{P}(\R)$, this is the generator of the frozen process
\begin{align}\label{eq:frozenprocess1}
dY^{x,\mu}_t = f(x,Y^{x,\mu}_t,\mu)dt+\tau_1(x,Y^{x,\mu}_t,\mu)dW_t+\tau_2(x,Y^{x,\mu}_t,\mu)dB_t
\end{align}
for $W_t,B_t$ independent $m-$ dimensional standard Brownian motions.
We now introduce a new generator $\mc{L}^2_{x,\bar{x},\mu}$ parameterized by $x,\bar{x}\in\R^d,\mu\in\mc{P}_2(\R^d)$ which acts on $\psi\in C^2_b(\R^2)$ by
\begin{align}\label{eq:2copiesgenerator}
\mc{L}^2_{x,\bar{x},\mu}\psi(y,\bar{y}) &= f(x,y,\mu)\cdot\partial_y\psi(y,\bar{y})+f(\bar{x},\bar{y},\mu)\cdot\partial_{\bar{y}}\psi(y,\bar{y})\\
&+ a(x,y,\mu):\partial^2_y\psi(y,\bar{y})+a(\bar{x},\bar{y},\mu):\partial^2_{\bar{y}}\psi_{\bar{y}\bar{y}}(y,\bar{y}).\nonumber
\end{align}

This is the generator associated to the 2-dimensional process solving 2 independent copies of Equation \eqref{eq:frozenprocess1} where the same parameter $\mu$ enters both equations, but different $x,\bar{x}$ enter each equation, i.e.
\begin{align}\label{eq:frozenprocess2}
dY^{x,\mu}_t &= f(x,Y^{x,\mu}_t,\mu)dt+\tau_1(x,Y^{x,\mu}_t,\mu)dW_t+\tau_2(x,Y^{x,\mu}_t,\mu)dB_t\\
d\bar{Y}^{\bar{x},\mu}_t &= f(\bar{x},\bar{Y}^{\bar{x},\mu}_t,\mu)dt+\tau_1(\bar{x},\bar{Y}^{\bar{x},\mu}_t,\mu)d\bar{W}_t+\tau_2(\bar{x},\bar{Y}^{\bar{x},\mu}_t,\mu)d\bar{B}_t\nonumber.
\end{align}
for $W_t,B_t,\bar{W}_t,\bar{B}_t$ independent $m-$ dimensional standard Brownian motions.

One can see that the unique distributional solution of the adjoint equation
\begin{align}
(\mc{L}^2_{x,\bar{x},\mu})^*\bar{\pi}(\cdot;x,\bar{x},\mu) &=0\nonumber\\
\int_{\R^2}\bar{\pi}(dy,d\bar{y};x,\bar{x},\mu)&=1,\forall x,\bar{x}\in\R,\mu\in\mc{P}(\R) \nonumber
\end{align}
 is given by
\begin{align}\label{eq:doublefrozeninvariantmeasure}
\bar{\pi}(dy,d\bar{y};x,\bar{x},\mu) = \pi(dy;x,\mu)\otimes\pi(d\bar{y};\bar{x},\mu)
\end{align}
where $\pi$ is as in Equation \eqref{eq:invariantmeasureold}.

We now consider $\chi(x,\bar{x},y,\bar{y},\mu):\R^d\times\R^d\times\R^d\times\R^d\times\mc{P}(\R^d)\tto \R^d$ solving
\begin{align}\label{eq:doublecorrectorproblem}
\mc{L}^2_{x,\bar{x},\mu}\chi_l(x,\bar{x},y,\bar{y},\mu) &= -[\partial_\mu \Phi(\bar{x},\bar{y},\mu)[x]b(x,y,\mu)]_l,l=1,...,d\\
\int_{\R^d}\int_{\R^d}\chi(x,\bar{x},y,\bar{y},\mu)\pi(dy;x,\mu)\pi(d\bar{y},\bar{x},\mu)&=0.\nonumber
\end{align}
Note that by the centering condition, Equation \eqref{eq:centeringconditionold}, the right hand side of Equation \eqref{eq:doublecorrectorproblem} integrates against $\bar{\pi}$ from Equation \eqref{eq:doublefrozeninvariantmeasure} to $0$ for all $x,\bar{x},\mu$.
By Proposition \ref{proposition:allneededregularity} and Lemma \ref{lemma:Ganguly1DCellProblemResult}/Remark \ref{remark:regularityfordoubledequations} we have existence and uniqueness for the solution $\chi$ to Equation \eqref{eq:doublecorrectorproblem}, and moreover $\chi$ enjoys more than enough regularity in $x,\bar{x},y,\bar{y},$ and $\mu$ to apply It\^o's formula for measure dependent functions to $\chi(\hat{X}^\epsilon_t,X^{\epsilon,\nu}_t,\hat{Y}^\epsilon_t,Y^{\epsilon,\nu}_t,\mc{L}(X^{\epsilon,\nu}_t))\cdot\psi(t,X^{\epsilon,\nu}_t,\mc{L}(X^{\epsilon,\nu}_t))$. Thus:
\begin{align*}
\int_0^t \psi^\top(s,X^{\epsilon,\nu}_s,\mc{L}(X^{\epsilon,\nu}_s))\partial_{\mu}\Phi(X^{\epsilon,\nu}_s,Y^{\epsilon,\nu}_s,\mc{L}(X^{\epsilon,\nu}_s))[\hat{X}^{\epsilon}_s]b(\hat{X}^{\epsilon}_s,\hat{Y}^{\epsilon}_s,\mc{L}(X^{\epsilon,\nu}_s)) ds& = \sum_{k=1}^6B^{\epsilon}(t)\\
\end{align*}
where
\begin{align*}
B_1^\epsilon(t)& =  \epsilon^2\biggl[\chi(\hat{X}^\epsilon_0,X^{\epsilon,\nu}_0,\hat{Y}^\epsilon_0,Y^{\epsilon,\nu}_0,\mc{L}(X^{\epsilon,\nu}_0))\cdot\psi(0,X^{\epsilon,\nu}_0,\mc{L}(X^{\epsilon,\nu}_0))-\chi(\hat{X}^\epsilon_t,X^{\epsilon,\nu}_t,\hat{Y}^\epsilon_t,Y^{\epsilon,\nu}_t,\mc{L}(X^{\epsilon,\nu}_t))\cdot\psi(t,X^{\epsilon,\nu}_t,\mc{L}(X^{\epsilon,\nu}_t)) \biggr] \\
B_2^\epsilon(t)&= \epsilon\int_0^t \psi^\top(s,X^{\epsilon,\nu}_s,\mc{L}(X^{\epsilon,\nu}_s))\partial_x \chi\hat{b} +\psi^\top(s,X^{\epsilon,\nu}_s,\mc{L}(X^{\epsilon,\nu}_s))\partial_{\bar{x}} \chi\tilde{b}+\chi^\top\partial_x\psi(s,X^{\epsilon,\nu}_s,\mc{L}(X^{\epsilon,\nu}_s))\tilde{b} \\
&+ \psi^\top(s,X^{\epsilon,\nu}_s,\mc{L}(X^{\epsilon,\nu}_s))\partial_y \chi\hat{g} + \psi^\top(s,X^{\epsilon,\nu}_s,\mc{L}(X^{\epsilon,\nu}_s))\partial_{\bar{y}} \chi\tilde{g}+ \psi^\top(s,X^{\epsilon,\nu}_s,\mc{L}(X^{\epsilon,\nu}_s))\hat{\tau}_1\hat{\sigma}^\top :\partial_x\partial_y \chi \\
&+ \psi^\top(s,X^{\epsilon,\nu}_s,\mc{L}(X^{\epsilon,\nu}_s))\tilde{\tau}_1\tilde{\sigma}^\top:\partial_{\bar{x}}\partial_{\bar{y}} \chi+\tilde{\tau}_1\tilde{\sigma}^\top:[(\partial_x\psi)^\top(s,X^{\epsilon,\nu}_s,\mc{L}(X^{\epsilon,\nu}_s))\partial_{\bar{y}} \chi] \\
& + \int_{\R^{2d}}\psi^\top(s,X^{\epsilon,\nu}_s,\mc{L}(X^{\epsilon,\nu}_s))\partial_\mu\chi(\hat{X}^\epsilon_s,X^{\epsilon,\nu}_s,\hat{Y}^\epsilon_s,Y^{\epsilon,\nu}_s,\mc{L}(X^{\epsilon,\nu}_s))[x]b(x,y,\mc{L}(X^{\epsilon,\nu}_s))\mc{L}(X^{\epsilon,\nu}_s,Y^{\epsilon,\nu}_s)(dx,dy) \\
& + \int_{\R^{2d}}\chi^\top(\hat{X}^\epsilon_s,X^{\epsilon,\nu}_s,\hat{Y}^\epsilon_s,Y^{\epsilon,\nu}_s,\mc{L}(X^{\epsilon,\nu}_s))\partial_\mu\psi(s,X^{\epsilon,\nu}_s,\mc{L}(X^{\epsilon,\nu}_s))[x]b(x,y,\mc{L}(X^{\epsilon,\nu}_s))\mc{L}(X^{\epsilon,\nu}_s,Y^{\epsilon,\nu}_s)(dx,dy)ds\\
B_3^\epsilon(t)&= \epsilon^2 \int_0^t \chi\cdot \dot{\psi}(s,X^{\epsilon,\nu}_s,\mc{L}(X^{\epsilon,\nu}_s))+ \psi^\top(s,X^{\epsilon,\nu}_s,\mc{L}(X^{\epsilon,\nu}_s))\partial_x \chi\hat{c} + \frac{1}{2}\psi^\top(s,X^{\epsilon,\nu}_s,\mc{L}(X^{\epsilon,\nu}_s))\hat{\sigma}\hat{\sigma}^\top:\partial^2_x \chi\\
&+\psi^\top(s,X^{\epsilon,\nu}_s,\mc{L}(X^{\epsilon,\nu}_s))\partial_{\bar{x}} \chi \tilde{c}+\chi^\top \partial_x\psi(s,X^{\epsilon,\nu}_s,\mc{L}(X^{\epsilon,\nu}_s))\tilde{c} \\
&+ \frac{1}{2}\psi^\top(s,X^{\epsilon,\nu}_s,\mc{L}(X^{\epsilon,\nu}_s))\tilde{\sigma}\tilde{\sigma}^\top:\partial^2_{\bar{x}} \chi+\tilde{\sigma}\tilde{\sigma}^\top:[(\partial_{\bar{x}} \chi)^\top\partial_x\psi(s,X^{\epsilon,\nu}_s,\mc{L}(X^{\epsilon,\nu}_s))]\\
&+\frac{1}{2}\chi^\top\tilde{\sigma}\tilde{\sigma}^\top:\partial^2_x\psi(s,X^{\epsilon,\nu}_s,\mc{L}(X^{\epsilon,\nu}_s))ds\\
B_4^\epsilon(t) &= \epsilon^2\int_0^t \int_{\R^{2d}}\psi^\top(s,X^{\epsilon,\nu}_s,\mc{L}(X^{\epsilon,\nu}_s))\partial_\mu\chi(\hat{X}^\epsilon_s,X^{\epsilon,\nu}_s,\hat{Y}^\epsilon_s,Y^{\epsilon,\nu}_s,\mc{L}(X^{\epsilon,\nu}_s))[x]c(x,y,\mc{L}(X^{\epsilon,\nu}_s))\mc{L}(X^{\epsilon,\nu}_s,Y^{\epsilon,\nu}_s)(dx,dy)\\
&+\int_{\R^{2d}}\chi^\top(\hat{X}^\epsilon_s,X^{\epsilon,\nu}_s,\hat{Y}^\epsilon_s,Y^{\epsilon,\nu}_s,\mc{L}(X^{\epsilon,\nu}_s))\partial_\mu\psi(s,X^{\epsilon,\nu}_s,\mc{L}(X^{\epsilon,\nu}_s))[x]c(x,y,\mc{L}(X^{\epsilon,\nu}_s))\mc{L}(X^{\epsilon,\nu}_s,Y^{\epsilon,\nu}_s)(dx,dy)\\
&+\frac{1}{2}\int_{\R^{2d}}\psi^\top(s,X^{\epsilon,\nu}_s,\mc{L}(X^{\epsilon,\nu}_s))\sigma\sigma^\top(x,y,\mc{L}(X^{\epsilon,\nu}_s)):\partial_z\partial_\mu\chi(\hat{X}^\epsilon_s,X^{\epsilon,\nu}_s,\hat{Y}^\epsilon_s,Y^{\epsilon,\nu}_s,\mc{L}(X^{\epsilon,\nu}_s))[x]\mc{L}(X^{\epsilon,\nu}_s,Y^{\epsilon,\nu}_s)(dx,dy)\\
&+\frac{1}{2}\int_{\R^{2d}}\chi^\top(\hat{X}^\epsilon_s,X^{\epsilon,\nu}_s,\hat{Y}^\epsilon_s,Y^{\epsilon,\nu}_s,\mc{L}(X^{\epsilon,\nu}_s))\sigma\sigma^\top(x,y,\mc{L}(X^{\epsilon,\nu}_s)):\partial_z\partial_\mu\psi(s,X^{\epsilon,\nu}_s,\mc{L}(X^{\epsilon,\nu}_s))[x]\mc{L}(X^{\epsilon,\nu}_s,Y^{\epsilon,\nu}_s)(dx,dy)ds\\
B_5^\epsilon(t)&= \epsilon\int_0^t \psi^\top(s,X^{\epsilon,\nu}_s,\mc{L}(X^{\epsilon,\nu}_s))  \partial_y \chi \hat{\tau}_1d\hat{W}_s + \epsilon \int_0^t \psi^\top(s,X^{\epsilon,\nu}_s,\mc{L}(X^{\epsilon,\nu}_s))\partial_{\bar{y}} \chi\tilde{\tau}_1 dW_s + \\
&+ \epsilon\int_0^t \psi^\top(s,X^{\epsilon,\nu}_s,\mc{L}(X^{\epsilon,\nu}_s)) \partial_y \chi \hat{\tau}_2d\hat{B}_s+\epsilon \int_0^t  \psi^\top(s,X^{\epsilon,\nu}_s,\mc{L}(X^{\epsilon,\nu}_s))\partial_{\bar{y}} \chi\tilde{\tau}_2dB_s\\
B_6^\epsilon(t)&= \epsilon^2 \int_0^t \psi^\top(s,X^{\epsilon,\nu}_s,\mc{L}(X^{\epsilon,\nu}_s))\partial_x \chi\hat{\sigma} d\hat{W}_s + \epsilon^2 \int_0^t \biggl[\psi^\top(s,X^{\epsilon,\nu}_s,\mc{L}(X^{\epsilon,\nu}_s))\partial_{\bar{x}} \chi+\chi^\top\partial_x\psi(s,X^{\epsilon,\nu}_s,\mc{L}(X^{\epsilon,\nu}_s))\biggr]\bar{\sigma}dW_s .
\end{align*}
where $\hat{c}$ denotes $(\hat{X}^\epsilon_s,\hat{Y}_s,\mc{L}(X^{\epsilon,\nu}_s))$ as an argument, $\tilde{c}$ denotes $(X^{\epsilon,\nu}_s,Y^{\epsilon,\nu}_s,\mc{L}(X^{\epsilon,\nu}_s))$ as an argument, and similarly for the other coefficients. The argument of $\chi$ where suppressed is $(\hat{X}^\epsilon_s,X^{\epsilon,\nu}_s,\hat{Y}^\epsilon_s,Y^{\epsilon,\nu}_s,\mc{L}(X^{\epsilon,\nu}_s))$.

Then, by computations along the same lines as Proposition \ref{prop:fluctuationestimateparticles2}, using the polynomial growth of $\chi$ and its derivatives in $y,\bar{y}$ uniformly in $x,\bar{x}$, and $\mu$ from Proposition \ref{proposition:allneededregularity}, we get there is $m\in\bb{N}$ such that:
\begin{align*}
&\biggl|\E\biggl[ \int_0^t \psi^\top(s,X^{\epsilon,\nu}_s,\mc{L}(X^{\epsilon,\nu}_s))\partial_{\mu}\Phi(X^{\epsilon,\nu}_s,Y^{\epsilon,\nu}_s,\mc{L}(X^{\epsilon,\nu}_s))[\hat{X}^{\epsilon}_s]b(\hat{X}^{\epsilon}_s,\hat{Y}^{\epsilon}_s,\mc{L}(X^{\epsilon,\nu}_s)) ds \biggr]\biggr|\\
&\leq C[\epsilon^2+(\epsilon+\epsilon^2)t](1+\sup_{t\in[0,T]}\E\biggl[|Y^{\epsilon,\nu}_t|^{2m} \biggr])\norm{\psi}_{\mc{M}_b^{\tilde{\bm{\zeta}}}([0,T]\times\R^d\times\mc{P}_2(\R^d);\R^d)}\\
&\leq C\epsilon (1+t)\norm{\psi}_{\mc{M}_b^{\tilde{\bm{\zeta}}}([0,T]\times\R^d\times\mc{P}_2(\R^d);\R^d)}
\end{align*}
by Lemma \ref{lemma:barYuniformbound} and the fact that $\epsilon\in (0,1]$.\end{proof}

\begin{proposition}\label{prop:newpurpleterm}
Assume \ref{assumption:uniformellipticity} - \ref{assumption:regularityofcoefficientsnew}. Consider $\psi\in \mc{M}_b^{\tilde{\bm{\zeta}}_2}([0,T]\times\R^d\times\mc{P}_2(\R^d);\R^d),$ where
\begin{align*}
\tilde{\bm{\zeta}}_2\coloneqq \br{(1,j,0), (1,0,j),(2,0,(0,k)):j=0,1,2,k=0,1}
\end{align*}
such that $\partial_\mu\dot{\psi}$ exists, is jointly continuous in all arguments, and $\sup_{t\in[0,T],x,z\in\R^d,\mu\in\mc{P}_2(\R^d)}|\partial_\mu\dot{\psi}(t,x,\mu)[z]|<\infty$. Then we have for any $t\in [0,T]$ and $\epsilon\in (0,1]$:
\begin{align*}
&\biggl|\E\biggl[\int_0^t  \int_{\R^{2d}}\Phi^\top(X^{\epsilon,\nu}_s,Y^{\epsilon,\nu}_s,\mc{L}(X^{\epsilon,\nu}_s))\partial_\mu\psi(s,X^{\epsilon,\nu}_s,\mc{L}(X^{\epsilon,\nu}_s)) [x]b(x,y,\mc{L}(X^{\epsilon,\nu}_s))\mc{L}(X^{\epsilon,\nu}_s,Y^{\epsilon,\nu}_s)(dx,dy)ds\biggr]\biggr|\\
&\leq C\epsilon(1+t)[\sup_{t\in [0,T]}\norm{\psi(t,\cdot,\cdot)}_{\mc{M}_b^{\tilde{\bm{\zeta}}_2}(\R^d\times\mc{P}_2(\R^d);\R^d)}+\sup_{t\in[0,T],x,z\in\R^d,\mu\in\mc{P}_2(\R^d)}|\partial_\mu\dot{\psi}(t,x,\mu)[z]|].
\end{align*}
\end{proposition}
\begin{proof}
As in Proposition \ref{prop:purpleterm2}, it suffices to show the bound for
\begin{align*}
&\biggl|\E\biggl[ \int_0^t\Phi^\top(X^{\epsilon,\nu}_s,Y^{\epsilon,\nu}_s,\mc{L}(X^{\epsilon,\nu}_s))\partial_{\mu}\psi(s,X^{\epsilon,\nu}_s,\mc{L}(X^{\epsilon,\nu}_s))[\hat{X}^{\epsilon}_s] b(\hat{X}^{\epsilon}_s,\hat{Y}^{\epsilon}_s,\mc{L}(X^{\epsilon,\nu}_s)) ds \biggr]\biggr|,
\end{align*}
where $(\hat{X}^{\epsilon},\hat{Y}^{\epsilon})$ is an independent copy of $(X^{\epsilon,\nu},Y^{\epsilon,\nu})$.
We consider the new doubled Poisson equation
\begin{align}\label{eq:tildechi}
\mc{L}^2_{x,\bar{x},\mu}\tilde{\chi}_{k,l}(x,\bar{x},y,\bar{y},\mu) &= -b_l(x,y,\mu)\Phi_k(\bar{x},\bar{y},\mu),l,k=1,...,d\\
\int_{\R^d}\int_{\R^d}\tilde{\chi}(x,\bar{x},y,\bar{y},\mu)\pi(dy;x,\mu)\pi(d\bar{y},\bar{x},\mu)&=0z\nonumber
\end{align}
where $\mc{L}^2_{x,\bar{x},\mu}$ is as in Equation \eqref{eq:2copiesgenerator}. We note once again that by assumption \eqref{eq:centeringconditionold}, the inhomogeneity is centered with respect to the invariant measure associated to $\mc{L}^2_{x,\bar{x},\mu}$ given in Equation \eqref{eq:doublefrozeninvariantmeasure}.
Applying It\^o's formula for measure dependent functions to $$\tilde{\chi}(\hat{X}^\epsilon_t,X^{\epsilon,\nu}_t,\hat{Y}^\epsilon_t,Y^{\epsilon,\nu}_t,\mc{L}(X^{\epsilon,\nu}_t)):\partial_{\mu}\psi(s,X^{\epsilon,\nu}_s,\mc{L}(X^{\epsilon,\nu}_s))[\hat{X}^{\epsilon}_s]$$ (using here the required differentiability of $\tilde{\chi}$ granted by Proposition \ref{proposition:allneededregularity}),
\begin{align*}
\int_0^t \Phi^\top(X^{\epsilon,\nu}_s,Y^{\epsilon,\nu}_s,\mc{L}(X^{\epsilon,\nu}_s))\partial_{\mu}\psi(s,X^{\epsilon,\nu}_s,\mc{L}(X^{\epsilon,\nu}_s))[\hat{X}^{\epsilon}_s] b(\hat{X}^{\epsilon}_s,\hat{Y}^{\epsilon}_s,\mc{L}(X^{\epsilon,\nu}_s))ds& = \sum_{j=1}^6C_j^{\epsilon}(t)
\end{align*}
where, for $k=1,...,d$, denoting by $\tilde{\chi}_k$ the vector comprising the $k$'th row of $\tilde{\chi}$:
\begin{align*}
&C_1^\epsilon(t) =  \epsilon^2\biggl[\tilde{\chi}(\hat{X}^\epsilon_0,X^{\epsilon,\nu}_0,\hat{Y}^\epsilon_0,Y^{\epsilon,\nu}_0,\mc{L}(X^{\epsilon,\nu}_0)):\partial_\mu\psi(0,X^{\epsilon,\nu}_0,\mc{L}(X^{\epsilon,\nu}_0))[\hat{X}^\epsilon_0]\\
&\quad-\tilde{\chi}(\hat{X}^\epsilon_t,X^{\epsilon,\nu}_t,\hat{Y}^\epsilon_t,Y^{\epsilon,\nu}_t,\mc{L}(X^{\epsilon,\nu}_t)):\partial_\mu\psi(t,X^{\epsilon,\nu}_t,\mc{L}(X^{\epsilon,\nu}_t))[\hat{X}^\epsilon_t] \biggr] \\
&C_2^\epsilon(t)= \epsilon \sum_{k=1}^d\biggl\lbrace\int_0^t (\partial_\mu\psi_k)^\top(s,X^{\epsilon,\nu}_s,\mc{L}(X^{\epsilon,\nu}_s))[\hat{X}^\epsilon_s]\partial_x\tilde{\chi}_k \hat{b} +(\tilde{\chi}_k)^\top\partial_z\partial_\mu\psi_k(s,X^{\epsilon,\nu}_s,\mc{L}(X^{\epsilon,\nu}_s))[\hat{X}^\epsilon_s]\hat{b} \\
&\quad+(\partial_\mu\psi_k)^\top(s,X^{\epsilon,\nu}_s,\mc{L}(X^{\epsilon,\nu}_s))[\hat{X}^\epsilon_s] \partial_{\bar{x}}\tilde{\chi}_k\tilde{b}+(\tilde{\chi}_k)^\top(\partial_\mu\partial_x\psi_k)^\top(s,X^{\epsilon,\nu}_s,\mc{L}(X^{\epsilon,\nu}_s))[\hat{X}^\epsilon_s]\tilde{b} \\
&\quad+ (\partial_\mu\psi_k)^\top(s,X^{\epsilon,\nu}_s,\mc{L}(X^{\epsilon,\nu}_s))[\hat{X}^\epsilon_s]\partial_y\tilde{\chi}_k\hat{g} + (\partial_\mu\psi_k)^\top(s,X^{\epsilon,\nu}_s,\mc{L}(X^{\epsilon,\nu}_s))[\hat{X}^\epsilon_s]\partial_{\bar{y}}\tilde{\chi}_k\tilde{g}\\
&\quad+ (\partial_\mu\psi_k)^\top(s,X^{\epsilon,\nu}_s,\mc{L}(X^{\epsilon,\nu}_s))[\hat{X}^\epsilon_s]\hat{\tau}_1\hat{\sigma}^\top:\partial_x\partial_y\tilde{\chi}_k+\hat{\tau}_1\hat{\sigma}^\top:[(\partial_z\partial_\mu\psi_k)^\top(s,X^{\epsilon,\nu}_s,\mc{L}(X^{\epsilon,\nu}_s))[\hat{X}^\epsilon_s]\partial_y\tilde{\chi}_k]\\
&\quad+ (\partial_\mu\psi_k)^\top(s,X^{\epsilon,\nu}_s,\mc{L}(X^{\epsilon,\nu}_s))[\hat{X}^\epsilon_s]\tilde{\tau}_1\tilde{\sigma}^\top:\partial_{\bar{x}}\partial_{\bar{y}}\tilde{\chi}_k+\tilde{\tau}_1\tilde{\sigma}^\top:[\partial_\mu\partial_x\psi_k(s,X^{\epsilon,\nu}_s,\mc{L}(X^{\epsilon,\nu}_s))[\hat{X}^\epsilon_s]\partial_{\bar{y}}\tilde{\chi}_k] \\
& + \int_{\R^{2d}}(\partial_\mu\psi_k)^\top(s,X^{\epsilon,\nu}_s,\mc{L}(X^{\epsilon,\nu}_s))[\hat{X}^\epsilon_s]\partial_\mu\tilde{\chi}_k(\hat{X}^\epsilon_s,X^{\epsilon,\nu}_s,\hat{Y}^\epsilon_s,Y^{\epsilon,\nu}_s,\mc{L}(X^{\epsilon,\nu}_s))[x]b(x,y,\mc{L}(X^{\epsilon,\nu}_s))\mc{L}(X^{\epsilon,\nu}_s,Y^{\epsilon,\nu}_s)(dx,dy) \\
&\quad + \int_{\R^{2d}}(\tilde{\chi}_k)^\top(\hat{X}^\epsilon_s,X^{\epsilon,\nu}_s,\hat{Y}^\epsilon_s,Y^{\epsilon,\nu}_s,\mc{L}(X^{\epsilon,\nu}_s))\partial^2_\mu\psi_k(s,X^{\epsilon,\nu}_s,\mc{L}(X^{\epsilon,\nu}_s))[\hat{X}^\epsilon_s,x]b(x,y,\mc{L}(X^{\epsilon,\nu}_s))\mc{L}(X^{\epsilon,\nu}_s,Y^{\epsilon,\nu}_s)(dx,dy)ds\biggr\rbrace\\
&C_3^\epsilon(t)= \epsilon^2 \sum_{k=1}^d\biggl\lbrace\int_0^t \tilde{\chi}_k\cdot\partial_\mu\dot{\psi_k}(s,X^{\epsilon,\nu}_s,\mc{L}(X^{\epsilon,\nu}_s))[\hat{X}^\epsilon_s]+ (\partial_\mu\psi_k)^\top(s,X^{\epsilon,\nu}_s,\mc{L}(X^{\epsilon,\nu}_s))[\hat{X}^\epsilon_s]\partial_x\tilde{\chi}_k\hat{c}\\
&\quad+(\tilde{\chi}_k)^\top\partial_z\partial_\mu\psi_k(s,X^{\epsilon,\nu}_s,\mc{L}(X^{\epsilon,\nu}_s))[\hat{X}^\epsilon_s]]\hat{c} + \frac{1}{2}(\partial_\mu\psi_k)^\top(s,X^{\epsilon,\nu}_s,\mc{L}(X^{\epsilon,\nu}_s))[\hat{X}^\epsilon_s]\hat{\sigma}\hat{\sigma}^\top:\partial^2_x\tilde{\chi}_k\\
&\quad+\hat{\sigma}\hat{\sigma}^\top:[(\partial_z\partial_\mu\psi_k)^\top(s,X^{\epsilon,\nu}_s,\mc{L}(X^{\epsilon,\nu}_s))[\hat{X}^\epsilon_s]\partial_x\tilde{\chi}_k] +\frac{1}{2}(\tilde{\chi}_k)^\top\hat{\sigma}\hat{\sigma}^\top:\partial^2_z\partial_\mu\psi_k(s,X^{\epsilon,\nu}_s,\mc{L}(X^{\epsilon,\nu}_s))[\hat{X}^\epsilon_s]\\
&\quad+(\partial_\mu\psi_k)^\top(s,X^{\epsilon,\nu}_s,\mc{L}(X^{\epsilon,\nu}_s))[\hat{X}^\epsilon_s]\partial_{\bar{x}}\tilde{\chi}_k\tilde{c}+(\tilde{\chi}_k)^\top(\partial_\mu\partial_x\psi_k)^\top(s,X^{\epsilon,\nu}_s,\mc{L}(X^{\epsilon,\nu}_s))[\hat{X}^\epsilon_s]\tilde{c} \\
&\quad+ \frac{1}{2}(\partial_\mu\psi_k)^\top(s,X^{\epsilon,\nu}_s,\mc{L}(X^{\epsilon,\nu}_s))[\hat{X}^\epsilon_s]\tilde{\sigma}\tilde{\sigma}^\top:\partial_{\bar{x}}\partial_{\bar{x}}\tilde{\chi}_k+\tilde{\sigma}\tilde{\sigma}^\top:[\partial_\mu\partial_x\psi_k(s,X^{\epsilon,\nu}_s,\mc{L}(X^{\epsilon,\nu}_s))[\hat{X}^\epsilon_s]\partial_{\bar{x}}\tilde{\chi}_k]\\
&\quad+\frac{1}{2}(\tilde{\chi}_k)^\top\tilde{\sigma}\tilde{\sigma}^\top:\partial^2_x\partial_\mu\psi_k(s,X^{\epsilon,\nu}_s,\mc{L}(X^{\epsilon,\nu}_s))[\hat{X}^\epsilon_s]ds\biggr\rbrace\\
&C_4^\epsilon(t) = \epsilon^2\sum_{k=1}^d\biggl\lbrace\int_0^t \int_{\R^{2d}}(\partial_\mu\psi_k)^\top(s,X^{\epsilon,\nu}_s,\mc{L}(X^{\epsilon,\nu}_s))[\hat{X}^\epsilon_s]\partial_\mu\tilde{\chi}_k(\hat{X}^\epsilon_s,X^{\epsilon,\nu}_s,\hat{Y}^\epsilon_s,Y^{\epsilon,\nu}_s,\mc{L}(X^{\epsilon,\nu}_s))[x]\\
&\hspace{10cm}c(x,y,\mc{L}(X^{\epsilon,\nu}_s))\mc{L}(X^{\epsilon,\nu}_s,Y^{\epsilon,\nu}_s)(dx,dy)\\
&\quad+\int_{\R^{2d}}(\tilde{\chi}_k)^\top(\hat{X}^\epsilon_s,X^{\epsilon,\nu}_s,\hat{Y}^\epsilon_s,Y^{\epsilon,\nu}_s,\mc{L}(X^{\epsilon,\nu}_s))\partial^2_\mu\psi_k(s,X^{\epsilon,\nu}_s,\mc{L}(X^{\epsilon,\nu}_s))[\hat{X}^\epsilon_s,x]c(x,y,\mc{L}(X^{\epsilon,\nu}_s))\mc{L}(X^{\epsilon,\nu}_s,Y^{\epsilon,\nu}_s)(dx,dy)\\
&\quad+\frac{1}{2}\int_{\R^{2d}}(\partial_\mu\psi_k)^\top(s,X^{\epsilon,\nu}_s,\mc{L}(X^{\epsilon,\nu}_s))[\hat{X}^\epsilon_s]\sigma\sigma^\top(x,y,\mc{L}(X^{\epsilon,\nu}_s))\\
&\hspace{7cm}:\partial_z\partial_\mu\tilde{\chi}_k(\hat{X}^\epsilon_s,X^{\epsilon,\nu}_s,\hat{Y}^\epsilon_s,Y^{\epsilon,\nu}_s,\mc{L}(X^{\epsilon,\nu}_s))[x]\mc{L}(X^{\epsilon,\nu}_s,Y^{\epsilon,\nu}_s)(dx,dy)\\
&\quad+\frac{1}{2}\int_{\R^{2d}}(\tilde{\chi}_k)^\top(\hat{X}^\epsilon_s,X^{\epsilon,\nu}_s,\hat{Y}^\epsilon_s,Y^{\epsilon,\nu}_s,\mc{L}(X^{\epsilon,\nu}_s))\sigma\sigma^\top(x,y,\mc{L}(X^{\epsilon,\nu}_s))\\
&\hspace{7cm}:\partial_{z_2}\partial^2_\mu\psi_k(s,X^{\epsilon,\nu}_s,\mc{L}(X^{\epsilon,\nu}_s))[\hat{X}^\epsilon_s,x]\mc{L}(X^{\epsilon,\nu}_s,Y^{\epsilon,\nu}_s)(dx,dy)ds\biggr\rbrace\\
&C_5^\epsilon(t)= \epsilon\sum_{k=1}^d\biggl\lbrace\int_0^t  (\partial_\mu\psi_k)^\top(s,X^{\epsilon,\nu}_s,\mc{L}(X^{\epsilon,\nu}_s))[\hat{X}^\epsilon_s]\partial_y\tilde{\chi}_k \hat{\tau}_1d\hat{W}_s + \epsilon \int_0^t  (\partial_\mu\psi_k)^\top(s,X^{\epsilon,\nu}_s,\mc{L}(X^{\epsilon,\nu}_s))[\hat{X}^\epsilon_s]\partial_{\bar{y}}\tilde{\chi}_k\tilde{\tau}_1dW_s \\
&\quad+ \epsilon\int_0^t  (\partial_\mu\psi_k)^\top(s,X^{\epsilon,\nu}_s,\mc{L}(X^{\epsilon,\nu}_s))[\hat{X}^\epsilon_s]\partial_y\tilde{\chi}_k \hat{\tau}_2d\hat{B}_s \\
&\quad+ \epsilon \int_0^t  (\partial_\mu\psi_k)^\top(s,X^{\epsilon,\nu}_s,\mc{L}(X^{\epsilon,\nu}_s))[\hat{X}^\epsilon_s]\partial_{\bar{y}}\tilde{\chi}_k\tilde{\tau}_2dB_s\biggr\rbrace\\
&C_6^\epsilon(t)= \epsilon^2 \sum_{k=1}^d\biggl\lbrace\int_0^t \biggl[(\partial_\mu\psi_k)^\top(s,X^{\epsilon,\nu}_s,\mc{L}(X^{\epsilon,\nu}_s))[\hat{X}^\epsilon_s]\partial_x\tilde{\chi}_k+(\tilde{\chi}_k)^\top\partial_z\partial_\mu\psi_k(s,X^{\epsilon,\nu}_s,\mc{L}(X^{\epsilon,\nu}_s))[\hat{X}^\epsilon_s]\biggr]\hat{\sigma} d\hat{W}_s \\
&\quad+ \epsilon^2 \int_0^t \biggl[(\partial_\mu\psi_k)^\top(s,X^{\epsilon,\nu}_s,\mc{L}(X^{\epsilon,\nu}_s))[\hat{X}^\epsilon_s]\partial_{\bar{x}}\tilde{\chi}_k+(\tilde{\chi}_k)^\top(\partial_\mu\partial_x\psi_k)^\top(s,X^{\epsilon,\nu}_s,\mc{L}(X^{\epsilon,\nu}_s))[\hat{X}^\epsilon_s]\biggr]\tilde{\sigma}dW_s\biggr\rbrace,
\end{align*}
where $\hat{c}$ denotes $(\hat{X}^\epsilon_s,\hat{Y}_s,\mc{L}(X_s))$ as an argument, $\tilde{c}$ denotes $(X^{\epsilon,\nu}_s,Y^{\epsilon,\nu}_s,\mc{L}(X^{\epsilon,\nu}_s))$ as an argument, and the other coefficients. The argument of $\chi$ where suppressed is $(\hat{X}^\epsilon_s,X^{\epsilon,\nu}_s,\hat{Y}^\epsilon_s,Y^{\epsilon,\nu}_s,\mc{L}(X^{\epsilon,\nu}_s))$.

Then,  by computations along the same lines as Proposition \ref{prop:fluctuationestimateparticles2}, using the polynomial growth of $\tilde{\chi}$ and its derivatives in $y,\bar{y}$ uniformly in $x,\bar{x}$, and $\mu$ from Proposition \ref{proposition:allneededregularity}, we get there is $m\in\bb{N}$ such that:
\begin{align*}
& \biggl|\E\biggl[\int_0^t  \Phi^\top(X^{\epsilon,\nu}_s,Y^{\epsilon,\nu}_s,\mc{L}(X^{\epsilon,\nu}_s))\partial_{\mu}\psi(s,X^{\epsilon,\nu}_s,\mc{L}(X^{\epsilon,\nu}_s))[\hat{X}^{\epsilon}_s] b(\hat{X}^{\epsilon}_s,\hat{Y}^{\epsilon}_s,\mc{L}(X^{\epsilon,\nu}_s)) ds \biggr]\biggr|\\
&\leq C[\epsilon^2+(\epsilon+\epsilon^2)t](1+\sup_{t\in[0,T]}\E\biggl[|Y^{\epsilon,\nu}_t|^{2m} \biggr])[\sup_{t\in [0,T]}\norm{\psi(t,\cdot,\cdot)}_{\mc{M}_b^{\tilde{\bm{\zeta}}_2}(\R^d\times\mc{P}_2(\R^d);\R^d)}+\\
&\qquad+\sup_{t\in[0,T],x,z\in\R^d,\mu\in\mc{P}_2(\R^d)}|\partial_\mu\dot{\psi}(t,x,\mu)[z]|]\\
&\leq C\epsilon (1+t)[\sup_{t\in [0,T]}\norm{\psi(t,\cdot,\cdot)}_{\mc{M}_b^{\tilde{\bm{\zeta}}_2}(\R^d\times\mc{P}_2(\R^d);\R^d)}+\sup_{t\in[0,T],x,z\in\R^d,\mu\in\mc{P}_2(\R^d)}|\partial_\mu\dot{\psi}(t,x,\mu)[z]|]
\end{align*}
by Lemma \ref{lemma:barYuniformbound} and the fact that $\epsilon\in (0,1]$.\end{proof}

\begin{proposition}\label{prop:llntypefluctuationestimate2}
Assume \ref{assumption:uniformellipticity} - \ref{assumption:strongexistence}, and let $F:\R^d\times\R^d\times\mc{P}_2(\R^d)\tto \R^k$ be any function such that there exists $\Xi$ the unique classical solution to Equation \eqref{eq:driftcorrectorproblem} with $\Xi \in \mc{M}_p^{\tilde{\bm{\zeta}}}(\R^d\times\R^d\times\mc{P}_2(\R^d);\R^d)$ and $\partial_y\Xi \in \mc{M}_p^{\bm{\zeta}_{x}}(\R^d\times\R^d\times\mc{P}_2(\R^d);\R^{k\times d})$ where $\tilde{\bm{\zeta}}$ is as in Proposition \ref{prop:fluctuationestimateparticles2} and
\begin{align*}
\bm{\zeta}_{x}\coloneqq \br{(0,j,0):j\in\br{0,1}}
\end{align*}
(in particular this holds with $F=\gamma,D$ under the additional Assumption \ref{assumption:regularityofcoefficientsnew} via Proposition \ref{proposition:allneededregularity}). Then for $\bar{F}(x,\mu)\coloneqq \int_{\R^d} F(x,y,\mu)\pi(dy;x,\mu)$, with $\pi$ as in Equation \eqref{eq:invariantmeasureold}, any $\psi\in \mc{M}_b^{\tilde{\bm{\zeta}}}([0,T]\times\R^d\times\mc{P}_2(\R^d);\R^k)$ and $t\in[0,T]$, and $\epsilon\in (0,1]$:
\begin{align*}
\biggl|\E\biggl[\int_0^t \biggl(F(X^{\epsilon,\nu}_s,Y^{\epsilon,\nu}_s,\mc{L}(X^{\epsilon,\nu}_s))-\bar{F}(X^{\epsilon,\nu}_s,\mc{L}(X^{\epsilon,\nu}_s))\biggr)\cdot\psi(s,X^{\epsilon,\nu}_s,\mc{L}(X^{\epsilon,\nu}_s))dt\biggr]\biggr|\leq C\epsilon(1+t)\norm{\psi}_{\mc{M}_b^{\tilde{\bm{\zeta}}}([0,T]\times\R^d\times\mc{P}_2(\R^d);\R^k)}.
\end{align*}
\end{proposition}
\begin{proof}
By assumption, we can consider $\Xi:\R^d\times\R^d\times\mc{P}(\R^d)\tto\R^k$ the unique classical solution to
\begin{align}\label{eq:driftcorrectorproblem}
\mc{L}_{x,\mu}\Xi_l(x,y,\mu) &= -[F_l(x,y,\mu)-\int_{\R^d}F_l(x,y,\mu)\pi(dy;x,\mu)],l=1,...,k\\
\int_{\R}\Xi(x,y,\mu)\pi(dy;x,\mu)&=0\nonumber.
\end{align}
Applying It\^o's formula for measure-dependent functions to $\Xi(X^{\epsilon,\nu}_t,Y^{\epsilon,\nu}_t,\mc{L}(X^{\epsilon,\nu}_t))\cdot\psi(t,X^{\epsilon,\nu}_t)$, we get:
\begin{align*}
\int_0^t \biggl(F(X^{\epsilon,\nu}_s,Y^{\epsilon,\nu}_s,\mc{L}(X^{\epsilon,\nu}_s))-\bar{F}(X^{\epsilon,\nu}_s,\mc{L}(X^{\epsilon,\nu}_s))\biggr)\cdot\psi(s,X^{\epsilon,\nu}_s,\mc{L}(X^{\epsilon,\nu}_s))dt & =\sum_{j=1}^6 D_j^\epsilon(t)
\end{align*}
where
\begin{align*}
D_1^\epsilon(t)& =  \epsilon^2\biggl[\Xi(X^{\epsilon,\nu}_0,Y^{\epsilon,\nu}_0,\mc{L}(X^{\epsilon,\nu}_0))\cdot\psi(0,X^{\epsilon,\nu}_0,\mc{L}(X^{\epsilon,\nu}_0))-\Xi(X^{\epsilon,\nu}_t,Y^{\epsilon,\nu}_t,\mc{L}(X^{\epsilon,\nu}_t))\cdot\psi(t,X^{\epsilon,\nu}_t,\mc{L}(X^{\epsilon,\nu}_t)) \biggr] \\
D_2^\epsilon(t)&= \epsilon\int_0^t \psi^\top(s,X^{\epsilon,\nu}_s,\mc{L}(X^{\epsilon,\nu}_s))\partial_x\Xi b+\Xi^\top\partial_x\psi(s,X^{\epsilon,\nu}_s,\mc{L}(X^{\epsilon,\nu}_s))b  + \psi^\top(s,X^{\epsilon,\nu}_s,\mc{L}(X^{\epsilon,\nu}_s))\partial_y\Xi g\\
&+ \psi^\top(s,X^{\epsilon,\nu}_s,\mc{L}(X^{\epsilon,\nu}_s))\tau_1\sigma^\top:\partial_x\partial_y\Xi+\tau_1\sigma^\top:[(\partial_x\psi)^\top(s,X^{\epsilon,\nu}_s,\mc{L}(X^{\epsilon,\nu}_s))\partial_y\Xi] \\
& + \int_{\R^{2d}}\psi^\top(s,X^{\epsilon,\nu}_s,\mc{L}(X^{\epsilon,\nu}_s))\partial_\mu\Xi(X^{\epsilon,\nu}_s,Y^{\epsilon,\nu}_s,\mc{L}(X^{\epsilon,\nu}_s))[x]b(x,y,\mc{L}(X^{\epsilon,\nu}_s))\mc{L}(X^{\epsilon,\nu}_s,Y^{\epsilon,\nu}_s)(dx,dy) \\
&+ \int_{\R^{2d}}\Xi^\top(X^{\epsilon,\nu}_s,Y^{\epsilon,\nu}_s,\mc{L}(X^{\epsilon,\nu}_s))\partial_\mu\psi(s,X^{\epsilon,\nu}_s,\mc{L}(X^{\epsilon,\nu}_s))[x]b(x,y,\mc{L}(X^{\epsilon,\nu}_s))\mc{L}(X^{\epsilon,\nu}_s,Y^{\epsilon,\nu}_s)(dx,dy) ds\\
D_3^\epsilon(t)&= \epsilon^2 \int_0^t \Xi\cdot\dot{\psi}(s,X^{\epsilon,\nu}_s,\mc{L}(X^{\epsilon,\nu}_s))+\psi^\top(s,X^{\epsilon,\nu}_s,\mc{L}(X^{\epsilon,\nu}_s))\partial_x\Xi c+\Xi^\top\partial_x\psi(s,X^{\epsilon,\nu}_s,\mc{L}(X^{\epsilon,\nu}_s)) c\\
&+ \frac{1}{2}\psi^\top(s,X^{\epsilon,\nu}_s,\mc{L}(X^{\epsilon,\nu}_s))\sigma\sigma^\top:\partial^2_x\Xi+\sigma\sigma^\top:[(\partial_x\Xi)^\top\partial_x\psi(s,X^{\epsilon,\nu}_s,\mc{L}(X^{\epsilon,\nu}_s))]\\
&+\frac{1}{2}\Xi^\top\sigma\sigma^\top:\partial^2_x\psi(s,X^{\epsilon,\nu}_s,\mc{L}(X^{\epsilon,\nu}_s))ds\\
D_4^\epsilon(t) &= \epsilon^2\int_0^t \int_{\R^{2d}}\psi^\top(s,X^{\epsilon,\nu}_s,\mc{L}(X^{\epsilon,\nu}_s))\partial_\mu\Xi(X^{\epsilon,\nu}_s,Y^{\epsilon,\nu}_s,\mc{L}(X^{\epsilon,\nu}_s))[x]c(x,y,\mc{L}(X^{\epsilon,\nu}_s))\mc{L}(X^{\epsilon,\nu}_s,Y^{\epsilon,\nu}_s)(dx,dy)\\
&+\int_{\R^{2d}}\Xi^\top(X^{\epsilon,\nu}_s,Y^{\epsilon,\nu}_s,\mc{L}(X^{\epsilon,\nu}_s))\partial_\mu\psi(s,X^{\epsilon,\nu}_s,\mc{L}(X^{\epsilon,\nu}_s))[x]c(x,y,\mc{L}(X^{\epsilon,\nu}_s))\mc{L}(X^{\epsilon,\nu}_s,Y^{\epsilon,\nu}_s)(dx,dy)\\
&+\frac{1}{2}\int_{\R^{2d}}\psi^\top(s,X^{\epsilon,\nu}_s,\mc{L}(X^{\epsilon,\nu}_s))\sigma\sigma^\top(x,y,\mc{L}(X^{\epsilon,\nu}_s)):\partial_z\partial_\mu\Xi(X^{\epsilon,\nu}_s,Y^{\epsilon,\nu}_s,\mc{L}(X^{\epsilon,\nu}_s))[x]\mc{L}(X^{\epsilon,\nu}_s,Y^{\epsilon,\nu}_s)(dx,dy)\\
&+\frac{1}{2}\int_{\R^{2d}}\Xi^\top(X^{\epsilon,\nu}_s,Y^{\epsilon,\nu}_s,\mc{L}(X^{\epsilon,\nu}_s))\sigma\sigma^\top(x,y,\mc{L}(X^{\epsilon,\nu}_s)):\partial_z\partial_\mu\psi(s,X^{\epsilon,\nu}_s,\mc{L}(X^{\epsilon,\nu}_s))[x]\mc{L}(X^{\epsilon,\nu}_s,Y^{\epsilon,\nu}_s)(dx,dy)ds\\
D_5^\epsilon(t)&=  \epsilon \int_0^t \psi^\top(s,X^{\epsilon,\nu}_s,\mc{L}(X^{\epsilon,\nu}_s))\partial_y\Xi \tau_1 dW_s+ \epsilon \int_0^t \psi^\top(s,X^{\epsilon,\nu}_s,\mc{L}(X^{\epsilon,\nu}_s))\partial_y\Xi\tau_2 dB_s\\
D_6^\epsilon(t)&= \epsilon^2 \int_0^t \biggl[\psi^\top(s,X^{\epsilon,\nu}_s,\mc{L}(X^{\epsilon,\nu}_s))\partial_x\Xi+\Xi^\top\partial_x\psi(s,X^{\epsilon,\nu}_s,\mc{L}(X^{\epsilon,\nu}_s))\biggr]\sigma dW_s .
\end{align*}
Here argument for the coefficients and $\Xi$ and its derivatives where it is suppressed is $(X^{\epsilon,\nu}_s,Y^{\epsilon,\nu}_s,\mc{L}(X^{\epsilon,\nu}_s))$. By computations along the same lines as Proposition \ref{prop:fluctuationestimateparticles2}, we get there is $m\in\bb{N}$ such that:
\begin{align*}
&\biggl|\E\biggl[\int_0^t \biggl(F(X^{\epsilon,\nu}_s,Y^{\epsilon,\nu}_s,\mc{L}(X^{\epsilon,\nu}_s))-\bar{F}(X^{\epsilon,\nu}_s,\mc{L}(X^{\epsilon,\nu}_s))\biggr)\psi(s,X^{\epsilon,\nu}_s,\mc{L}(X^{\epsilon,\nu}_s))dt\biggr]\biggr|\\
&\leq C[\epsilon+(\epsilon+\epsilon^2)t](1+\sup_{t\in[0,T]}\E\biggl[|Y^{\epsilon,\nu}_t|^{2m} \biggr])\norm{\psi}_{\mc{M}_b^{\tilde{\bm{\zeta}}}([0,T]\times\R^d\times\mc{P}_2(\R^d);\R^k)}\\
&\leq C\epsilon (1+t)\norm{\psi}_{\mc{M}_b^{\tilde{\bm{\zeta}}}([0,T]\times\R^d\times\mc{P}_2(\R^d);\R^k)}
\end{align*}
by Lemma \ref{lemma:barYuniformbound} and the fact that $\epsilon\in (0,1]$.
\end{proof}
\section{On the Cauchy Problem \eqref{eq:W2CauchyProblem}}\label{sec:onthecauchyproblem}
With the results of Section \ref{section:mckeanvlasovergodictheorems}, we are almost ready  to use the Cauchy Problem on Wasserstein Space in order to prove our main result, Theorem \ref{theo:mckeanvlasovaveraging}. Before doing so, we need a final Lemma, which provides the needed existence, uniqueness, and regularity of solutions to Equation \eqref{eq:W2CauchyProblem}, and is a refinement of \cite{CST} Theorem 2.15:
\begin{lemma}\label{lem:cauchyproblemregularity}
Let $\tau\in [0,T]$. Assume \ref{assumption:uniformellipticity}-\ref{assumption:uniformellipticityDbar} and let $G\in \mc{M}_{b,L}^{\dot{\bm{\zeta}}}(\mc{P}_2(\R^d);\R),$ where $\dot{\bm{\zeta}}$ is as in Theorem \ref{theo:mckeanvlasovaveraging}.
Then $U(t,\mu) = G(\mc{L}(X^\mu_t))$ is the unique solution to the PDE \eqref{eq:W2CauchyProblem}, and $U\in \mc{M}_{b,L}^{\dot{\bm{\zeta}}}([0,\tau]\times\mc{P}_2(\R^d);\R)$, with
\begin{align}\label{eq:regularityofW2cauchyproblemintermsofIC}
\norm{U}_{\mc{M}_b^{\dot{\bm{\zeta}}}([0,\tau]\times\mc{P}_2(\R^d);\R)}\leq C(T)\norm{G}_{\mc{M}_b^{\dot{\bm{\zeta}}}(\mc{P}_2(\R^d);\R)}.
\end{align}
Here $C(T)$ is independent of $G$, and depends only on $T$ and $\norm{\bar{\gamma}}_{\mc{M}_b^{\hat{\bm{\zeta}}}(\R^d\times\mc{P}_2(\R^d);\R^d)},\norm{\bar{D}^{1/2}}_{\mc{M}_b^{\hat{\bm{\zeta}}}(\R^d\times\mc{P}_2(\R^d);\R^{d\times d})}$ and their Lipschitz constants, where $\hat{\bm{\zeta}}$ is as in Equation \eqref{eq:collectionsofmultiindices}.

Moreover, $\partial_\mu \dot{U}(s,\mu)[z],\partial_z\partial_\mu \dot{U}(s,\mu)[z],\partial^2_\mu \dot{U}(s,\mu)[z_1,z_2]$ all exist and are jointly continuous in $t,\mu,z_1,z_2$, and
\begin{align}\label{eq:mixedtimederivatives}
\sup_{t\in [0,\tau],z_1,z_2\in \R^d,\mu\in\mc{P}_2(\R^d)}\max \br{|\partial_\mu \dot{U}(t,\mu)[z_1]|,|\partial_z\partial_\mu \dot{U}(t,\mu)[z_1]|,|\partial^2_\mu \dot{U}(t,\mu)[z_1,z_2]|}&\leq C\sup_{t\in [0,\tau]}\norm{U(t,\cdot)}_{\mc{M}_b^{\dot{\bm{\zeta}}}(\mc{P}_2(\R^d);\R)} \\
&\leq C(T)\norm{G}_{\mc{M}_b^{\dot{\bm{\zeta}}}(\mc{P}_2(\R^d);\R)}\nonumber.
\end{align}

Recall here that the superscript $\mu\in \mc{P}_2(\R^d)$ is denoting that $X^\mu_t$ from Equation \eqref{eq:averagedMcKeanVlasov} is initialized at a random variable $\xi$ independent from $W^2$ with $\mc{L}(\xi)=\mu$. Thus, varying $\mu$ in $U(t,\mu)$ is varying the initial distribution of the McKean-Vlasov SDE \eqref{eq:averagedMcKeanVlasov}.
\end{lemma}

\begin{proof}
First we note that indeed $\bar{\gamma}\in \mc{M}_{b,L}^{\hat{\bm{\zeta}}}(\R^d\times\mc{P}_2(\R^d);\R^d),\bar{D}^{1/2}\in \mc{M}_{b,L}^{\hat{\bm{\zeta}}}(\R^d\times\mc{P}_2(\R^d);\R^{d\times d})$ by Proposition \ref{proposition:allneededregularity}.

The representation for the solution on terms of $G$, uniqueness, and time differentiability is the subject of Theorem 7.2 in \cite{BLPR}. Note that by the flow property discussed in Remark \ref{remark:McKeanVlasovFlowProperty}, one can adapt the terminal condition formulation of the Cauchy Problem found in that paper to our form for the initial condition (see \cite{CM} Equation 1.2 and Theorem 5.8). \cite{CST} Theorem 2.15 outlines how to extend that theorem to higher derivatives, but there they do not track entirely what derivatives of the coefficients and initial condition are needed to control each specific derivative of $U$, and just write things in terms of the order of the derivatives of $U$ and the coefficients. We could directly apply this result, but that would be requiring $\bar{\gamma},\bar{D}^{1/2}$ have all mixed derivatives in $(x,\mu,z)$ of order 4 bounded and Lipschitz, and would imply that $U$ have all derivatives of order 4 in $(\mu,z)$ which are bounded and Lipschitz. We don't require order 4 differentiability of $U$ in $\mu$, and hence make a slight refinement.

If one repeats the computations of \cite{CST} in the full setting (using the proof of \cite{BLPR} Lemma 6.2 for guidance), we see $U\in \mc{M}_{b,L}^{\dot{\bm{\zeta}}}([0,\tau]\times\mc{P}_2(\R^d);\R)$ corresponds to the coefficients satisfying $\bar{\gamma}\in \mc{M}_{b,L}^{\hat{\bm{\zeta}}}(\R^d\times\mc{P}_2(\R^d);\R^d),\bar{D}^{1/2}\in \mc{M}_{b,L}^{\hat{\bm{\zeta}}}(\R^d\times\mc{P}_2(\R^d);\R^{d\times d})$. One should note that a small typo is made when going from (6.23) to (6.24) in \cite{BLPR}, and in (6.25) the terms multiplied by $\partial_x\partial_\mu \tilde{X}$ should have $\tilde{\xi}$ replaced by $z$, those multiplied by $\partial_x \bar{X}$ should have $\bar{\xi}$ replaced by $z$, and those multiplied by $\partial_x \tilde{X}$ should have $\tilde{\xi}$ replaced by $z$. It is also evident that the dependence is linear in terms of derivatives of the initial condition, even though neither of these results are framed as such. This establishes \eqref{eq:regularityofW2cauchyproblemintermsofIC}.

Finally, using that U solves \eqref{eq:W2CauchyProblem}, we can use a computation similar to Example 5 in Section 5.2.2 in \cite{CD} to see:
\begin{align*}
\partial_\mu\dot{U}(t,\mu)[z]& = (\partial_x\bar{\gamma})^\top(z,\mu)\partial_\mu U(t,\mu)[z]+(\partial_z\partial_\mu U)^\top(t,\mu)[z]\bar{\gamma}(z,\mu)+ \partial_z\partial_\mu U(\mu)[z]:\partial_x \bar{D}(z,\mu)\\
&+\bar{D}(z,\mu):\partial^2_z\partial_\mu U(\mu)[z]+\int_{\R^d}(\partial_\mu \bar{\gamma})^\top(y,\mu)[z]\partial_\mu U(t,\mu)[y]+ (\partial^2_\mu U(t,\mu))^\top[y,z]\bar{\gamma}(y,\mu) \\
&+ \partial_z\partial_\mu U(\mu)[y]:\partial_\mu\bar{D}(y,\mu)[z]+\bar{D}(y,\mu):\partial_{z_1}\partial^2_\mu U(\mu)[y,z]\mu(dy)\\
\end{align*}
and can arrive at similar representations for $\partial_z\partial_\mu\dot{U}(t,\mu)[z]$ and $\partial^2_\mu\dot{U}(t,\mu)[z_1,z_2]$, which yield the bounds \eqref{eq:mixedtimederivatives}.\end{proof}

Before providing the proof of Theorem \ref{theo:mckeanvlasovaveraging}, let us make a few observations about the role of the Equation \eqref{eq:W2CauchyProblem} and compare the proof method of this paper to that of \cite{RocknerFullyCoupled} Theorem 2.3, which uses the standard Cauchy Problem associated to SDEs (the backward Kolmogorov equation) to prove the analogous fully-coupled averaging result in the setting without measure dependence of the coefficients.

Therein, under sufficient regularity of the coefficients in the limiting equation and of the test function $\phi$, standard PDE results can be used to show that the solution of the backward Kolmogorov equation is in $C^{1,4}_b([0,T]\times\R^d)$, and that the rate of convergence is $\biggl|\E\biggl[\phi(X^\epsilon_t)- \phi(X_t) \biggr]\biggr|\leq C\epsilon \sup_{t\in[0,T]}[\norm{u(t,\cdot)}_{C^{4}_b(\R^d)}+\norm{\dot{u}(t,\cdot)}_{C^{2}_b(\R^d)}]$, where $C$ does not depend on $\phi$ or the initial condition of $X^\epsilon$ (here $X^\epsilon$ and $X$ are solutions of a standard SDEs, not McKean-Vlasov SDEs). Though this is not explicitly stated, due to the representation of solutions of the Cauchy problem as the initial condition integrated against the fundamental solution of the Cauchy problem, one can find that in fact $\sup_{t\in[0,T]}[\norm{u(t,\cdot)}_{C^{4}_b(\R^d)}+\norm{\dot{u}(t,\cdot)}_{C^{2}_b(\R^d)}]\leq \norm{\phi}_{C^{4}_b(\R^d)}$ (for standard PDE methods see e.g. \cite{LSU} Chapter 4 Sections 5 and 14 and \cite{Friedman} Chapter 9 Theorem 3, and for a probabilistic method see \cite{CerraiBook} Theorem 1.7.5 ). This is effectively our Corollary \ref{corollary:lineartestfunction} in the case of standard Fully-Coupled Slow-Fast SDEs.

In order to establish Corollary \ref{corollary:lineartestfunction}, one might be tempted to treat the coefficients of $X^\nu_t$ from Equation \eqref{eq:averagedMcKeanVlasov} as time dependent coefficients for a standard SDE, i.e. to let $\tilde{\gamma}(t,x) = \bar{\gamma}(x,\mc{L}(X^{\delta_x}_t))$ and $\tilde{D}(t,x) = \bar{D}(x,\mc{L}(X^{\delta_x}_t))$, and consider $\tilde{u}:[0,T]\times\R^d\tto \R$ solving
\begin{align*}
\dot{\tilde{u}}(t,x) &= \tilde{L}\tilde{u}(t,x)\coloneqq \tilde{\gamma}(t,x)\cdot\partial_x\tilde{u}(t,x)+\tilde{D}(t,x):\partial^2_x\tilde{u}(t,x), t\in (0,T],x\in\R^d,\\
u(0,x)& = \phi(x)
\end{align*}
and use the method of \cite{RocknerFullyCoupled} to establish a rate of convergence of $X^{\epsilon,\delta_x}$ to $X^{\delta_x}$. However, there are a few disadvantages to this approach compared when considering the convergence of McKean-Vlasov SDEs.

The first is that the PDE results regarding differentiability of the fundamental solution of the Cauchy problem depends on having regularity (differentiability and H\"older continuity in $x$ and uniform continuity in $t$) of the coefficients $\tilde{\gamma}(t,x),\tilde{D}(t,x)$. Since these coefficients are dependent implicitly on $\mc{L}(X^{\delta_x}_t)$, this would require understanding regularity properties of the Law of $X^{\delta_x}_t$, and understanding how these regularity properties transfer through the second argument of $\bar{\gamma}$ and $\bar{D}$, which a priori may have any form of dependence on their measure argument as long as it is smooth enough. Recently, such regularity was studied in \cites{CM,HW} and \cites{DF1,DF2} through the use of Malliavin Calculus and a parametrix method respectively. However, the first two results don't extend completely to cover the derivatives that we need since they are formed specifically for functions of the form $G_\phi(\mu)=\langle \mu,\phi\rangle$, and thus some extra work would be needed to use these for the required bounds on mixed derivatives in time and space of $\tilde{u}$. In addition, in the latter two, the authors are only interested in proving enough regularity in order to establish strong solutions of the Cauchy problem on Wasserstein Space (see Equation \eqref{eq:W2CauchyProblem} below) and to apply it to establish rates for the Propagation of Chaos for systems without multiscale structure (in other words they only need up to 2 derivatives in space, where we need 4). We should also mention \cite{Tse}, where such regularity is studied via an alternative method using the Linear Functional Derivative, and the associated McKean-Vlasov SDE has additive noise and is posed on the torus.

The other issues are related to the following remark regarding the ``flow'' property of solutions to McKean-Vlasov SDEs with respect to their initial conditions:
\begin{remark}\label{remark:McKeanVlasovFlowProperty}
It is important that we parameterize the initial condition in Equations \eqref{eq:slow-fastMcKeanVlasov} and \eqref{eq:averagedMcKeanVlasov} by $\nu\in\mc{P}_2(\R^d)$ rather than $x\in \R^d$, since due to the dependence of the coefficients on their law, McKean-Vlasov SDEs do not enjoy the same flow property on the reals as standard SDEs. In other words, while for $\tilde{X}^{\epsilon,x}$ corresponding to Equation \eqref{eq:slow-fastMcKeanVlasov} with all coefficients independent of $\mu$ (that is, when Equation \eqref{eq:slow-fastMcKeanVlasov} is a standard SDE and not a McKean-Vlasov SDE)  and deterministic initial condition $x\in\R^d$, $\tilde{X}^{\epsilon, y}_t\bigg|_{y=\tilde{X
}^{\epsilon,x}_s} = \tilde{X}^{\epsilon,x}_{s+t}$, in general $X^{\epsilon,\delta_y}_t\bigg|_{y=X^{\epsilon,\delta_x}_s} \neq X^{\epsilon,\delta_x}_{s+t}$, where here we mean equality in distribution. We do, however, have via strong existence and uniqueness (see Remark \ref{remark:ontheassumptions1}) that $X^{\epsilon, \bar{\nu}}_t\bigg|_{\bar{\nu}=\mc{L}(X^{\epsilon,\nu}_s)}= X^{\epsilon,\nu}_{s+t}$ for all $\nu\in\mc{P}_2(\R^d)$.  In other words, the flow property for solutions of the McKean-Vlasov SDE holds viewed as a flow of measures, not as a flow on $\R^d$.

To illustrate why one cannot hope to have a flow property on $\R^d$, we note that for a random variable $\eta\sim\nu$, $X^{\epsilon,\delta_x}_t\bigg|_{x=\eta}\neq X^{\epsilon,\nu}_t$. This becomes evident by setting all coefficients equal to $0$ but $c$, and letting $c(x,y,\mu)=c(\mu) = \langle \mu,\phi\rangle$ for some deterministic, non-constant $\phi:\R^d\tto \R^d$. Then we no longer have $\epsilon$ dependence, and
\begin{align*}
X^{\delta_x}_t &= x + \int_0^t \langle \mc{L}(X^{\delta_x}_s),\phi\rangle ds\\
& = x + \int_0^t \langle \delta_{X^{\delta_x}_s},\phi\rangle ds\\
& = x + \int_0^t \phi(X^{\delta_x}_s)ds
\end{align*}
is deterministic, with
\begin{align*}
X^{\delta_x}_t\bigg|_{x=\eta} = \eta + \int_0^t \phi(X^{\delta_x}_s)\bigg|_{x=\eta}ds,
\end{align*}
while
\begin{align*}
X^{\nu}_t & = \eta + \int_0^t \langle \mc{L}(X^{\nu}_s),\phi\rangle ds\\
& = \eta + \int_0^t \E[\phi(X^{\nu}_s)]ds,
\end{align*}
so
\begin{align*}
X^{\delta_x}_t\bigg|_{x=\eta} = X^{\nu}_t \Leftrightarrow \int_0^t \phi(X^{\delta_x}_s)\bigg|_{x=\eta}ds = \int_0^t \E[\phi(X^{\nu}_s)]ds,
\end{align*}
the left of which is a random variable and the right of which is deterministic. Thus, this is requiring that $\eta$ be deterministic to hold in general. Thus, via the aforementioned flow property on measures, $X^{\epsilon,\delta_x}_{s+t} = X^{\epsilon, \nu}_t\bigg|_{\nu=\mc{L}(X^{\epsilon,\delta_x}_s)} \neq X^{\epsilon, \delta_y}_t\bigg|_{y=X^{\epsilon,\delta_x}_s} $. See \cite{Wang} p.3-4 for a good further exposition on this, and how it relates to the non-linearity of the generator of McKean-Vlasov SDEs.
\end{remark}

This leads us to the next major disadvantage of using the standard Cauchy Problem over the Cauchy Problem on Wasserstein space for studying convergence of McKean-Vlasov SDEs. Where in the standard averaging case, we have for $\xi\sim \nu$,
\begin{align*}
\biggl|\E\biggl[\phi(\tilde{X}^{\epsilon,\xi}_t)- \phi(\tilde{X}^\xi_t) \biggr]\biggr|& =\biggl|\int_{\R^d}\E\biggl[\phi(\tilde{X}^{\epsilon,x}_t)- \phi(\tilde{X}^x_t) \biggr]\nu(dx)\biggr|\\
&\leq  \int_{\R^d}\biggl|\E\biggl[\phi(\tilde{X}^{\epsilon,x}_t)- \phi(\tilde{X}^x_t) \biggr]\biggr|\nu(dx)\\
&\leq \int_{\R^d}C\epsilon |\phi|_{C^{4}_b(\R^d)}\nu(dx) \\
&= C\epsilon |\phi|_{C^{4}_b(\R^d)},
\end{align*}
so the result established via the standard Cauchy Problem for deterministic initial conditions $x\in\R^d$ also holds in the case of random initial conditions. Meanwhile, for the McKean-Vlasov setting, due to the non-linearity of the Fokker-Plank equations associated to $X^{\epsilon,\nu}$ and $X^\nu$ (see p.4 in \cite{Wang}):
\begin{align*}
\E\biggl[\phi(X^{\epsilon,\nu}_t)- \phi(X^{\nu}_t) \biggr]\neq \int_{\R^d} \E\biggl[\phi(X^{\epsilon,\delta_{x}}_t)- \phi(X^{\delta_{x}}_t) \biggr]\nu(dx)
\end{align*}
in general, so the result does not have the same upshot, and a proof of convergence of $\mc{L}(X^{\epsilon,\delta_{x}}_t)$ to $\mc{L}(X^{\delta_{x}}_t)$ for each $x\in\R^d$, which is what would be obtained via the standard Cauchy problem, does not immediately lift to the case of random initial conditions.
Lastly, even assuming sufficient regularity on the coefficients to apply It\^o's formula to $\tilde{u}(t,X^{\epsilon,\delta_x}_t)$, we would end up needing to control terms of the form $[\bar{\gamma}(X^{\epsilon,\delta_x}_t,\mc{L}(X^{\delta_{y}}_t))|_{y=X^{\epsilon,\delta_x}_t}-\gamma(X^{\epsilon,\delta_x}_t,Y^{\epsilon,\delta_x}_t,\mc{L}(X^{\epsilon,\delta_x}_t))]\cdot\partial_x\tilde{u}(t,X^{\epsilon,\delta_x}_t)$ (see the last equation on p.1230 of \cite{RocknerFullyCoupled}) rather than the term \\$[\bar{\gamma}(X^{\epsilon,\delta_x}_t,\mc{L}(X^{\epsilon,\delta_x}_t))-\gamma(X^{\epsilon,\delta_x}_t,Y^{\epsilon,\delta_x}_t,\mc{L}(X^{\epsilon,\delta_x}_t))]\cdot \partial_\mu U(t,\mc{L}(X^{\epsilon,\delta_x}_t))[X^{\epsilon,\delta_x}_t]$. The latter is what appears when using the Cauchy-Problem on Wasserstein space (see the term $R_4$ in the proof of Theorem \ref{theo:mckeanvlasovaveraging} below), and is such that  we can get vanishing bounds in expectation on its time integral via auxiliary Poisson equation methods (see Proposition \ref{prop:llntypefluctuationestimate2}). The former would require some kind of artificial Lipschitz estimates for $\bar{\gamma}(X^{\epsilon,\delta_x}_t,\mc{L}(X^{\epsilon,\delta_x}_t))-\bar{\gamma}(X^{\epsilon,\delta_x}_t,\mc{L}(X^{\delta_x}_t))$ (note the difference in the measure argument). This is do to an effective decoupling of the evolution of the measure which appears in the coefficients $\tilde{\gamma},\tilde{D}$ from the process whose generator is given by the right hand side of the equation which $\tilde{u}$ satisfies. In fact, this is the very decoupling used in \cite{BLPR} Equation (3.2) in order to define the real-valued component of the evolution of the Cauchy Problem on Wasserstein space in the full setting.

These issues reflect that if we were going to attempt to use $\tilde{u}$ as above, we would be attempting to re-frame the McKean-Vlasov SDE \eqref{eq:averagedMcKeanVlasov} in a way so that we are considering it as a flow on $\R^d$, rather than $\mc{P}_2(\R^d)$. Using instead the Cauchy Problem on Wasserstein space, we are ``lifting'' our view of the process to a deterministic evolution of a measure on $\mc{P}_2(\R^d)$, that is, its Law. This is reflected by the fact that in Equation \eqref{eq:W2CauchyProblem}, we are using the infinitesimal generator of the strongly continuous semi-group $\mc{P}_t,t\geq 0$ on the space of uniformly bounded, uniformly continuous functions $G:\mc{P}_2(\R^d)\tto \R$ such that $[\mc{P}_tG](\mu) = G(\mc{L}(X^\mu_t))$, rather than the infinitesimal generator associated to the nonlinear semi-group $T_t,t\geq 0$ acting on functions $\phi\in C^2_b(\R^d)$ by $[T_t\phi](x)=\E[\phi(X^{\delta_x}_t)]$, as is the case for $\tilde{u}$ (see Section 5.7.4 in \cite{CD}). This viewpoint is also what allows us to consider possibly nonlinear test functions $G$, where if we were to use $\tilde{u}$, we would be restricting ourselves to linear interactions with the measure.

Other than the modification of using the Cauchy Problem on Wasserstein space, the main idea of the proof of Theorem \ref{theo:mckeanvlasovaveraging} is similar to that of \cite{RocknerFullyCoupled} Theorem 2.3. That is, we can re-express the distance between any sufficiently regular test function $G$ evaluated at the law of $X^{\epsilon,\nu}_\tau$ and the law of $X^\nu_\tau$ for a given time $\tau$ as the distance between the time-reversed solution $\tilde{U}(t,\mu;\tau)\coloneqq U(\tau-t,\mu)$ to Equation \eqref{eq:W2CauchyProblem} evaluated at $(t,\mu)=(\tau,\mc{L}(X^{\epsilon,\nu}_\tau))$ and $(0,\nu)$ respectively. Then we can apply It\^o's formula and use the Propositions of Section \ref{section:mckeanvlasovergodictheorems} to conclude this distance is $O(\epsilon)$, and to express it linearly in terms of Lions derivatives of $U$, and hence in terms of the test function $G$. Intuitively what is going on with this approach to rates of convergence in distribution is, via the representation for $U(t,\mu)$ provided in Lemma \ref{lem:cauchyproblemregularity}, for any fixed $\tau\in [0,T]$, $t\mapsto \tilde{U}(t,\mc{L}(X^\nu_t);\tau)$ is constant and equal to $U(\tau-t,\mc{L}(X^\nu_t))=G(\mc{L}(X^{\mc{L}(X^\nu_t)}_{\tau-t})) = G(\mc{L}(X^\nu_\tau)) = U(\tau,\nu)=\tilde{U}(0,\nu;\tau)$ for any $\nu\in\mc{P}_2(\R^d)$ and $t\in [0,\tau]$. Thus, the change of $t\mapsto\tilde{U}(t,\mc{L}(X^{\epsilon,\nu}_t);\tau)$ from its initial value $\tilde{U}(0,\nu;\tau)$ as $t$ varies from $0$ to $\tau$ is in a sense is measuring the distance of $\mc{L}(X^{\epsilon,\nu}_\tau)$ from $\mc{L}(X^\nu_\tau)$.
\section{Proof of Theorem \ref{theo:mckeanvlasovaveraging}}\label{section:proofofmainresults}
\begin{proof}
We use the Cauchy problem on Wasserstein Space defined in Equation \eqref{eq:W2CauchyProblem}.

Fix $G\in \mc{M}_{b,L}^{\dot{\bm{\zeta}}}(\mc{P}_2(\R^d);\R)$, $\tau \in (0,T]$ and consider the solution $U:[0,\tau]\times\mc{P}_2(\R^d)\tto \R$ to \eqref{eq:W2CauchyProblem}, with this choice of $G$.
Then, by Lemma \ref{lem:cauchyproblemregularity}, we have that $U$ is given uniquely by $U(t,\mu) = G(\mc{L}(X_t^\mu))$, where we recall here the superscript $\mu$ is denoting that $X^\mu_t$ is initialized with a random variable $\xi$ independent from $W^2$ such that $\mc{L}(\xi)=\mu$, and that $U\in\mc{M}_{b,L}^{\tilde{\bm{\zeta}}}([0,\tau]\times\mc{P}_2(\R^d);\R)$. Then, letting $\tilde{U}(t,\mu) = U(\tau-t,\mu)$ for $t\in [0,\tau]$, we have $\tilde{U}(\tau,\mu) = U(0,\mu)=G(\mu)$. Thus, for any initial distribution $\nu\in\mc{P}_2(\R^d)$:
\begin{align*}
\biggl|G(\mc{L}(X^{\epsilon,\nu}_\tau))- G(\mc{L}(X^\nu_\tau)) \biggr|& = \biggl|G(\mc{L}(X^{\epsilon,\nu}_\tau)) - U(\tau,\nu) \biggr|\\
& = \biggl| \tilde{U}(\tau,\mc{L}(X^{\epsilon,\nu}_\tau))- \tilde{U}(0,\nu)\biggr|.
\end{align*}
Note that $\tilde{U}$ is parameterized by the end time of the interval, $\tau$, but we suppress this in the notation for presentation purposes.

Now we can apply the Chain Rule for Measure Dependent Functions from Proposition 5.102 in \cite{CD} to express the above in terms of Lions derivatives of $\tilde{U}$ in a way that does not depend on $X^{\nu}$ or its Law. We get, using $\dot{\tilde{U}}(s,\nu) = -\dot{U}(\tau-s,\nu)$ and that $U$ satisfies \eqref{eq:W2CauchyProblem}:
\begin{align*}
\biggl|G(\mc{L}(X^{\epsilon,\nu}_\tau)) - G(\mc{L}(X^\nu_\tau)) \biggr|& = \biggl|\int_0^\tau \dot{\tilde{U}}(s,\mc{L}(X^{\epsilon,\nu}_s))\\
&+\E\biggl[\biggl[\frac{1}{\epsilon}b(X^{\epsilon,\nu}_s,Y^{\epsilon,\nu}_s,\mc{L}(X^{\epsilon,\nu}_s))+c(X^{\epsilon,\nu}_s,Y^{\epsilon,\nu}_s,\mc{L}(X^{\epsilon,\nu}_s))\biggr]\cdot \partial_\mu \tilde{U}(s,\mc{L}(X^{\epsilon,\nu}_s))[X^{\epsilon,\nu}_s] \\
&+ \frac{1}{2}\sigma\sigma^\top(X^{\epsilon,\nu}_s,Y^{\epsilon,\nu}_s,\mc{L}(X^{\epsilon,\nu}_s)):\partial_z\partial_\mu \tilde{U}(s,\mc{L}(X^{\epsilon,\nu}_s))[X^{\epsilon,\nu}_s]\biggr]ds \biggr|\\
& = \biggl|\int_0^\tau -\E\biggl[\bar{\gamma}(X^{\epsilon,\nu}_s,\mc{L}(X^{\epsilon,\nu}_s))\cdot\partial_\mu \tilde{U}(s,\mc{L}(X^{\epsilon,\nu}_s))[X^{\epsilon,\nu}_s] \\
&+ \bar{D}(X^{\epsilon,\nu}_s,\mc{L}(X^{\epsilon,\nu}_s)):\partial_z\partial_\mu \tilde{U}(s,\mc{L}(X^{\epsilon,\nu}_s))[X^{\epsilon,\nu}_s]\biggr]\\
&+\E\biggl[\biggl[\frac{1}{\epsilon}b(X^{\epsilon,\nu}_s,Y^{\epsilon,\nu}_s,\mc{L}(X^{\epsilon,\nu}_s))+c(X^{\epsilon,\nu}_s,Y^{\epsilon,\nu}_s,\mc{L}(X^{\epsilon,\nu}_s))\biggr]\cdot\partial_\mu \tilde{U}(s,\mc{L}(X^{\epsilon,\nu}_s))[X^{\epsilon,\nu}_s] \\
&+ \frac{1}{2}\sigma\sigma^\top(X^{\epsilon,\nu}_s,Y^{\epsilon,\nu}_s,\mc{L}(X^{\epsilon,\nu}_s)):\partial_z\partial_\mu \tilde{U}(s,\mc{L}(X^{\epsilon,\nu}_s))[X^{\epsilon,\nu}_s]\biggr]ds \biggr|\\
& = \biggl|R_1(\tau) + R_2(\tau)+R_3(\tau)+R_4(\tau)+R_5(\tau)\biggr|
\end{align*}
where
\begin{align*}
R_{1}(\tau)& = \int_0^\tau \E\biggl[\frac{1}{\epsilon}b(X^{\epsilon,\nu}_s,Y^{\epsilon,\nu}_s,\mc{L}(X^{\epsilon,\nu}_s))\cdot\partial_\mu \tilde{U}(s,\mc{L}(X^{\epsilon,\nu}_s))[X^{\epsilon,\nu}_s] -  \gamma_1(X^{\epsilon,\nu}_s,Y^{\epsilon,\nu}_s,\mc{L}(X^{\epsilon,\nu}_s))\cdot\partial_\mu \tilde{U}(s,\mc{L}(X^{\epsilon,\nu}_s))[X^{\epsilon,\nu}_s]\\
&-2D_1(X^{\epsilon,\nu}_s,Y^{\epsilon,\nu}_s,\mc{L}(X^{\epsilon,\nu}_s)):\partial_z\partial_\mu \tilde{U}(s,\mc{L}(X^{\epsilon,\nu}_s))[X^{\epsilon,\nu}_s]\\
&- \int_{\R^{2d}}(\partial_\mu \tilde{U})^\top(s,\mc{L}(X^{\epsilon,\nu}_s))[X^{\epsilon,\nu}_s]\partial_\mu \Phi(X^{\epsilon,\nu}_s,Y^{\epsilon,\nu}_s,\mc{L}(X^{\epsilon,\nu}_s))[x]b(x,y,\mc{L}(X^{\epsilon,\nu}_s))\mc{L}(X^{\epsilon,\nu}_s,Y^{\epsilon,\nu}_s)(dx,dy)\\
&- \int_{\R^{2d}}\Phi^\top(X^{\epsilon,\nu}_s,Y^{\epsilon,\nu}_s,\mc{L}(X^{\epsilon,\nu}_s))\partial^2_\mu \tilde{U}(s,\mc{L}(X^{\epsilon,\nu}_s))[X^{\epsilon,\nu}_s,x]b(x,y,\mc{L}(X^{\epsilon,\nu}_s))\mc{L}(X^{\epsilon,\nu}_s,Y^{\epsilon,\nu}_s)(dx,dy)\biggr]ds\\
R_{2}(\tau)& = \int_0^\tau \E\biggl[\int_{\R^{2d}}(\partial_\mu \tilde{U})^\top(s,\mc{L}(X^{\epsilon,\nu}_s))[X^{\epsilon,\nu}_s]\partial_\mu \Phi(X^{\epsilon,\nu}_s,Y^{\epsilon,\nu}_s,\mc{L}(X^{\epsilon,\nu}_s))[x]b(x,y,\mc{L}(X^{\epsilon,\nu}_s))\mc{L}(X^{\epsilon,\nu}_s,Y^{\epsilon,\nu}_s)(dx,dy) \biggr]ds\\
R_3(\tau)& = \int_0^\tau \E\biggl[\int_{\R^{2d}}\Phi^\top(X^{\epsilon,\nu}_s,Y^{\epsilon,\nu}_s,\mc{L}(X^{\epsilon,\nu}_s))\partial^2_\mu \tilde{U}(s,\mc{L}(X^{\epsilon,\nu}_s))[X^{\epsilon,\nu}_s,x]b(x,y,\mc{L}(X^{\epsilon,\nu}_s))\mc{L}(X^{\epsilon,\nu}_s,Y^{\epsilon,\nu}_s)(dx,dy) \biggr]ds\\
R_{4}(\tau) & = \int_0^\tau \E\biggl[\gamma(X^{\epsilon,\nu}_s,Y^{\epsilon,\nu}_s,\mc{L}(X^{\epsilon,\nu}_s))\cdot\partial_\mu \tilde{U}(s,\mc{L}(X^{\epsilon,\nu}_s))[X^{\epsilon,\nu}_s] \\
&-\int_{\R^d}\gamma(X^{\epsilon,\nu}_s,y,\mc{L}(X^{\epsilon,\nu}_s))\pi(dy;X^{\epsilon,\nu}_s,\mc{L}(X^{\epsilon,\nu}_s))\cdot\partial_\mu \tilde{U}(s,\mc{L}(X^{\epsilon,\nu}_s))[X^{\epsilon,\nu}_s]\biggr]ds \\
R_{5}(\tau) & =  \int_0^\tau \E\biggl[D(X^{\epsilon,\nu}_s,Y^{\epsilon,\nu}_s,\mc{L}(X^{\epsilon,\nu}_s)):\partial_z\partial_\mu \tilde{U}(s,\mc{L}(X^{\epsilon,\nu}_s))[X^{\epsilon,\nu}_s] \\
&- \int_{\R^d}D(X^{\epsilon,\nu}_s,y,\mc{L}(X^{\epsilon,\nu}_s))\pi(dy;X^{\epsilon,\nu}_s,\mc{L}(X_s)):\partial_z\partial_\mu \tilde{U}(s,\mc{L}(X^{\epsilon,\nu}_s))[X^{\epsilon,\nu}_s]\biggr]ds.
\end{align*}
Here the martingale terms from It\^o's formula vanish in expectation thanks to Lemma \ref{lemma:barYuniformbound}, and we use that $\partial_z\partial_\mu \tilde{U}(s,\mu)[z]$ is a symmetric matrix for any $\mu\in\mc{P}_2(\R^d), s\in [0,\tau],$ and $z\in\R^d$ (see \cite{CD} Corollary 5.89) to write:
\begin{align*}
\frac{1}{2}\sigma\sigma^\top(x,y,\mu): \partial_z\partial_\mu \tilde{U}(s,\mu)[z]&=D(x,y,\mu): \partial_z\partial_\mu \tilde{U}(s,\mu)[z]-[D_1(x,y,\mu)+D_1^\top(x,y,\mu)]:\partial_z\partial_\mu \tilde{U}(s,\mu)[z]\\
& = D(x,y,\mu): \partial_z\partial_\mu \tilde{U}(s,\mu)[z]-2D_1(x,y,\mu):\partial_z\partial_\mu \tilde{U}(s,\mu)[z].
\end{align*}

Applying Proposition \ref{prop:fluctuationestimateparticles2} with $\psi(s,x,\mu) = \partial_\mu \tilde{U}(s,\mu)[x]$,  we get, letting $\tilde{\bm{\zeta}}$ be as in the statement of that Proposition:
\begin{align*}
|R_1(\tau)|&\leq C\epsilon[1+\tau]\norm{(t,x,\mu)\mapsto \partial_\mu \tilde{U}(t,\mu)[x]}_{\mc{M}_b^{\tilde{\bm{\zeta}}}([0,\tau]\times\R^d\times\mc{P}_2(\R^d);\R^d)}\\
&\leq C\epsilon[1+\tau] \biggl[\sup_{t\in[0,\tau]}\norm{U(t,\cdot)}_{\mc{M}_b^{\dot{\bm{\zeta}}}(\mc{P}_2(\R^d);\R)}+\sup_{t\in [0,\tau],z\in \R^d,\mu\in\mc{P}_2(\R^d)}|\partial_\mu \dot{U}(t,\mu)[z]|\biggr]\\
&\leq \epsilon C(T) \norm{G}_{\mc{M}_b^{\dot{\bm{\zeta}}}(\mc{P}_2(\R^d);\R)}
\end{align*}
by Lemma \ref{lem:cauchyproblemregularity}.
Applying Proposition \ref{prop:purpleterm2} with $\psi(s,x,\mu) = \partial_\mu \tilde{U}(s,\mu)[x]$, we get in the same way:
\begin{align*}
|R_2(\tau)|\leq \epsilon C(T) \norm{G}_{\mc{M}_b^{\dot{\bm{\zeta}}}(\mc{P}_2(\R^d);\R)}.
\end{align*}

Applying Proposition \ref{prop:newpurpleterm} with $\psi(s,x,\mu) = \partial_\mu \tilde{U}(s,\mu)[x]$, letting $\tilde{\bm{\zeta}}_2$ be as in the statement of that Proposition, we get:
\begin{align*}
|R_3(\tau)|&\leq C\epsilon[1+\tau]\biggl[\sup_{t\in [0,\tau]}\norm{(x,\mu)\mapsto \partial_\mu \tilde{U}(t,\mu)[x]}_{\mc{M}_b^{\tilde{\bm{\zeta}}_2}(\R^d\times\mc{P}_2(\R^d);\R^d)}+\sup_{t\in [0,\tau],z_1,z_2\in \R^d,\mu\in\mc{P}_2(\R^d)}|\partial^2_\mu \dot{U}(t,\mu)[z_1,z_2]|\biggr]\\
&\leq C\epsilon[1+\tau] \biggl[\sup_{t\in[0,\tau]}\norm{U(t,\cdot)}_{\mc{M}_b^{\dot{\bm{\zeta}}}(\mc{P}_2(\R^d);\R)}+\sup_{t\in [0,\tau],z_1,z_2\in \R^d,\mu\in\mc{P}_2(\R^d)}|\partial^2_\mu \dot{U}(t,\mu)[z_1,z_2]|\biggr]\\
&\leq \epsilon C(T) \norm{G}_{\mc{M}_b^{\dot{\bm{\zeta}}}(\mc{P}_2(\R^d);\R)}
\end{align*}
by Lemma \ref{lem:cauchyproblemregularity}.

Applying Proposition \ref{prop:llntypefluctuationestimate2} with $\psi(s,x,\mu) = \partial_\mu \tilde{U}(s,\mu)[x]$, $F=\gamma$, and $k=d$, we get in the same way as for $R_1(\tau),R_2(\tau)$ that:
\begin{align*}
|R_4(\tau)|\leq \epsilon C(T) \norm{G}_{\mc{M}_b^{\dot{\bm{\zeta}}}(\mc{P}_2(\R^d);\R)}.
\end{align*}

Lastly, Proposition \ref{prop:llntypefluctuationestimate2} with $\psi(s,x,\mu) = \partial_z\partial_\mu \tilde{U}(s,\mu)[x]$, $F=D$, and $k=d\times d$, we get:
\begin{align*}
|R_5(\tau)|&\leq C\epsilon[1+\tau]\norm{(t,x,\mu)\mapsto \partial_z\partial_\mu \tilde{U}(t,\mu)[x]}_{\mc{M}_b^{\tilde{\bm{\zeta}}}([0,\tau]\times\R^d\times\mc{P}_2(\R^d);\R^{d\times d})}\\
&\leq C\epsilon[1+\tau] \biggl[\sup_{t\in[0,\tau]}\norm{U(t,\cdot)}_{\mc{M}_b^{\dot{\bm{\zeta}}}(\mc{P}_2(\R^d);\R)}+\sup_{t\in [0,\tau],z\in \R^d,\mu\in\mc{P}_2(\R^d)}|\partial_z\partial_\mu \dot{U}(t,\mu)[z]|\biggr]\\
&\leq \epsilon C(T) \norm{G}_{\mc{M}_b^{\dot{\bm{\zeta}}}(\mc{P}_2(\R^d);\R)}
\end{align*}
by Lemma \ref{lem:cauchyproblemregularity}.
Then, noting that these bounds are all uniform in $\tau\in [0,T]$, we have
\begin{align*}
\sup_{s\in [0,T]}\biggl|G(\mc{L}(X^{\epsilon,\nu}_s)) - G(\mc{L}(X^\nu_s)) \biggr|&\leq \epsilon C(T)\norm{G}_{\mc{M}_b^{\dot{\bm{\zeta}}}(\mc{P}_2(\R^d);\R)},
\end{align*}
as desired.

\end{proof}

\section{Conclusions and Future Work}\label{S:ConclusionsFutureWork}
In this paper we have derived an averaging principle for fully coupled McKean-Vlasov SDEs, along with an associated rate of weak convergence.

In this paper, due to limitations in the literature regarding regularity of solutions to the Cauchy Problem on Wasserstein space, we made strong assumptions on the regularity of the coefficients of Equation \eqref{eq:slow-fastMcKeanVlasov}. It is known that under weaker assumptions, similar rates of convergence can be derived in case of fully coupled standard SDEs (without coefficients which have explicit dependence on the law of the process). See Remarks \ref{remark:ontheassumptions2} and \ref{remark:nomeasuredependence} for further discussion of this. An interesting avenue of future research would be to see if the method proposed in this paper can be extended to weaker assumptions via improving the existing regularity results for PDEs of the type \eqref{eq:W2CauchyProblem}.

Another interesting extension would be to establish an averaging principle for fully-coupled SDEs in the setting where the coefficients the fast and slow process in Equation \eqref{eq:slow-fastMcKeanVlasov} depend on the law of the fast process, $\mc{L}(Y^{\epsilon,\nu}_t)$. In this setting, the solution of the Poisson Equation $\Phi$ whose derivatives appear in the coefficients of limiting equation \eqref{eq:averagedMcKeanVlasov} will have to solve a PDE on Wasserstein Space, since the generator obtained from considering the $O(1/\epsilon^2)$ terms from the generator of $(X^{\epsilon,\nu}_t,Y^{\epsilon,\nu}_t)$ and ``freezing'' the terms associated to the slow process will be that of a McKean-Vlasov SDE. See \cites{delgadino2020,XLLM,LWX} for related results in this direction.

\appendix
\section{Regularity of the Poisson Equations}\label{sec:regularityofthecellproblem}
Throughout this subsection we assume \ref{assumption:uniformellipticity} and \ref{assumption:retractiontomean}, and prove the needed regularity of the Poisson Equations \eqref{eq:cellproblemold},\eqref{eq:doublecorrectorproblem}, \eqref{eq:tildechi}, and \eqref{eq:driftcorrectorproblem} and the averaged coefficients from \eqref{eq:averagedlimitingcoefficients} in order for the results in Sections \ref{section:mckeanvlasovergodictheorems}, \ref{sec:onthecauchyproblem}, and \ref{section:proofofmainresults} to go through. The final result containing this needed regularity is Proposition \ref{proposition:allneededregularity}.

The proofs here are analogous to those found in \cite{RocknerFullyCoupled}. Thus, for brevity, we choose only to sketch the proofs and focus on the necessary additional steps which arise from the fact that we take derivatives in $\mc{P}_2(\R^d)$ in addition to standard spatial derivatives.\begin{lemma}\label{lemma:Ganguly1DCellProblemResult}
 Consider $B:\R^d\times\R^d\times\bb{P}_2(\R^d)\tto \R^k$ such that
\begin{align*}
\int_{\R^d}B(x,y,\mu)\pi(dy;x,\mu)=0,\forall x\in \R^d,\mu\in \mc{P}_2(\R^d),
\end{align*}
$B$ grows at most polynomially in $y$ uniformly in  $x,\mu$ as $|y|\tto\infty$, and $B$ is locally H\"older continuous in $y$ uniformly in $(x,\mu)$.
Then there exists a unique solution $u:\R^d\times\R^d\times\mc{P}_2(\R^d)\tto \R^k$ to
\begin{align*}
\mc{L}_{x,\mu}u_l(x,y,\mu)=B_l(x,y,\mu),l\in\br{1,...,k}
\end{align*}
$\int_{\R^d}u(x,y,\mu)\pi(dy;x,\mu)=0$, and $u$, $\partial_y u$, $\partial^2_y u$ are locally H\"older continuous in $y$ uniformly in $x,\mu$ and have at most polynomial growth as $|y|\tto \infty$.\end{lemma}
\begin{proof}
This follows as in the proof of Theorem 2.1 i) in \cite{RocknerFullyCoupled} after noting that the estimates on the transition density associated to $\mc{L}_{x,\mu}$ collected in Lemma 3 extend to our situation where the $f$ has linear growth in $y$ and $\mc{L}_{x,\mu}^*$ is dissipative via the transformation argument provided in \cite{EidelmanBook} Chapter 1 Section 5.
\end{proof}

\begin{lemma}\label{lemma:derivativetransferformulas}
Consider $h:\R^d\times\R^d\times\mc{P}_2(\R^d)\tto \R$. Suppose $h$ in grows at most polynomially in $y$ and is locally H\"older continuous uniformly in $x,\mu$. Then
\begin{align}\label{eq:muLipschitztransferformula}
&\int_{\R^d}h(x,y,\mu_1)\pi(dy;x,\mu_1) -\int_{\R^d}h(x,y,\mu_2)\pi(dy;x,\mu_2)  = \nonumber\\ &\hspace{5cm}=\int_{\R^d} h(x,y,\mu_1)-h(x,y,\mu_2)-[\mc{L}_{x,\mu_1}-\mc{L}_{x,\mu_2}]v(x,y,\mu_2)\pi(dy;x,\mu_1)
\end{align}
for all $x\in \R^d,\mu_1,\mu_2\in \mc{P}_2(\R^d)$,
and
\begin{align}\label{eq:xLipschitztransferformula}
&\int_{\R^d}h(x_1,y,\mu)\pi(dy;x_1,\mu) -\int_{\R^d}h(x_2,y,\mu)\pi(dy;x_2,\mu)  = \nonumber\\
&\hspace{5cm} = \int_{\R^d} h(x_1,y,\mu)-h(x_2,y,\mu)-[\mc{L}_{x_1,\mu}-\mc{L}_{x_2,\mu}]v(x_2,y,\mu)\pi(dy;x_1,\mu)
\end{align}
for all $x_1,x_2\in \R^d,\mu\in \mc{P}_2(\R^d)$.

In addition, consider $v$ solving
\begin{align}\label{eq:poissoneqfortransferformulas}
\mc{L}_{x,\mu}v(x,y,\mu)& = h(x,y,\mu)- \int_{\R^d} h(x,\bar{y},\mu)\pi(d\bar{y};x,\mu),
\end{align}
and $\mc{L}^{(k,j,\bm{\alpha}(\bm{p}_k))}_{x,\mu}[z_{\bm{p}_k}]$ is the differential operator acting on $\phi \in C^2_b(\R^d)$ by
\begin{align*}
\mc{L}^{(k,j,\bm{\alpha}(\bm{p}_k))}_{x,\mu}[z_{\bm{p}_k}]\phi(y)=D^{(k,j,\bm{\alpha}(\bm{p}_k))}f(x,y,\mu)[z_{\bm{p}_k}]\cdot\nabla\phi(y)+D^{(k,j,\bm{\alpha}(\bm{p}_k))}a(x,y,\mu)[z_{\bm{p}_k}]:\nabla^2\phi(y),
\end{align*}
where the inner products are taken in each $d$-dimensional component of the derivative matrices of $f$ and each $d\times d$-dimensional component of the derivative matrices of $a$.

Suppose that for some complete collection of multi-indices $\bm{\zeta}$ that $h\in \mc{M}_{p}^{\bm{\zeta}}(\R^d\times\R^d\times \mc{P}_2(\R^d);\R),f\in \mc{M}_{p}^{\bm{\zeta}}(\R^d\times\R^d\times \mc{P}_2(\R^d);\R^d),a\in \mc{M}_{p}^{\bm{\zeta}}(\R^d\times\R^d\times \mc{P}_2(\R^d);\R^{d\times d})$, and that $\partial_yv\in \mc{M}_{p}^{\bm{\zeta}'}(\R^d\times\R^d\times \mc{P}_2(\R^d);\R^d),\partial^2_yv\in \mc{M}_{p}^{\bm{\zeta}'}(\R^d\times\R^d\times \mc{P}_2(\R^d);\R^{d\times d})$, where $\bm{\zeta}'$ is obtained from removing any multi-indices which contain the maximal first and second values from $\bm{\zeta}$.Then for any multi-index $(n,l,\bm{\beta})\in\bm{\zeta}$:
\begin{align}\label{eq:derivativetransferformula}
&D^{(n,l,\bm{\beta})}\int_{\R^d} h(x,y,\mu)\pi(dy;x,\mu)[z_1,...,z_n]
 = \int_{\R^d}\left(D^{(n,l,\bm{\beta})}h(x,y,\mu)[z_1,...,z_n] - \right.\\
 &\hspace{2cm}\left.-\sum_{k=0}^n\sum_{j=0}^l\sum_{\bm{p}_k} C_{(\bm{p}_k,j,n,l)} \mc{L}^{(k,j,\bm{\alpha}(\bm{p}_k))}_{x,\mu}[z_{\bm{p}_k}] D^{(n-k,l-j,\bm{\alpha}(\bm{p}'_{n-k}))}v(x,y,\mu)[z_{\bm{p}'_{n-k}}]\right) \pi(dy;x,\mu)\nonumber
\end{align}
where here $\bm{p}_k\in \binom{\br{1,...,n}}{k}$ with $\bm{p}'_{n-k} = \br{1,...,n} \setminus \bm{p}_k$, for $\bm{p}_k = \br{p_1,...,p_k}$, the argument $[z_{\bm{p}_k}]$ denotes $[z_{p_1},...,z_{p_k}]$, and $\bm{\alpha}(\bm{p}_k)\in \bb{N}^k$ is determined by $\bm{\beta}=(\beta_1,...,\beta_n)$ by $\bm{\alpha}(\bm{p}_k) = (\alpha_1,...,\alpha_k)$, $\alpha_j = \beta_{p_j},j\in\br{1,...,k}$, and similarly for $\bm{\alpha}(\bm{p}'_{n-k})$. Also here $C_{(\bm{p}_0,0,n,l)}=0$, and $C_{(\bm{p}_k,j,n,l)}>0,C_{(\bm{p}_k,j,n,l)}\in\bb{N}$ for $(k,j)\in \bb{N}^2, (k,j)\neq (0,0).$  See Remark \ref{Rem:C_constants} for an iterative way to define the constants  $C_{(p,j,n,\ell)}$.
\end{lemma}
\begin{proof}

Note that by the current assumptions, the integrand on the right hand side of \eqref{eq:derivativetransferformula} grows at most polynomially in $|y|$ uniformly in $x,\mu$, and hence is integrable against $\pi$. This kind of ``transfer formula'' on the regularity of derivatives in the coefficients of averaged functions against an invariant measure is the subject of Lemmas 3.2 and 3.7 in \cite{RocknerFullyCoupled} and \cite{RocknerSPDE} respectively.

The result follows almost directly from the aforementioned Lemmas, taking a bit of care to account for the fact that we are dealing with Lions derivatives, and that our assumptions are a bit different than those found in those papers.
Namely, in \cite{RocknerFullyCoupled}, boundedness of $f$ is assumed. Tracking where this assumption is being used in Lemma 3.2 i), we see they come up only when employing the Equation (28) from \cite{PV2}. In the proof of that result, however, the boundedness assumption on the coefficients is only needed in order to obtain the regularity for the transition density in Proposition 2 of \cite{PV2}. Upon inspecting the results in the PDE literature that they are using, one can see that in fact the same regularity holds in our regime, where $f$ itself has linear growth but $\partial_y f$ is bounded (under the additional assumption of local H\"older continuity of derivatives of the coefficients in $y$ and dissipativity). In particular, the parametrix method used in Chapter 9 of \cite{Friedman} to prove the bounds collected as Proposition 2 in \cite{PV2} can be extended to the case of dissipative parabolic PDEs with growth in their coefficients. See \cite{EidelmanBook} Chapter 1 Section 5 (this result is also partially stated in terms of the transition density for SDEs as Theorem 4.1 in \cite{PavliotisSPA}).

Similarly, in Lemma 3.2 ii), the boundedness of $f$ is only being used to appeal to the expression given for $\mc{K}_2$ in Lemma 4.1 in \cite{RocknerHolderContinuous}, wherein this boundedness assumption again is only used for the same estimates on the transition density (collected as Lemma 3.3 there), so again by the extension provided by \cite{EidelmanBook} Chapter 1 Section 5, the result applies in our setting as well. Note that the expression for $\mc{K}_2$ in that paper is coming from the same computation as \cite{PV2} Equations (25) and (26), except without dividing by $h$.

Thus, by this discussion, Equation \eqref{eq:derivativetransferformula} holds in the case that $n=0$ (i.e. only derivatives in $x$) and \eqref{eq:xLipschitztransferformula}, follow directly from Lemma 3.2 in \cite{RocknerFullyCoupled}.

Once we know that the result holds for $n=0,l=1$, so
\begin{align}\label{eq:1xderivativetransfer}
D^{(0,1,0)}\biggl[\int_{\R^d}h(x,y,\mu)\pi(dy;x,\mu) \biggr]& = \int_{\R^d}D^{(0,1,0)}h(x,y,\mu)-\mc{L}^{(0,1,0)}_{x,\mu}v(x,y,\mu)\pi(dy;x,\mu),
\end{align}
we can reapply the result for one derivative to the functions $p_k(x,y,\mu)=\biggl[D^{(0,1,0)}h(x,y,\mu)-\mc{L}^{(0,1,0)}_{x,\mu}v(x,y,\mu)\biggr]_k,$ $k=1,...,d$. The result applies since by assumption, each $p_k$ is jointly continuous and grows at most polynomially in $y$ uniformly in $x$ and $\mu$.

Then
\begin{align}\label{eq:pkrepresentation}
D^{(0,1,0)}\biggl[\int_{\R^d}p_k(x,y,\mu)\pi(dy;x,\mu) \biggr] &= \int_{\R^d}D^{(0,1,0)}p_k(x,y,\mu)-\mc{L}^{(0,1,0)}_{x,\mu}w_k(x,y,\mu)\pi(dy;x,\mu)
\end{align}
where
\begin{align*}
\mc{L}_{x,\mu}w_k(x,y,\mu) = p_k(x,y,\mu)-\int_{\R^d}p_k(x,y,\mu)\pi(dy;x,\mu),k\in\br{1,...,d}.
\end{align*}

In other words,
\begin{align*}
\mc{L}_{x,\mu}w_k(x,y,\mu) &= \partial_{x_k}h(x,y,\mu)-\partial_{x_k}f(x,y,\mu)\cdot \partial_yv(x,y,\mu)-\partial_{x_k}a(x,y,\mu):\partial^2_y v(x,y,\mu)\\
& - \int_{\R^d}\partial_{x_k}h(x,y,\mu)-\partial_{x_k}f(x,y,\mu)\cdot \partial_yv(x,y,\mu)-\partial_{x_k}a(x,y,\mu):\partial^2_y v(x,y,\mu)\pi(dy;x,\mu)\\
& = \partial_{x_k}h(x,y,\mu)-\partial_{x_k}f(x,y,\mu)\cdot \partial_yv(x,y,\mu)-\partial_{x_k}a(x,y,\mu):\partial^2_y v(x,y,\mu)\\
&-\partial_{x_k}\biggl[\int_{\R^d}h(x,y,\mu)\pi(dy;x,\mu) \biggr],
\end{align*}
where in the second equality we used the expression \eqref{eq:1xderivativetransfer}.

But, differentiating the expression which $v$ satisfies \eqref{eq:poissoneqfortransferformulas} in $x_k$, we have
\begin{align*}
\mc{L}_{x,\mu}\partial_{x_k}v &= \partial_{x_k}h(x,y,\mu) - \partial_{x_k}\biggl[\int_{\R^d}h(x,y,\mu)\pi(dy;x,\mu) \biggr]-\partial_{x_k}f(x,y,\mu)\cdot \partial_yv(x,y,\mu)\\
&-\partial_{x_k}a(x,y,\mu):\partial^2_y v(x,y,\mu),
\end{align*}
and by the uniqueness granted by Lemma \ref{lemma:Ganguly1DCellProblemResult}, we have in fact that $w_k=\partial_{x_k}v$. By rewriting the expression \eqref{eq:pkrepresentation} in terms of $h$ and $v$, it reads:
\begin{align*}
&D^{(0,2,0)}\biggl[\int_{\R^d}h(x,y,\mu)\pi(dy;x,\mu) \biggr] \nonumber\\ &\qquad=\int_{\R^d}D^{(0,2,0)}h(x,y,\mu)-\mc{L}^{(0,2,0)}_{x,\mu}v(x,y,\mu) - 2\mc{L}^{(0,1,0)}_{x,\mu}D^{(0,1,0)}v(x,y,\mu)\pi(dy;x,\mu).
\end{align*}

In addition, since we take no derivatives in the expression \eqref{eq:muLipschitztransferformula}, the proof follows in the exact same way as when the difference is taken in the parameters which live in Euclidean space. In fact, it also holds that if $\hat{h}:\R^d\times\R^d\times \R^d\times \mc{P}_2(\R^d)\tto \R$ satisfies the same polynomial growth and H\"older continuity assumptions in $y$, then
\begin{align}\label{eq:muandzLipschitztransferformula}
&\int_{\R^d}\hat{h}(x,y,z_1,\mu_1)\pi(dy;x,\mu_1) -\int_{\R^d}\hat{h}(x,y,z_2,\mu_2)\pi(dy;x,\mu_2)\\
& =  \int_{\R^d} \hat{h}(x,y,z_1,\mu_1)-\hat{h}(x,y,z_2,\mu_2)-[\mc{L}_{x,\mu_1}-\mc{L}_{x,\mu_2}]\hat v(x,y,z_2,\mu_2)\pi(dy;x,\mu_1)\nonumber
\end{align}
where
\begin{align*}
\mc{L}_{x,\mu}\hat v(x,y,z,\mu) = \hat{h}(x,y,z,\mu)-\int_{\R^d}\hat{h}(x,y,z,\mu)\pi(dy;x,\mu).
\end{align*}
As we will see, this is useful for dealing with the Lions derivatives when proving Equation \eqref{eq:derivativetransferformula}.

In order to arrive a the full expression in Equation \eqref{eq:derivativetransferformula}, we first examine the case of $n=1$, $l,\beta=0$. Then the result will follow in the iterative manner which we just outlined above.

Let $\tilde{h},\tilde{f},\tilde{a},\tilde{\pi}$ be the lifted functions on some $L^2(\tilde\W,\tilde\F,\tilde\Prob;\R^d)$ as per Definition \ref{def:lionderivative}. For $\phi :L^2(\tilde\W,\tilde\F,\tilde\Prob;\R^d)\tto \R^k$ Gateaux differentiable, we denote by $D[\phi(\tilde{X});\tilde{Y}) = \lim_{t\downarrow 0}[\phi(\tilde{X}+t\tilde{Y})-\phi(\tilde{X})]/t$ the Gateaux derivative of $\phi$ at $\tilde{X}$ in the direction $\tilde{Y}$.

Fixing $x,\tilde{X},\tilde{Y}$, we can define $H(t,y) = \tilde{h}(x,y,\tilde{X}+t\tilde{Y}),F(t,y) = \tilde{f}(x,y,\tilde{X}+t\tilde{Y}),A(t,y) = \tilde{a}(x,y,\tilde{X}+t\tilde{Y}),\Pi(dy;t) =\tilde{\pi}(x,y,\tilde{X}+t\tilde{Y})$. Then, since $x$ and $\mu$ are only parameters, we have for all $t\in \R$ $\Pi(\cdot;t)$ is the unique invariant measure associated to the linear operator with first order coefficients given by $F(t,y)$ and second order coefficients given by $A(t,y)$, and $\int_{\R^d} H(t,y)\Pi(dy;t)=0$. By definition of the Lions derivative, we also have the derivative of the coefficients with respect to $t$ have at most polynomial growth in $y$ uniformly in $t$ by assumption. Then by Lemma 3.2 1) in \cite{RocknerFullyCoupled} with $\eta=1$, we have
\begin{align*}
\partial_t\biggl[\int_{\R} H(t,y)\Pi(dy;t) \biggr]=\int_{\R^d} \partial_tH(t,y)-(\partial_tF(t,y))\cdot V(t,y)-(\partial_tA(t,y)): V(t,y)\Pi(dy;t)
\end{align*}
where
\begin{align*}
F(t,y)\cdot \partial_y V(t,y)+A(t,y):\partial^2_y V(t,y) = H(t,y) - \int_{\R^d}H(t,\bar{y})\Pi(\bar{y};t).
\end{align*}

Thus, evaluating at $t=0$, we have:

\begin{align*}
D[\int_{\R^d}\tilde{h}(x,y,\tilde{X})\tilde{\pi}(y;x,\tilde{X})dy;\tilde{Y}] & = \int_{\R^d} D[\tilde{h}(x,y,\tilde{X});\tilde{Y}]  - D[\tilde{f}(x,y,\tilde{X});\tilde{Y}] \cdot \partial_y\tilde{v}(x,y,\tilde{X})\\
&-D[\tilde{a}(x,y,\tilde{X});\tilde{Y}]:\partial^2_y \tilde{v}(x,y,\tilde{X}) \pi(dy;x,\tilde{X})
\end{align*}
for all $x,\in \R^d,\tilde{X},\tilde{Y}\in L^2(\tilde\W,\tilde\F,\tilde\Prob;\R^d)$. Here
\begin{align*}
\tilde{\mc{L}}_{x,\tilde{X}}\tilde{v}(x,y,\tilde{X})\coloneqq\tilde{f}(x,y,\tilde{X})\cdot\partial_y\tilde{v}(x,y,\tilde{X})+\tilde{a}(x,y,\tilde{X}):\partial^2_y\tilde{v}(x,y,\tilde{X})& = \tilde{h}(x,y,\tilde{X})-\int_{\R^d}\tilde{h}(x,y,\tilde{X})\pi(dy;x,\tilde{X}).
\end{align*}

Now we note that by definition of the liftings and the Lions derivative:
\begin{align*}
D\biggl[\int_{\R^d}\tilde{h}(x,y,\tilde{X})\tilde{\pi}(y;x,\tilde{X})dy;\tilde{Y} \biggr]& = \tilde{\E}\biggl[\int_{\R^d} D\tilde{h}(x,y,\tilde{X})- D\tilde{f}(x,y,\tilde{X})\cdot\partial_y v(x,y,\mu)\\
&-D\tilde{a}(x,y,\tilde{X}):\partial^2_y v(x,y,\mu)\pi(dy;x,\mu)\cdot\tilde{Y}   \biggr]\\
&=\tilde{\E}\biggl[\int_{\R^d} \partial_\mu h(x,y,\mu)[\tilde{X}]- \partial_\mu f(x,y,\mu)[\tilde{X}]\cdot\partial_y v(x,y,\mu)\\
&-\partial_\mu a(x,y,\mu)[\tilde{X}]:\partial^2_yv(x,y,\mu)\pi(dy;x,\mu)\cdot\tilde{Y}   \biggr],
\end{align*}
where $D\phi(\tilde{X})$ is the Fr\'echet derivative of $\phi:L^2(\tilde\W,\tilde\F,\tilde\Prob;\R^d)\tto \R^k$ at $\tilde{X}$, and we denote the expectation with respect to $\tilde{\Prob}$ by $\tilde{\E}$.

Now it is useful to note that by assumption, the integrand on the right hand side of the above display is locally H\"older continuous and has at most polynomial growth in $y$ uniformly in $x,\mu,z$. So we have, letting $\hat{h}(x,y,z,\mu) = \partial_\mu h(x,y,\mu)[z]- \partial_\mu f(x,y,\mu)[z]\cdot\partial_y v(x,y,\mu)-\partial_\mu a(x,y,\mu)[z]:\partial^2_y v(x,y,\mu)$, the expression  (\ref{eq:muandzLipschitztransferformula}) holds with this choice of $\hat{h}$. By Lemma \ref{lemma:Ganguly1DCellProblemResult}, $\partial_y\hat{v}$ and $\partial^2_y\hat{v}$ grow at most polynomially in $y$ uniformly in $x,\mu,z$ (it makes no difference to add another space parameter). Thus, by the assumptions on the coefficients:
\begin{align*}
&\biggl|\int_{\R^d}\hat{h}(x,y,z_1,\mu_1)\pi(dy;x,\mu_1) -\int_{\R^d}\hat{h}(x,y,z_2,\mu_2)\pi(dy;x,\mu_2)\biggr|\\
& \leq  \int_{\R^d} |\hat{h}(x,y,z_1,\mu_1)-\hat{h}(x,y,z_2,\mu_2)|+|\mc{L}_{x,\mu_1}-\mc{L}_{x,\mu_2}]\hat v(x,y,z_2,\mu_2)|\pi(dy;x,\mu_1)\nonumber\\
&\leq C \int_{\R^d} [\bb{W}_2(\mu_1,\mu_2) + |z_1-z_2|](1+|y|^p)\pi(dy;x,\mu_1),\text{ for some }p\in\bb{N}\\
&\leq C[\bb{W}_2(\mu_1,\mu_2) + |z_1-z_2|].
\end{align*}
Then, taking $\br{\tilde{X}^N}_{N\in\bb{N}} \subset L^2(\tilde\W,\tilde\F,\tilde\Prob;\R^d)$ such that $\tilde{X}^N\tto \tilde{X}$ and $\tilde{X}^N\sim \mu^N,\tilde{X}\sim \mu$, we have

\begin{align*}
&\lim_{N\toinf}\sup_{\tilde{Y}:\norm{\tilde{Y}}_{L^2(\tilde\W,\tilde\F,\tilde\Prob;\R^d)}\neq 0}\biggl(\tilde{\E}\biggl[\int_{\R^d} \partial_\mu h(x,y,\mu^N)[\tilde{X}^N]- \partial_\mu f(x,y,\mu^N)[\tilde{X}^N]\cdot\partial_yv(x,y,\mu^N)\\
&\hspace{4cm}-\partial_\mu a(x,y,\mu^N)[\tilde{X}^N]:\partial^2_y v(x,y,\mu^N)\pi(dy;x,\mu^N)\cdot\tilde{Y}   \biggr] \nonumber\\
&- \tilde{\E}\biggl[\int_{\R^d} \partial_\mu h(x,y,\mu)[\tilde{X}]- \partial_\mu f(x,y,\mu)[\tilde{X}]\cdot \partial_y v(x,y,\mu)\\
&\hspace{4cm}-\partial_\mu a(x,y,\mu)[\tilde{X}]:\partial^2_y v(x,y,\mu)\pi(dy;x,\mu)\cdot\tilde{Y}   \biggr]\biggr)/ \norm{\tilde{Y}}_{L^2(\tilde\W,\tilde\F,\tilde\Prob;\R^d)}\\
&\leq \lim_{N\toinf}\tilde{\E}\biggl[\biggl|\int_{\R^d} \left\{\partial_\mu h(x,y,\mu^N)[\tilde{X}^N]- \partial_\mu f(x,y,\mu^N)[\tilde{X}^N]\cdot\partial_y v(x,y,\mu^N)-\right.\\
&\hspace{7cm}\left.-\partial_\mu a(x,y,\mu^N)[\tilde{X}^N]\partial^2_yv(x,y,\mu^N)\right\}\pi(y;x,\mu^N)   \\
&- \biggl\lbrace\partial_\mu h(x,y,\mu)[\tilde{X}]- \partial_\mu f(x,y,\mu)[\tilde{X}]\cdot\partial_yv(x,y,\mu)-\partial_\mu a(x,y,\mu)[\tilde{X}]:\partial^2_yv(x,y,\mu)\biggr\rbrace\pi(y;x,\mu)dy\biggr|^2  \biggr]^{1/2}\\
& \text{ by H\"older's inequality}\\
&\leq C\lim_{N\toinf} \biggl\lbrace\tilde{\E}\biggl[|\tilde{X}^N-\tilde{X}|^2\biggr]+\bb{W}_2(\mu^N,\mu)\biggr\rbrace \\
& = 0 .
\end{align*}
Then in fact, by, e.g. Proposition 3.2.15 in \cite{DM}, $\int_{\R^d}h(x,y,\tilde{X})\pi(y;x,\tilde{X})dy$ is Fr\'echet differentiable, and
\begin{align*}
\tilde{\E}\biggl[D\int_{\R^d}\tilde{h}(x,y,\tilde{X})\tilde{\pi}(y;x,\tilde{X})dy\cdot \tilde{Y}\biggr] &=  \tilde{\E}\biggl[\int_{\R^d} \partial_\mu h(x,y,\mu)[\tilde{X}]- \partial_\mu f(x,y,\mu)[\tilde{X}]\cdot\partial_y v(x,y,\mu)\\
&-\partial_\mu a(x,y,\mu)[\tilde{X}]:\partial^2_yv(x,y,\mu)\pi(dy;x,\mu)\cdot\tilde{Y}   \biggr].
\end{align*}

Since this holds for all $\tilde{Y}\in L^2(\tilde\W,\tilde\F,\tilde\Prob;\R^d)$, we have
\begin{align*}
D\int_{\R^d}\tilde{h}(x,y,\tilde{X})\tilde{\pi}(y;x,\tilde{X})dy& = \int_{\R^d} \partial_\mu h(x,y,\mu)[\tilde{X}]- \partial_\mu f(x,y,\mu)[\tilde{X}]\cdot\partial_y v(x,y,\mu)\\
&-\partial_\mu a(x,y,\mu)[\tilde{X}]:\partial^2_yv_{yy}(x,y,\mu)\pi(dy;x,\mu)
\end{align*}
$\tilde{\Prob}$-almost surely, for all $x\in\R^d$, $\tilde{X}\in L^2(\tilde\W,\tilde\F,\tilde\Prob;\R^d),\mu = \mc{L}(\tilde{X})$.

Then, by definition,
\begin{align*}
D^{(1,0,0)}\int_{\R^d}h(x,y,\mu)\pi(y;x,\mu)dy[z]& = \int_{\R^d} D^{(1,0,0)} h(x,y,\mu)[z]- \mc{L}^{(1,0,0)}_{x,\mu}[z]v(x,y,\mu)\pi(dy;x,\mu),
\end{align*}
for all $x\in\R^d,\mu \in \mc{P}_2(\R^d)$, and $\mu$- almost every $z\in \R^d$.

Then, using the same iterative argument as outlined above in the case of the $x$ derivatives both in $l$ and $n$ along with this same argument to handle the Fr\'echet derivatives, we arrive at the expression for the mixed derivatives given in \eqref{eq:derivativetransferformula}, with $\bm{\beta}=0$. Finally, differentiating the expression in $z_1,...,z_n$ according to $\bm{\beta}$ and using uniform integrability to pass the derivatives into the integral against $\pi$, we obtain the full expression.\end{proof}

\begin{remark}\label{Rem:C_constants}
The non-negative integers $C_{(\bm{p}_k,j,n,l)}$ in the statement of Lemma \ref{lemma:derivativetransferformulas} can be iteratively computed according to the following rules:
\begin{align*}
C_{(\bm{p}_0,0,0,0)}&=0.
\end{align*}
and to go up in $l$ (taking an $x$ derivative), we have for any $l,n\in\bb{N}$, $\bb{N}\ni j\leq l+1$, $\bb{N}\ni k\leq n$, and $\bm{p}_k\in\binom{\br{1,...,n}}{k}$:
\begin{align*}
C_{(\bm{p}_k,j,n,l+1)}& =
\begin{cases}
C_{(\bm{p}_k,l,n,l)}, &\text{ if }j=l+1\\
C_{(\bm{p}_k,0,n,l)}, &\text{ if }j=0\\
C_{(\bm{p}_k,1,n,l)}+1, &\text{ if }j=1 \\
C_{(\bm{p}_k,j-1,n,l)} +C_{(\bm{p}_k,j,n,l)}, &\text{ otherwise}
\end{cases}.
\end{align*}

To go up in $n$ (taking a measure derivative), we have for any $l,n\in\bb{N}$, $\bb{N}\ni j\leq l$, $\bb{N}\ni k\leq n+1$, and $\bm{p}_k\in\binom{\br{1,...,n+1}}{k}$
\begin{align*}
C_{(\bm{p}_k,j,n+1,l)}& =
\begin{cases}
C_{(\br{1,...,n},j,n,l)}&\text{ if }k=n+1\\
C_{(\bm{p}_0,j,n,l)}&\text{ if }k=0\\
\1_{\bm{p}_1 = \br{n+1}}+C_{(\bm{p}_1,j,n,l)}\1_{\bm{p}_1 \neq  \br{n+1}}&\text{ if }k=1\\
C_{(\bm{p}_{k}\setminus \br{n+1},j,n,l)} \1_{\br{n+1}\in\bm{p}_{k} } +C_{(\bm{p}_k,j,n,l)}\1_{\br{n+1}\not\in\bm{p}_{k}}, &\text{ otherwise}\\
\end{cases}.
\end{align*}
This can be seen from tracking the constants in the iterative argument outlined in the proof.
\end{remark}

In dimension $d=1$, one can actually strengthen Lemma \ref{lemma:derivativetransferformulas} using the available exact formula for the invariant distribution. In fact in $d=1$ we have the Lemma \ref{lemma:derivativetransferformulas_d1}.
\begin{lemma}\label{lemma:derivativetransferformulas_d1}
Consider the case where $h,f,a:\R\times\R\times\mc{P}_2(\R)\tto \R$, i.e., $d=1$. Then, the results of Lemma \ref{lemma:derivativetransferformulas} hold for $h$ that is jointly continuous in $(x,y,\mu)$ and grows at most polynomially in $y$ uniformly in $x,\mu$, and dropping the local H\"older continuity assumption from the Assumptions \ref{assumption:uniformellipticity} and \ref{assumption:retractiontomean} and the definition of $\mc{M}_p^{\bm{\zeta}}$ (that is, Equation \eqref{eq:LionsClassYLipschitz}).

\end{lemma}
\begin{proof}
The essential idea of this argument is that appealing to Equation (28) from \cite{PV2} and the expression given for $\mc{K}_2$ in Lemma 4.1 in \cite{RocknerHolderContinuous} is unnecessary, and we can instead use a direct argument using the explicit form of $\pi$ and $v$ that we have in the 1D situation (see \cite{GS} p. 105). Thus the use of local H\"older continuity in order to establish bounds on the fundamental solution of the Cauchy problem associated to $\mc{L}_{x,\mu}$ is unnecessary.

We have (with some abuse of notation also denoting the density of $\pi$ by $\pi$):
\begin{align}
\label{eq:explicit1Dpi}\pi(y;x,\mu)& = \frac{Z(x,\mu)}{a(x,y,\mu)}\exp\biggl(\int_0^y \frac{f(x,\bar{y},\mu)}{a(x,\bar{y},\mu)}d\bar{y}\biggr)\\
\label{eq:explicit1Dv}v(x,y,\mu)& = \int_{-\infty}^y \frac{1}{a(x,\bar{y},\mu)\pi(\bar{y};x,\mu)}\biggl[\int_{-\infty}^{\bar{y}}h(x,\tilde{y},\mu)\pi(\tilde{y};x,\mu)d\tilde{y} - \bar{h}(x,\mu)\int_{-\infty}^{\bar{y}}\pi(\tilde{y};x,\mu)d\tilde{y}\biggr]d\bar{y}
\end{align}
where $Z^{-1}(x,\mu)\coloneqq \int_{\R}\frac{1}{a(x,y,\mu)}\exp\biggl(\int_0^y \frac{f(x,\bar{y},\mu)}{a(x,\bar{y},\mu)}d\bar{y}\biggr)dy$ is the normalizing constant, and we use $\bar{h}(x,\mu)\coloneqq \int_{\R}h(x,\bar{y},\mu)\pi(\bar{y};x,\mu)d\bar{y}$. In the proof we will at times use the subscript notation e.g. $f_x$ rather than $\partial_x f$ to denote partial derivatives for presentation purposes.

First we see that Equation \eqref{eq:derivativetransferformula} holds with $n=0$ and $l=1$.

Using the dominated convergence theorem to pass derivatives inside the integrals where necessary, we have
\begin{align*}
\partial_x\int_{\R} h(x,y,\mu)\pi(y;x,\mu)dy & = \int_{\R}h_x(x,y,\mu)\pi(y;x,\mu)dy +\int_{\R}h(x,y,\mu)\pi_x(y;x,\mu)dy\\
\pi_x(y;x,\mu)& = \pi(y;x,\mu)\biggl[\int_0^y \frac{1}{a(x,\bar{y},\mu)}[f_x(x,\bar{y},\mu)-f(x,\bar{y},\mu)a_x(x,\bar{y},\mu)/a(x,\bar{y},\mu)]d\bar{y}+\frac{Z_x(x,\mu)}{Z(x,\mu)}\\
&\hspace{8cm}-\frac{a_x(x,y,\mu)}{a(x,y,\mu)} \biggr]\\
v_y(x,y,\mu)& = \frac{1}{a(x,y,\mu)\pi(y;x,\mu)}\biggl[\int_{-\infty}^y h(x,\bar{y},\mu)\pi(\bar{y};x,\mu)d\bar{y}-\bar{h}(x,\mu)\int_{-\infty}^y \pi(\bar{y};x,\mu)d\bar{y}\biggr]\\
v_{yy}(x,y,\mu)& = -\frac{f(x,y,\mu)}{a^2(x,y,\mu)\pi(y;x,\mu)}\biggl[\int_{-\infty}^y h(x,\bar{y},\mu)\pi(\bar{y};x,\mu)d\bar{y}-\bar{h}(x,\mu)\int_{-\infty}^y \pi(\bar{y};x,\mu)d\bar{y}\biggr]\\
&+\frac{h(x,y,\mu)-\bar{h}(x,\mu)}{a(x,y,\mu)}.
\end{align*}
We need to establish that
\begin{align*}
\int_{\R}h(x,y,\mu)\pi_x(y;x,\mu)dy & = -\int_{\R} \mc{L}^{(0,1,0)}_{x,\mu}v(x,y,\mu)\pi(y;x,\mu)dy
\end{align*}
so
\begin{align*}
&\frac{Z_x(x,\mu)}{Z(x,\mu)}\bar{h}(x,\mu)+\int_{\R}h(x,y,\mu)\biggl[\int_0^y \frac{f_x(x,\bar{y},\mu)-f(x,\bar{y},\mu)a_x(x,\bar{y},\mu)/a(x,\bar{y},\mu)}{a(x,\bar{y},\mu)}d\bar{y}-\frac{a_x(x,y,\mu)}{a(x,y,\mu)} \biggr]\pi(y;x,\mu)dy\\
&= \int_{\R}\frac{f_x(x,y,\mu)-f(x,y,\mu)a_x(x,y,\mu)/a(x,y,\mu)}{a(x,y,\mu)}\biggl[\bar{h}(x,\mu)\int_{-\infty}^y \pi(\bar{y};x,\mu)d\bar{y}-\int_{-\infty}^y h(x,\bar{y},\mu)\pi(\bar{y};x,\mu)d\bar{y}\biggr]\\
&+a_x(x,y,\mu)\frac{\bar{h}(x,\mu)-h(x,y,\mu)}{a(x,y,\mu)}\pi(y;x,\mu)dy.
\end{align*}
Noting the last terms are already the same, we can just establish
\begin{align*}
&\frac{Z_x(x,\mu)}{Z(x,\mu)}\bar{h}(x,\mu)+\int_{\R}h(x,y,\mu)\int_0^y \frac{f_x(x,\bar{y},\mu)-f(x,\bar{y},\mu)a_x(x,\bar{y},\mu)/a(x,\bar{y},\mu)}{a(x,\bar{y},\mu)}d\bar{y} \pi(y;x,\mu)dy\\
&= \int_{\R}\frac{f_x(x,y,\mu)-f(x,y,\mu)a_x(x,y,\mu)/a(x,y,\mu)}{a(x,y,\mu)}\biggl[\bar{h}(x,\mu)\int_{-\infty}^y \pi(\bar{y};x,\mu)d\bar{y}-\int_{-\infty}^y h(x,\bar{y},\mu)\pi(\bar{y};x,\mu)d\bar{y}\biggr]\\
&+\bar{h}(x,\mu)\frac{a_x(x,y,\mu)}{a(x,y,\mu)}\pi(y;x,\mu)dy.
\end{align*}

We work from the first term and get to the second. Changing the order of integration, we have
\begin{align*}
&\frac{Z_x(x,\mu)}{Z(x,\mu)}\bar{h}(x,\mu)+\int_{\R}h(x,y,\mu)\int_0^y \frac{f_x(x,\bar{y},\mu)-f(x,\bar{y},\mu)a_x(x,\bar{y},\mu)/a(x,\bar{y},\mu)}{a(x,\bar{y},\mu)}d\bar{y} \pi(y;x,\mu)dy \\
&= \frac{Z_x(x,\mu)}{Z(x,\mu)}\bar{h}(x,\mu)+\int_0^\infty \frac{f_x(x,\bar{y},\mu)-f(x,\bar{y},\mu)a_x(x,\bar{y},\mu)/a(x,\bar{y},\mu)}{a(x,\bar{y},\mu)}\int_{\bar{y}}^\infty h(x,y,\mu) \pi(y;x,\mu)dy d\bar{y} \\
&- \int_{-\infty}^0 \frac{f_x(x,\bar{y},\mu)-f(x,\bar{y},\mu)a_x(x,\bar{y},\mu)/a(x,\bar{y},\mu)}{a(x,\bar{y},\mu)}\int_{-\infty}^{\bar{y}} h(x,y,\mu) \pi(y;x,\mu)dy d\bar{y}.
\end{align*}
Then using
\begin{align*}
\int_{\bar{y}}^\infty h(x,y,\mu) \pi(y;x,\mu)dy &= \bar{h}(x,\mu) - \int_{-\infty}^{\bar{y}} h(x,y,\mu) \pi(y;x,\mu)dy
\end{align*}
we can continue
\begin{align*}
&= \frac{Z_x(x,\mu)}{Z(x,\mu)}\bar{h}(x,\mu)+\int_0^\infty \frac{f_x(x,\bar{y},\mu)-f(x,\bar{y},\mu)a_x(x,\bar{y},\mu)/a(x,\bar{y},\mu)}{a(x,\bar{y},\mu)}[\bar{h}(x,\mu) - \int_{-\infty}^{\bar{y}} h(x,y,\mu) \pi(y;x,\mu)dy] d\bar{y} \\
&- \int_{-\infty}^0 \frac{f_x(x,\bar{y},\mu)-f(x,\bar{y},\mu)a_x(x,\bar{y},\mu)/a(x,\bar{y},\mu)}{a(x,\bar{y},\mu)}\int_{-\infty}^{\bar{y}} h(x,y,\mu) \pi(y;x,\mu)dy d\bar{y}\\
& = \frac{Z_x(x,\mu)}{Z(x,\mu)}\bar{h}(x,\mu)+\bar{h}(x,\mu) \int_0^\infty \frac{f_x(x,\bar{y},\mu)-f(x,\bar{y},\mu)a_x(x,\bar{y},\mu)/a(x,\bar{y},\mu)}{a(x,\bar{y},\mu)}d\bar{y}  \\
&- \int_{\R} \frac{f_x(x,\bar{y},\mu)-f(x,\bar{y},\mu)a_x(x,\bar{y},\mu)/a(x,\bar{y},\mu)}{a(x,\bar{y},\mu)}\int_{-\infty}^{\bar{y}} h(x,y,\mu) \pi(y;x,\mu)dy d\bar{y}.
\end{align*}
Then using that $\pi$ integrates to $1$, we have
\begin{align*}
& = \frac{Z_x(x,\mu)}{Z(x,\mu)}\bar{h}(x,\mu)+\bar{h}(x,\mu) \int_0^\infty \frac{f_x(x,\bar{y},\mu)-f(x,\bar{y},\mu)a_x(x,\bar{y},\mu)/a(x,\bar{y},\mu)}{a(x,\bar{y},\mu)}\int_{\R}\pi(y;x,\mu)dyd\bar{y}  \\
&- \int_{\R} \frac{f_x(x,\bar{y},\mu)-f(x,\bar{y},\mu)a_x(x,\bar{y},\mu)/a(x,\bar{y},\mu)}{a(x,\bar{y},\mu)}\int_{-\infty}^{\bar{y}} h(x,y,\mu) \pi(y;x,\mu)dy d\bar{y}\\
& = \frac{Z_x(x,\mu)}{Z(x,\mu)}\bar{h}(x,\mu)+\nonumber\\
&+\bar{h}(x,\mu) \int_0^\infty \frac{f_x(x,\bar{y},\mu)-f(x,\bar{y},\mu)a_x(x,\bar{y},\mu)/a(x,\bar{y},\mu)}{a(x,\bar{y},\mu)}\biggl[\int_{-\infty}^{\bar{y}}\pi(y;x,\mu)dy+\int_{\bar{y}}^{\infty}\pi(y;x,\mu)dy\biggr]d\bar{y}  \\
&- \int_{\R} \frac{f_x(x,\bar{y},\mu)-f(x,\bar{y},\mu)a_x(x,\bar{y},\mu)/a(x,\bar{y},\mu)}{a(x,\bar{y},\mu)}\int_{-\infty}^{\bar{y}} h(x,y,\mu) \pi(y;x,\mu)dy d\bar{y}\\
& = \frac{Z_x(x,\mu)}{Z(x,\mu)}\bar{h}(x,\mu)+\bar{h}(x,\mu) \int_0^\infty \frac{f_x(x,\bar{y},\mu)-f(x,\bar{y},\mu)a_x(x,\bar{y},\mu)/a(x,\bar{y},\mu)}{a(x,\bar{y},\mu)}\int_{-\infty}^{\bar{y}}\pi(y;x,\mu)dyd\bar{y}  \\
&+\bar{h}(x,\mu) \int_{-\infty}^0 \frac{f_x(x,\bar{y},\mu)-f(x,\bar{y},\mu)a_x(x,\bar{y},\mu)/a(x,\bar{y},\mu)}{a(x,\bar{y},\mu)}\int_{-\infty}^{\bar{y}}\pi(y;x,\mu)dyd\bar{y} \\
& + \bar{h}(x,\mu) \int_0^\infty \frac{f_x(x,\bar{y},\mu)-f(x,\bar{y},\mu)a_x(x,\bar{y},\mu)/a(x,\bar{y},\mu)}{a(x,\bar{y},\mu)}\int_{\bar{y}}^{\infty}\pi(y;x,\mu)dyd\bar{y}\\
&-\bar{h}(x,\mu) \int_{-\infty}^0 \frac{f_x(x,\bar{y},\mu)-f(x,\bar{y},\mu)a_x(x,\bar{y},\mu)/a(x,\bar{y},\mu)}{a(x,\bar{y},\mu)}\int_{-\infty}^{\bar{y}}\pi(y;x,\mu)dyd\bar{y}\\
&- \int_{\R} \frac{f_x(x,\bar{y},\mu)-f(x,\bar{y},\mu)a_x(x,\bar{y},\mu)/a(x,\bar{y},\mu)}{a(x,\bar{y},\mu)}\int_{-\infty}^{\bar{y}} h(x,y,\mu) \pi(y;x,\mu)dy d\bar{y}\\
& = \frac{Z_x(x,\mu)}{Z(x,\mu)}\bar{h}(x,\mu)+\bar{h}(x,\mu) \int_\R \frac{f_x(x,\bar{y},\mu)-f(x,\bar{y},\mu)a_x(x,\bar{y},\mu)/a(x,\bar{y},\mu)}{a(x,\bar{y},\mu)}\int_{-\infty}^{\bar{y}}\pi(y;x,\mu)dyd\bar{y}  \\
& + \bar{h}(x,\mu) \int_\R \int_0^{y}\frac{f_x(x,\bar{y},\mu)-f(x,\bar{y},\mu)a_x(x,\bar{y},\mu)/a(x,\bar{y},\mu)}{a(x,\bar{y},\mu)}d\bar{y}\pi(y;x,\mu)dy\\
&- \int_{\R} \frac{f_x(x,\bar{y},\mu)-f(x,\bar{y},\mu)a_x(x,\bar{y},\mu)/a(x,\bar{y},\mu)}{a(x,\bar{y},\mu)}\int_{-\infty}^{\bar{y}} h(x,y,\mu) \pi(y;x,\mu)dy d\bar{y},\\
\end{align*}
where in the last step we again changed the order of integration.

Now using
\begin{align*}
\pi_x(y;x,\mu)& = \pi(y;x,\mu)\biggl[\int_0^y \frac{1}{a(x,\bar{y},\mu)}[f_x(x,\bar{y},\mu)-f(x,\bar{y},\mu)a_x(x,\bar{y},\mu)/a(x,\bar{y},\mu)]d\bar{y}+\frac{Z_x(x,\mu)}{Z(x,\mu)}-\frac{a_x(x,y,\mu)}{a(x,y,\mu)} \biggr],
\end{align*}
we have
\begin{align*}
&\frac{Z_x(x,\mu)}{Z(x,\mu)}\bar{h}(x,\mu) + \bar{h}(x,\mu) \int_\R \int_0^{y}\frac{f_x(x,\bar{y},\mu)-f(x,\bar{y},\mu)a_x(x,\bar{y},\mu)/a(x,\bar{y},\mu)}{a(x,\bar{y},\mu)}d\bar{y}\pi(y;x,\mu)dy\\
& = \bar{h}(x,\mu)\int_\R \biggl[\int_0^{y}\frac{f_x(x,\bar{y},\mu)-f(x,\bar{y},\mu)a_x(x,\bar{y},\mu)/a(x,\bar{y},\mu)}{a(x,\bar{y},\mu)}d\bar{y}+\frac{Z_x(x,\mu)}{Z(x,\mu)} \biggr]\pi(y;x,\mu)dy\\
& = \bar{h}(x,\mu)\int_{\R}\pi_x(y;x,\mu) dy + \bar{h}(x,\mu)\int_{\R}\frac{a_x(x,y,\mu)}{a(x,y,\mu)}\pi(y;x,\mu)dy\\
& = \bar{h}(x,\mu)\partial_x\int_{\R}\pi(y;x,\mu) dy + \bar{h}(x,\mu)\int_{\R}\frac{a_x(x,y,\mu)}{a(x,y,\mu)}\pi(y;x,\mu)dy\\
& = \bar{h}(x,\mu)\int_{\R}\frac{a_x(x,y,\mu)}{a(x,y,\mu)}\pi(y;x,\mu)dy,
\end{align*}
where in the first and last step we again used $\pi$ integrates to 1 for all $x,\mu$.

Thus we can replace the first and third term in our chain of equalities to get:
\begin{align*}
&\int_\R \frac{f_x(x,\bar{y},\mu)-f(x,\bar{y},\mu)a_x(x,\bar{y},\mu)/a(x,\bar{y},\mu)}{a(x,\bar{y},\mu)}\biggl[\bar{h}(x,\mu)\int_{-\infty}^{\bar{y}}\pi(y;x,\mu)dy-\int_{-\infty}^{\bar{y}} h(x,y,\mu) \pi(y;x,\mu)dy \biggr]d\bar{y}  \\
& + \bar{h}(x,\mu) \int_\R \frac{a_x(x,y,\mu)}{a(x,y,\mu)}\pi(y;x,\mu)dy
\end{align*}
as desired.

For the case that $n=1,l=0$, we use the same argument, but using Gateaux derivatives of the lifted functions $\tilde{h},\tilde{a},\tilde{f},\tilde{v},\tilde{\pi}$ in the direction of some $\tilde{Y}\in L^2(\tilde{\W},\tilde{\F},\tilde{\Prob};\R)$, and a similar argument to that in Lemma \ref{lemma:derivativetransferformulas} to establish Fr\'echet differentiability of the lifted function $\int_\R \tilde{h}(x,y,\tilde{X})\tilde{\pi}(y;x,\tilde{X})$ and hence Lions differentiability of $\int_{\R}h(x,y,\mu)\pi(y;x,\mu)dy$.

Once the equality \eqref{eq:derivativetransferformula} is established in these two base cases, the same iterative argument as outlined in Lemma \ref{lemma:derivativetransferformulas} applies to arrive at the full expression.

Now we show the equality \eqref{eq:muLipschitztransferformula}. We have
\begin{align*}
&\int_{\R} h(x,y,\mu_1)-h(x,y,\mu_2)-[\mc{L}_{x,\mu_1}-\mc{L}_{x,\mu_2}]v(x,y,\mu_2)\pi(dy;x,\mu_1)\\
& = \bar{h}(x,\mu_1) - \int_{\R}h(x,y,\mu_2)+\mc{L}_{x,\mu_1}v(x,y,\mu_2)\pi(dy;x,\mu_1)  +\int_{\R}h(x,y,\mu_2)\pi(dy;x,\mu_1) - \bar{h}(x,\mu_2)\\
& = \bar{h}(x,\mu_1) - \bar{h}(x,\mu_2)- \int_{\R}\mc{L}_{x,\mu_1}v(x,y,\mu_2)\pi(dy;x,\mu_1),
\end{align*}
so to show
\begin{align*}
&\int_{\R}h(x,y,\mu_1)\pi(dy;x,\mu_1) -\int_{\R}h(x,y,\mu_2)\pi(dy;x,\mu_2)  = \nonumber\\
&\qquad= \int_{\R} h(x,y,\mu_1)-h(x,y,\mu_2)-[\mc{L}_{x,\mu_1}-\mc{L}_{x,\mu_2}]v(x,y,\mu_2)\pi(dy;x,\mu_1)
\end{align*}
it is sufficient to show
\begin{align*}
\int_{\R}\mc{L}_{x,\mu_1}v(x,y,\mu_2)\pi(dy;x,\mu_1) & =0.
\end{align*}

Using the previous calculation for $v_y$ and $v_{yy}$, we have
\begin{align*}
&\int_{\R}\mc{L}_{x,\mu_1}v(x,y,\mu_2)\pi(dy;x,\mu_1)=\int_{\R}\frac{f(x,y,\mu_1)-f(x,y,\mu_2)a(x,y,\mu_1)/a(x,y,\mu_2)}{a(x,y,\mu_2)}\frac{\pi(y;x,\mu_1)}{\pi(y;x,\mu_2)}\times\nonumber\\
&\hspace{4cm}\times\biggl[\int_{-\infty}^y h(x,\bar{y},\mu_2)\pi(\bar{y};x,\mu_2)d\bar{y}-\bar{h}(x,\mu_2)\int_{-\infty}^y \pi(\bar{y};x,\mu_2)d\bar{y}\biggr]dy\\
&\quad+\int_{\R}a(x,y,\mu_1)\frac{h(x,y,\mu_2)-\bar{h}(x,\mu_2)}{a(x,y,\mu_2)}\pi(y;x,\mu_1)dy\\
& =\int_{\R}g(x,y,\mu_1,\mu_2)\exp(\int_0^y g(x,\bar{y},\mu_1,\mu_2)d\bar{y}) \biggl[\int_{-\infty}^y h(x,\bar{y},\mu_2)\pi(\bar{y};x,\mu_2)d\bar{y}-\bar{h}(x,\mu_2)\int_{-\infty}^y \pi(\bar{y};x,\mu_2)d\bar{y}\biggr]dy\\
&\quad+\int_{\R}a(x,y,\mu_1)\frac{h(x,y,\mu_2)-\bar{h}(x,\mu_2)}{a(x,y,\mu_2)}\pi(y;x,\mu_1)dy,\\
\end{align*}
where $g(x,y,\mu_1,\mu_2)=\frac{f(x,y,\mu_1)}{a(x,y,\mu_1)}-\frac{f(x,y,\mu_2)}{a(x,y,\mu_2)}$, and here we used the explicit form of $\pi(y;x,\mu)$. Then by integration by parts, using that
\begin{align*}
\lim_{y\tto\pm\infty}\int_{-\infty}^y h(x,\bar{y},\mu_2)\pi(\bar{y};x,\mu_2)d\bar{y}-\bar{h}(x,\mu_2)\int_{-\infty}^y \pi(\bar{y};x,\mu_2)d\bar{y}=0,
\end{align*}
we get
\begin{align*}
&\int_{\R}\mc{L}_{x,\mu_1}v(x,y,\mu_2)\pi(dy;x,\mu_1)=\\
&= - \int_{\R}\exp(\int_0^y \frac{f(x,\bar{y},\mu_1)}{a(x,\bar{y},\mu_1)}-\frac{f(x,\bar{y},\mu_2)}{a(x,\bar{y},\mu_2)}d\bar{y})[h(x,y,\mu_2) - \bar{h}(x,\mu_2)]\pi(y;x,\mu_2)dy \\
&\quad +\int_{\R}a(x,y,\mu_1)\frac{h(x,y,\mu_2)-\bar{h}(x,\mu_2)}{a(x,y,\mu_2)}\pi(y;x,\mu_1)dy\\
& = - \int_{\R} \frac{a(x,y,\mu_1)}{a(x,y,\mu_2)}\pi(y;x,\mu_1)[h(x,y,\mu_2) - \bar{h}(x,\mu_2)]dy+\int_{\R}a(x,y,\mu_1)\frac{h(x,y,\mu_2)-\bar{h}(x,\mu_2)}{a(x,y,\mu_2)}\pi(y;x,\mu_1)dy\\
&=0
\end{align*}
as desired. The proof of the equality \eqref{eq:xLipschitztransferformula} follows in the exact same way.
\end{proof}

\begin{lemma}\label{lemma:explicitrateofgrowthofderivativesinparameters1D}
Consider $B:\R^d\times\R^d\times\bb{P}_2(\R^d)\tto \R$ such that
\begin{align*}
\int_{\R^d}B(x,y,\mu)\pi(dy;x,\mu)=0,\forall x\in \R^d,\mu\in \mc{P}_2(\R^d).
\end{align*}
Suppose that for some complete collection of multi-indices $\bm{\zeta}$ that $B\in \mc{M}_{p}^{\bm{\zeta}}(\R^d\times\R^d\times \mc{P}_2(\R^d);\R),f\in \mc{M}_{p}^{\bm{\zeta}}(\R^d\times\R^d\times \mc{P}_2(\R^d);\R^d),a\in \mc{M}_{p}^{\bm{\zeta}}(\R^d\times\R^d\times \mc{P}_2(\R^d);\R^{d\times d})$.
Then for the unique classical solution $u:\R^d\times\R^d\times\mc{P}_2(\R^d)\tto \R$ to
\begin{align*}
\mc{L}_{x,\mu}u(x,y,\mu)=B(x,y,\mu)
\end{align*}
such that $\int_{\R^d}u(x,y,\mu)\pi(dy;x,\mu)=0$, and $u$ has at most polynomial growth as $|y|\tto \infty$ (which exists by Lemma \ref{lemma:Ganguly1DCellProblemResult}): $u,\partial_{y_i}u,\partial_{y_i}\partial_{y_j}u\in \mc{M}_{p}^{\bm{\zeta}}(\R^d\times\R^d\times \mc{P}_2(\R^d);\R)$ for all $i,j=1,...,d$.
\end{lemma}
\begin{proof}
The proof essentially uses the same tools and a similar method to Lemma \ref{lemma:derivativetransferformulas} here and Theorem 2.1 Step 3 in \cite{RocknerFullyCoupled}, so we will only check this in the case for $(n,l,\bm{\beta})=(0,1,0)$ and then comment on how the rest of the terms follow. Importantly, Lemma \ref{lemma:derivativetransferformulas} only assumes existence and polynomial growth of derivatives up to one order less than the derivative we obtain from Equation \eqref{eq:derivativetransferformula}.

The result for $(n,l,\bm{\beta})=(0,0,0)$ is just another way of writing Lemma \ref{lemma:Ganguly1DCellProblemResult}, once we establish continuity of $u,\partial_y u,\partial^2_y u$ in $x,\mu$. For this, the proof is similar to Step 4 in the proof of Theorem 2.1 in \cite{RocknerFullyCoupled}.

We first note that
\begin{align*}
\mc{L}_{x,\mu_1}[u(x,y,\mu_1) - u(x,y,\mu_2)] &=B(x,y,\mu_1) -  \mc{L}_{x,\mu_1}u(x,y,\mu_2) \nonumber\\
&= B(x,y,\mu_1) - B(x,y,\mu_2) -[\mc{L}_{x,\mu_1}-\mc{L}_{x,\mu_2}]u(x,y,\mu_2).
\end{align*}

By the equality \eqref{eq:muLipschitztransferformula} from Lemma \ref{lemma:derivativetransferformulas}, we have
\begin{align*}
&\int_{\R^d}B(x,y,\mu_1) - B(x,y,\mu_2) -[\mc{L}_{x,\mu_1}-\mc{L}_{x,\mu_2}]u(x,y,\mu_2)\pi(dy,x,\mu_1) = \nonumber\\
&\hspace{3cm}= \int_{\R^d}B(x,y,\mu_1)\pi(dy;x,\mu_1) -\int_{\R^d}B(x,y,\mu_2)\pi(dy;x,\mu_2)=0,
\end{align*}
so in fact the inhomogeneity in the above Poisson equation is centered. Now, via the assumptions on $B,f,a$ and the fact that the bound on the growth of $u$ in Lemma \ref{lemma:Ganguly1DCellProblemResult} depends linearly on the local H\"older semi-norm of the Poisson Equation's inhomogeneity as per the Proof of Theorem 2.1 i) in \cite{RocknerFullyCoupled}, we have:
\begin{align*}
|u(x,y,\mu_1) - u(x,y,\mu_2)|&\leq C\bb{W}_2(\mu_1,\mu_2)(1+|y|)^{p}\\
|\partial_y u(x,y,\mu_1) - \partial_y u(x,y,\mu_2)|&\leq C\bb{W}_2(\mu_1,\mu_2)(1+|y|)^{p'}\\
|\partial^2_yu(x,y,\mu_1) - \partial^2_yu(x,y,\mu_2)|&\leq C\bb{W}_2(\mu_1,\mu_2)(1+|y|)^{p''},
\end{align*}
for some $C>0$ and $p,p',p''\in\bb{N}$.

The proof with $\mu_1,\mu_2$ replaced by $x_1,x_2$ follows in the same way, and thus the desired continuity in $x,\bb{W}_2$ is established.

To obtain continuity of and a rate of polynomial growth for $\partial_xu$, we differentiate the equation that $u$ satisfies to get
\begin{align*}
\mc{L}_{x,\mu}D^{(0,1,0)}u(x,y,\mu)&=D^{(0,1,0)}B(x,y,\mu)- D^{(0,1,0)}f(x,y,\mu)\cdot \partial_y u(x,y,\mu) - D^{(0,1,0)}a(x,y,\mu):\partial^2_yu(x,y,\mu)\\
& = D^{(0,1,0)}B(x,y,\mu) - \mc{L}^{(0,1,0)}_{x,\mu}u(x,y,\mu)
\end{align*}
in the notation of Lemma \ref{lemma:derivativetransferformulas}, where a priori the derivative of $u$ in $x$ is in the weak sense.

But by the centering condition on $B$, we have by letting $B=h$ in Lemma \ref{lemma:derivativetransferformulas}, that $u=v$ in the statement of that same lemma. Thus, we have
\begin{align*}
\int_{\R^d}D^{(0,1,0)}B(x,y,\mu) - \mc{L}^{(0,1,0)}_{x,\mu}u(x,y,\mu)\pi(dy;x,\mu)& = D^{(0,1,0)}\biggl[\int_{\R^d}B(x,y,\mu)\pi(dy;x,\mu)\biggr]=0,
\end{align*}
and the inhomogeneity of the elliptic PDE that $D^{(0,1,0)}u$ solves in fact obeys the centering condition, and since we already know that $u$ is locally H\"older continuous with polynomial growth in $y$, Lemma \ref{lemma:Ganguly1DCellProblemResult} applies. This establishes that $D^{(0,1,0)}u$ grows at most polynomially and is locally H\"older continuous in $y$ uniformly in $x,\mu$. Then continuity in $x,\mu$ follows as above, but with $D^{(0,1,0)}B(x,y,\mu) - \mc{L}^{(0,1,0)}_{x,\mu}u(x,y,\mu)$ in the place of $B$, and $D^{(0,1,0)}u$ in the place of $u$.
The same process applies to $D^{(1,0,0)}u$. To establish its continuity in $z$, we first recall that for all $x,y,z\in\R^d,\mu\in\mc{P}_2(\R^d)$
\begin{align*}
\mc{L}_{x,\mu}D^{(1,0,0)}u(x,y,\mu)[z] = D^{(1,0,0)}B(x,y,\mu)[z]-\mc{L}^{(1,0,0)}_{x,\mu}[z]u(x,y,\mu)
\end{align*}
so
\begin{align*}
\mc{L}_{x,\mu}\biggl[D^{(1,0,0)}u(x,y,\mu)[z_1]-D^{(1,0,0)}u(x,y,\mu)[z_2]\biggr]&=D^{(1,0,0)}B(x,y,\mu)[z_1]-\mc{L}^{(1,0,0)}_{x,\mu}[z_1]u(x,y,\mu)\\
&-\biggl[D^{(1,0,0)}B(x,y,\mu)[z_2]-\mc{L}^{(1,0,0)}_{x,\mu}[z_2]u(x,y,\mu)\biggr].
\end{align*}
Then by the equality \eqref{eq:derivativetransferformula} from Lemma \ref{lemma:derivativetransferformulas}, we have for all $x,z\in\R^d,\mu\in\mc{P}_2(\R^d)$:
\begin{align*}
\int_{\R^d}D^{(1,0,0)}B(x,y,\mu)[z]-\mc{L}^{(1,0,0)}_{x,\mu}[z]u(x,y,\mu)\pi(dy;x,\mu) = D^{(0,1,0)}\int_{\R^d}B(x,y,\mu)\pi(dy;x,\mu)[z] =0,
\end{align*}
so the inhomogeneity in the Poisson equation above in centered. Thus, using the same argument as for the other continuity proofs and the assumed continuity of $D^{(1,0,0)}B,D^{(1,0,0)}f,D^{(1,0,0)}a$ in $z$ with the fact that $\partial_yu,\partial^2_yu$ grow at most polynomially in $y$, we get there is $p\in \bb{N}$ and $C>0$ such that
\begin{align*}
\biggl|D^{(1,0,0)}u(x,y,\mu)[z_1]-D^{(1,0,0)}u(x,y,\mu)[z_2]\biggr|\leq C|z_1-z_2| (1+|y|)^p,
\end{align*}
and similarly for $D^{(1,0,0)}\partial_y u$ and $D^{(1,0,0)}\partial^2_yu$.

All of the remaining bounds work in the same way, with the inhomogeneity of the elliptic PDE of the desired derivative of $u$ solves being the integrand of the expression for the corresponding derivative of $\bar{B}(x,y,\mu)$ from Equation \eqref{eq:derivativetransferformula} in Lemma \ref{lemma:derivativetransferformulas}. Put explicitly:
\begin{align}\label{eq:formulasatisfiedbyderivatives}
&\mc{L}_{x,\mu}D^{(n,l,\bm{\beta})}u(x,y,\mu)[z_1,...,z_n]\\
& = D^{(n,l,\bm{\beta})}B(x,y,\mu)[z_1,...,z_n] - \sum_{k=0}^n\sum_{j=0}^l\sum_{\bm{p}_k} C_{(\bm{p}_k,j,n,l)} \mc{L}^{(k,j,\bm{\alpha}(\bm{p}_k))}_{x,\mu}[z_{\bm{p}_k}] D^{(n-k,l-j,\bm{\alpha}(\bm{p}'_{n-k}))}u(x,y,\mu)[z_{\bm{p}'_{n-k}}].\nonumber
\end{align}

The inhomogeneity is always a jointly continuous (in the sense of Equation \eqref{eq:LionsClassYLipschitz}) function which grows at most polynomially in $y$ uniformly in $x$ and $\mu$, and only depends on lower order derivatives of $u$. Thus, it is clear the result follows by proceeding inductively on $n,l$.
\end{proof}

\begin{lemma}\label{lem:regularityofaveragedcoefficients}
Suppose that for some complete collection of multi-indices $\bm{\zeta}$ that $h\in \mc{M}_{p}^{\bm{\zeta}}(\R^d\times\R^d\times \mc{P}_2(\R^d);\R),f\in \mc{M}_{p}^{\bm{\zeta}}(\R^d\times\R^d\times \mc{P}_2(\R^d);\R^d),a\in \mc{M}_{p}^{\bm{\zeta}}(\R^d\times\R^d\times \mc{P}_2(\R^d);\R^{d\times d})$. Then $\bar{h}(x,\mu)\coloneqq\int_{\R^d} h(x,y,\mu)\pi(dy;x,\mu)\in \mc{M}_{b,L}^{\bm{\zeta}}(\R^d\times \mc{P}_2(\R^d);\R)$.\end{lemma}
\begin{proof}
This follows via Lemmas \ref{lemma:derivativetransferformulas} and \ref{lemma:Ganguly1DCellProblemResult} in a similar way to Lemma \ref{lemma:explicitrateofgrowthofderivativesinparameters1D}.

Boundedness and continuity of the derivatives when the coefficients are in $\mc{M}_{p}^{\bm{\zeta}}(\R^d\times\R^d\times \mc{P}_2(\R^d);\R^k),k\in\br{1,d,d\times d}$ follows in the same way as Lipschitz continuity. Thus, we only show the latter.
For the Lipschitz property when $(n,l,\bm{\beta})=(0,0,0)$, we have via Equations \eqref{eq:muLipschitztransferformula} and \eqref{eq:xLipschitztransferformula}, using the polynomial growth in $y$ of $v,\partial_yv,$ and $\partial^2_yv$ from Equation \eqref{eq:poissoneqfortransferformulas} gained by Lemma \ref{lemma:Ganguly1DCellProblemResult} with $B=h-\bar{h}$ and the assumed Lipschitz continuity of $f,a,h$:
\begin{align*}
|\bar{h}(x,\mu_1)-\bar{h}(x,\mu_2)|&\leq \int_{\R^d} C\bb{W}_2(\mu_1,\mu_2)(1+|y|)^p\pi(dy;x,\mu_1) \text{ for some $p\in \bb{N}$}\\
&\leq C\bb{W}_2(\mu_1,\mu_2)
\end{align*}
and
\begin{align*}
|\bar{h}(x_1,\mu)-\bar{h}(x_2,\mu)|&\leq \int_{\R^d} C|x_1-x_2|(1+|y|)^p\pi(dy;x_1,\mu) \text{ for some $p\in \bb{N}$}\\
&\leq C|x_1-x_2|,
\end{align*}
for all $x_1,x_2,x\in\R^d,\mu_1,\mu_2,\mu\in\mc{P}_2(\R^d)$, so the result follows.
Now, as with the previous two results, the rest follows via an induction argument on $n,l$. For the case $(n,l,\bm{\beta})=(0,1,0)$, via Equation \ref{lemma:derivativetransferformulas}, we have
\begin{align*}
|D^{(0,1,0)}\bar{h}(x,\mu)| & = \biggl|\int_{\R^d}D^{(0,1,0)}h(x,y,\mu)- \mc{L}^{(0,1,0)}_{x,\mu}v(x,y,\mu)\pi(y;x,\mu)dy\biggr|\\
&\leq \int_{\R^d} C(1+|y|)^p\pi(dy;x,\mu) \text{ for some $p\in \bb{N}$}\\
&\leq C
\end{align*}
from the fact that $\partial_x h,\partial_xf,\partial_xa,\partial_yv,\partial^2_yv$ grow at most polynomially in $y$ uniformly in $x,\mu$ by assumption and Lemma \ref{lemma:Ganguly1DCellProblemResult}.

In addition, from  Equation \eqref{eq:formulasatisfiedbyderivatives} with $B=h-\bar{h}$, we know
\begin{align*}
\mc{L}_{x,\mu}D^{(0,1,0)}v &= D^{(0,1,0)}h -D^{(0,1,0)}\bar{h} - \mc{L}^{(0,1,0)}_{x,\mu}v(x,y,\mu) \nonumber\\
&= D^{(0,1,0)}h  - \mc{L}^{(0,1,0)}_{x,\mu}v(x,y,\mu)-\int_{\R^d}D^{(0,1,0)}h - \mc{L}^{(0,1,0)}_{x,\mu}v(x,y,\mu)\pi(dy;x,\mu)
\end{align*}
the right side of which we now know grows at most polynomially in $y$ uniformly in $x,\mu$. Then
\begin{align*}
D^{(0,1,0)}\bar{h}(x,\mu_1)-D^{(0,1,0)}\bar{h}(x,\mu_2) & = \int_{\R^d}D^{(0,1,0)}h(x,y,\mu_1)- \mc{L}^{(0,1,0)}_{x,\mu_1}v(x,y,\mu_1)\pi(y;x,\mu_1)dy\\
&-\int_{\R^d}D^{(0,1,0)}h(x,y,\mu_2)- \mc{L}^{(0,1,0)}_{x,\mu_2}v(x,y,\mu_2)\pi(y;x,\mu_2)dy
\end{align*}
so using Equation \eqref{eq:xLipschitztransferformula} with $h(x,y,\mu) = D^{(0,1,0)}h(x,y,\mu)- D^{(0,1,0)}\mc{L}_{x,\mu}v(x,y,\mu)$, we get
\begin{align*}
D^{(0,1,0)}\bar{h}(x,\mu_1)-D^{(0,1,0)}\bar{h}(x,\mu_2) & = \int_{\R^d} D^{(0,1,0)}h(x,y,\mu_1)- \mc{L}^{(0,1,0)}_{x,\mu_1}v(x,y,\mu_1)-[D^{(0,1,0)}h(x,y,\mu_2)- \mc{L}^{(0,1,0)}_{x,\mu_2}v(x,y,\mu_2)]\\
&-[\mc{L}_{x,\mu_1}-\mc{L}_{x,\mu_2}]D^{(0,1,0)}v(x,y,\mu_2)\pi(dy;x,\mu_1),
\end{align*}
so
\begin{align*}
|D^{(0,1,0)}\bar{h}(x,\mu_1)-D^{(0,1,0)}\bar{h}(x,\mu_2)| & \leq \int_{\R^d} C(1+|y|)^p\bb{W}_2(\mu_1,\mu_2)\pi(dy;x,\mu_1)\text{ for some $p\in \bb{N}$}\\
&\leq C\bb{W}_2(\mu_1,\mu_2),
\end{align*}
where here we used from Lemma \ref{lemma:explicitrateofgrowthofderivativesinparameters1D} that $D^{(0,1,0)}\partial_y v$ and $D^{(0,1,0)}\partial^2_y v$ grow at most polynomially in $y$ uniformly in $x,\mu$.
In the same way, we can get
\begin{align*}
|D^{(0,1,0)}\bar{h}(x_1,\mu)-D^{(0,1,0)}\bar{h}(x_2,\mu)|\leq C|x_1-x_2|.
\end{align*}

The proof for $(k,l,\mu) = (1,0,0)$ being bounded Lipschitz in $x,\mu$ follows in essentially the same way. To see that $D^{(1,0,0)}h(x,y,\mu)[z]$ is Lipschitz in $z$, we use the representation
\begin{align*}
D^{(1,0,0)}h(x,y,\mu)[z]& = \int_{\R^d} \partial_\mu h(x,y,\mu)[z]- \mc{L}^{(1,0,0)}_{x,\mu}[z]v(x,y,\mu)\pi(dy;x,\mu),
\end{align*}
so
\begin{align*}
D^{(1,0,0)}h(x,y,\mu)[z_1]-D^{(1,0,0)}h(x,y,\mu)[z_2]& = \int_{\R^d} \partial_\mu h(x,y,\mu)[z_1]- \mc{L}^{(1,0,0)}_{x,\mu}[z_1]v(x,y,\mu) \\
&-[\partial_\mu h(x,y,\mu)[z_2]- \mc{L}^{(1,0,0)}_{x,\mu}[z_2]v(x,y,\mu)] \pi(dy;x,\mu),
\end{align*}
and by the Lipschitz properties of $D^{(1,0,0)}h,D^{(1,0,0)}f,D^{(1,0,0)}a$ in $z$, we see
\begin{align*}
|D^{(1,0,0)}h(x,y,\mu)[z_1]-D^{(1,0,0)}h(x,y,\mu)[z_2]|&\leq \int_{\R^d} C(1+|y|)^p|z_1-z_2|\pi(dy;x,\mu_1)\text{ for some $p\in \bb{N}$}\\
&\leq C|z_1-z_2|.
\end{align*}

Once again, the result for higher derivatives follows from iterating on the above method.

\end{proof}

\begin{remark}\label{remark:regularityfordoubledequations}
Although Lemmas \ref{lemma:Ganguly1DCellProblemResult},\ref{lemma:derivativetransferformulas}.\ref{lemma:explicitrateofgrowthofderivativesinparameters1D}, and \ref{lem:regularityofaveragedcoefficients} were stated for simplicity in terms of $f,a,\mc{L}_{x,\mu}$, and $\pi$, the only assumptions needed other than those posed in the statement of each Lemma are those on $f$ and $a$ from Assumptions \ref{assumption:uniformellipticity} and \ref{assumption:retractiontomean}. Thus, if we take care to change the domains of the functions in the statements of these Lemmas, we can also apply them to gain regularity of the ``doubled'' Poisson Equations \eqref{eq:doublecorrectorproblem} and \eqref{eq:tildechi}. Put explicitly:

Consider some $\tilde{f}:\R^j\times \R^j\times \mc{P}_2(\R^l)\tto \R^j$ such that there exists constants $C',\beta'>0$ independent of $x,y\in \R^j$ and $\mu\in\mc{P}_2(\R^l)$ such that
\begin{align}\label{eq:fdecayimplicationgeneral}
\tilde{f}(x,y,\mu)\cdot y\leq -\beta' |y|^2 +C',\forall x,y\in\R^j,\mu\in\mc{P}_2(\R^l),
\end{align}
$\tilde{f}$ grows at most linearly in $|y|$, $\tilde{f}$ has two uniformly bounded derivatives in $y$, and $\tilde{f}$ and both these derivatives are H\"older continuous in $y$ uniformly in $(x,\mu)$ and $\tilde{a}:\R^j\times \R^j\times \mc{P}(\R^l)\tto \R^{j\times j}$ such that
there exists $\lambda'_-,\lambda'_+>0$ such that $0<\lambda'_-\leq 2\frac{z^\top \tilde{a}(x,y,\mu)z}{|z|^2}\leq \lambda'_+<\infty$, $\forall x,y,z\in\R^j,z\neq 0,\mu\in\mc{P}_2(\R^l)$ and $\tilde{a}$ is bounded, has two uniformly bounded derivatives in $y$, and $\tilde{a}$ and both these derivatives are H\"older continuous in $y$ uniformly in $(x,\mu)$.

We can then conclude that Lemmas \ref{lemma:Ganguly1DCellProblemResult},\ref{lemma:derivativetransferformulas}.\ref{lemma:explicitrateofgrowthofderivativesinparameters1D}, and \ref{lem:regularityofaveragedcoefficients} hold replacing $\mc{L}_{x,\mu}$ by $\tilde{\mc{L}}_{x,\mu}$ which acts on $\phi \in C_b^2(\R^j)$ by
\begin{align}\label{eq:frozengeneratorgeneral}
\tilde{\mc{L}}_{x,\mu}\phi(y) = \tilde{f}(x,y,\mu)\cdot\nabla \phi(y)+\tilde{a}(x,y,\mu):\nabla^2\phi(y),
\end{align}
$\pi$ by $\tilde{\pi}$ the unique (by \cite{PV1} Proposition 1) probability measure satisfying $\tilde{\mc{L}}^*_{x,\mu}\tilde{\pi}=0$, $f$ by $\tilde{f}$, $a$ by $\tilde{a}$, $\R^d$ by $\R^j,$ $\R^{d\times d}$ by $\R^{j\times j}$, and $\mc{P}_2(\R^d)$ by $\mc{P}_2(\R^l)$ in their statements.

In particular, considering $\tilde{f}:\R^{2d}\times\R^{2d}\times \mc{P}_2(\R^d)\tto \R^{2d}$ given by $\tilde{f}(x,y,\mu)=[f(x_1,y_1,\mu),f(x_2,y_2,\mu)]^\top$ and $\tilde{a}:\R^{2d}\times\R^{2d}\times\mc{P}_2(\R^d)\tto \R^{2d\times 2d}$ given by $\tilde{a}(x,y,\mu)=\text{diag}[a(x_1,y_1,\mu),a(x_2,y_2,\mu)]$ where $x\in\R^{2d}=(x_1,x_2),x_1,x_2\in\R^d$ and $y\in\R^{2d}=(y_1,y_2),y_1,y_2\in\R^d$, we have $\tilde{\mc{L}}_{x,\mu}$ in the above discussion is equal to $\mc{L}^2_{x_1,x_2,\mu}$ from Equation \eqref{eq:2copiesgenerator}, and $\tilde{\pi}(dy;x,\mu)$ is equal to $\bar{\pi}(dy_1,dy_2;x_1,x_2,\mu)$ from Equation \eqref{eq:doublefrozeninvariantmeasure}.

Moreover, under Assumption \ref{assumption:retractiontomean}
\begin{align*}
\tilde{f}(x,y,\mu)\cdot y &=f(x_1,y_1,\mu)\cdot y_1+f(x_2,y_2,\mu)\cdot y_2 \\
&\leq -\beta [|y_1|^2+|y_2|^2] +2C\\
&=-\beta |y|^2+2C, \forall x,y\in\R^{2d},\mu\in\mc{P}_2(\R^d)
\end{align*}
and under assumption \ref{assumption:uniformellipticity}, writing $z\in \R^{2d}\setminus \br{0}$ as $(z_1,z_2)\in \R^d$ and taking any $x,y\in\R^{2d}$ and $\mu \in \mc{P}_2(\R^d)$ as before, we know
\begin{align*}
0<\lambda_-|z_i|^2\leq 2z_i^\top a(x_i,y_i,\mu)z_i\leq \lambda_+|z_i|^2,i=1,2
\end{align*}
and
\begin{align*}
z^\top \tilde{a}(x,y,\mu)z = z_1^\top a(x_1,y_1,\mu)z_1+z_2^\top a(x_2,y_2,\mu)z_2
\end{align*}
so
\begin{align*}
0<\lambda_-|z|^2=\lambda_-[|z_1|^2+|z_2|^2]\leq 2z^\top \tilde{a}(x,y,\mu)z\leq  \lambda_+[|z_1|^2+|z_2|^2]=\lambda_+|z|^2.
\end{align*}

Lastly, the desired regularity and growth properties of $\tilde{a}$ and $\tilde{f}$ are clearly inherited from those of $a$ and $f$. So indeed we can apply Lemmas \ref{lemma:Ganguly1DCellProblemResult} and \ref{lemma:explicitrateofgrowthofderivativesinparameters1D} to gain regularity of the Poisson Equations \eqref{eq:doublecorrectorproblem} and \eqref{eq:tildechi} by keeping these minor changes to the domains of the functions in the statements of the Lemmas in mind.

\end{remark}
\begin{proposition}\label{proposition:allneededregularity}
Under Assumptions \ref{assumption:uniformellipticity}-\ref{assumption:regularityofcoefficientsnew},
\begin{align*}
\Phi,\partial_{y_i}\Phi,\partial_{y_i}\partial_{y_j}\Phi&\in \mc{M}_p^{\hat{\bm{\zeta}}_1}(\R^d\times\R^d\times\mc{P}_2(\R^d);\R^d)\\
\tilde{\chi},\partial_{y_i}\tilde{\chi},\partial_{\bar{y}_i}\tilde{\chi},\partial_{y_i}\partial_{y_j}\tilde{\chi},\partial_{\bar{y}_i}\partial_{y_j}\tilde{\chi},\partial_{\bar{y}_i}\partial_{\bar{y}_j}\tilde{\chi}&\in \mc{M}_p^{\hat{\bm{\zeta}}_1}(\R^{2d}\times\R^{2d}\times\mc{P}_2(\R^d);\R^{d\times d})\\
\chi,\partial_{y_i}\chi,\partial_{\bar{y}_i}\chi,\partial_{y_i}\partial_{y_j}\chi,\partial_{\bar{y}_i}\partial_{y_j}\chi,\partial_{\bar{y}_i}\partial_{\bar{y}_j}\chi &\in \mc{M}_p^{\hat{\bm{\zeta}}_{2}}(\R^{2d}\times\R^{2d}\times\mc{P}_2(\R^d);\R^d)
\end{align*}
and
\begin{align*}
\Xi,\partial_{y_i}\Xi,\partial_{y_i}\partial_{y_j} \Xi \in \mc{M}_p^{\hat{\bm{\zeta}}}(\R^d\times\R^d\times\mc{P}_2(\R^d);\R^d)
\end{align*}
for $F=\gamma$,
\begin{align*}
\Xi,\partial_{y_i}\Xi,\partial_{y_i}\partial_{y_j} \Xi \in \mc{M}_p^{\hat{\bm{\zeta}}}(\R^d\times\R^d\times\mc{P}_2(\R^d);\R^{d\times d})
\end{align*}
for $F=D$, for all $i,j=1,...,d$, where $\Phi,\chi,\tilde{\chi},$ and $\Xi$ are the the unique classical solutions to the PDEs \eqref{eq:cellproblemold},\eqref{eq:doublecorrectorproblem},\eqref{eq:tildechi}, and \eqref{eq:driftcorrectorproblem} respectively, and $F$ is denoting the function which enters the right-hand side of Equation \eqref{eq:driftcorrectorproblem}. Moreover, $\bar{\gamma}\in \mc{M}_{b,L}^{\hat{\bm{\zeta}}}(\R^d\times\mc{P}_2(\R^d);\R^d)$ and $\bar{D}\in \mc{M}_{b,L}^{\hat{\bm{\zeta}}}(\R^d\times\mc{P}_2(\R^d);\R^{d\times d})$, where both these coefficients are defined in Equation \eqref{eq:limitingcoefficients}. Here $\hat{\bm{\zeta}},\hat{\bm{\zeta}}_1$ are as in Equation \eqref{eq:collectionsofmultiindices}, and
\begin{align*}
\hat{\bm{\zeta}}_2\ni \br{(0,j_1,0),(1,j_2,j_3),(2,j_4,(j_5,0)):j_1\in \br{0,1,...,4},j_2+j_3\leq 2,j_4+j_5\leq 1}.
\end{align*}
Additionally assuming \ref{assumption:uniformellipticityDbar}, we get further that $\bar{D}^{1/2}\in \mc{M}_{b,L}^{\hat{\bm{\zeta}}}(\R^d\times\mc{P}_2(\R^d);\R^{d\times d})$.
\begin{proof}
The fact that $\Phi,\partial_{y_i}\Phi,\partial_{y_i}\partial_{y_j}\Phi\in \mc{M}_p^{\hat{\bm{\zeta}}_1}(\R^d\times\R^d\times\mc{P}_2(\R^d))$ is immediate from Lemma \ref{lemma:explicitrateofgrowthofderivativesinparameters1D} applied to each $\Phi_l,l=1,...,d$.

To see $\tilde{\chi},\partial_{y_i}\tilde{\chi},\partial_{\bar{y}_i}\tilde{\chi},\partial_{y_i}\partial_{y_j}\tilde{\chi},\partial_{\bar{y}_i}\partial_{y_j}\tilde{\chi},\partial_{\bar{y}_i}\partial_{\bar{y}_j}\tilde{\chi}\in \mc{M}_p^{\hat{\bm{\zeta}}_1}(\R^{2d}\times\R^{2d}\times\mc{P}_2(\R^d);\R^{d\times d})$, we note that, as per Remark \ref{remark:regularityfordoubledequations}, Lemmas \ref{lemma:Ganguly1DCellProblemResult} and \ref{lemma:explicitrateofgrowthofderivativesinparameters1D} hold with $\mc{L}^2_{x,\bar{x},\mu}$ in the place of $\mc{L}_{x,\mu}$ and $\bar{\pi}$ from Equation \eqref{eq:doublefrozeninvariantmeasure} in the place of $\pi$. By product rule and triangle inequality that $b,\Phi\in \mc{M}_p^{\hat{\bm{\zeta}}_1}(\R^d\times\R^d\times\mc{P}_2(\R^d);\R^d)$ implies $G:\R^{2d}\times\R^{2d}\times\mc{P}(\R^d)\tto \R^{d\times d}$ defined by $G(\tilde{x},\tilde{y},\mu) = \Phi(\bar{x},\bar{y},\mu)\otimes b(x,y,\mu)$ where $x,\bar{x},y,\bar{y}\in\R^d$ are such that $\tilde{x}=(x,\bar{x}),\tilde{y}=(y,\bar{y})$ satisfies $G\in \mc{M}_p^{\hat{\bm{\zeta}}_1}(\R^{2d}\times\R^{2d}\times\mc{P}_2(\R^d);\R^{d\times d})$, and under Assumption \ref{assumption:regularityofcoefficientsnew}, $\tilde{f}\in \mc{M}^{\hat{\bm{\zeta}}_1}_p(\R^{2d}\times\R^{2d}\times \mc{P}_2(\R^d);\R^{2d}),$ and $\tilde{a}\in \mc{M}^{\hat{\bm{\zeta}}_1}_p(\R^{2d}\times\R^{2d}\times \mc{P}_2(\R^d);\R^{2d\times 2d})$ (using the notation of Remark \ref{remark:regularityfordoubledequations}). As noted after \eqref{eq:doublecorrectorproblem}, we also have $G$ integrates against $\bar{\pi}$ to $0$, so we get the result holds via Lemma \ref{lemma:explicitrateofgrowthofderivativesinparameters1D} applied to each $\tilde{\chi}_{k,l}$ for $k,l=1,...,d$.

The fact that $\chi,\partial_{y_i}\chi,\partial_{\bar{y}_i}\chi,\partial_{y_i}\partial_{y_j}\chi,\partial_{\bar{y}_i}\partial_{y_j}\chi,\partial_{\bar{y}_i}\partial_{\bar{y}_j}\chi \in \mc{M}_p^{\hat{\bm{\zeta}}_{2}}(\R^{2d}\times\R^{2d}\times\mc{P}_2(\R^d);\R^d)$ holds similarly. Here we see that, again by product rule and triangle inequality, $b,\Phi\in \mc{M}_p^{\hat{\bm{\zeta}}_1}(\R^d\times\R^d\times\mc{P}_2(\R^d);\R^d)$ implies $G:\R^{2d}\times\R^{2d}\times\mc{P}(\R^d)\tto \R^{d}$ defined by $G(\tilde{x},\tilde{y},\mu) = \partial_\mu\Phi(\bar{x},\bar{y},\mu)[x]b(x,y,\mu)$ where $x,\bar{x},y,\bar{y}\in\R^d$ are such that $\tilde{x}=(x,\bar{x}),\tilde{y}=(y,\bar{y})$ satisfies $G\in \mc{M}_p^{\hat{\bm{\zeta}}_2}(\R^{2d}\times\R^{2d}\times\mc{P}_2(\R^d);\R^d)$, where here one must track how many derivatives in $x,\mu,$ and $z$ of $\partial_\mu\Phi(x,y,\mu)[z]$ one gets by the fact that $\Phi\in \mc{M}_p^{\hat{\bm{\zeta}}_1}(\R^d\times\R^d\times\mc{P}_2(\R^d);\R^d)$. This is what results in the smaller collection of multi-indices $\hat{\bm{\zeta}}_2$. The result then follows again from Lemma \ref{lemma:explicitrateofgrowthofderivativesinparameters1D} applied to the coordinate functions of $\chi$.

For $\Xi,\partial_{y_i}\Xi,\partial_{y_i}\partial_{y_j} \Xi \in \mc{M}_p^{\hat{\bm{\zeta}}}(\R^d\times\R^d\times\mc{P}_2(\R^d);\R^d)$ for $F=\gamma$ and $\Xi,\partial_{y_i}\Xi,\partial_{y_i}\partial_{y_j} \Xi \in \mc{M}_p^{\hat{\bm{\zeta}}}(\R^d\times\R^d\times\mc{P}_2(\R^d);\R^{d\times d})$ for $F=D$, we first establish that $\gamma\in\mc{M}_p^{\hat{\bm{\zeta}}}(\R^d\times\R^d\times\mc{P}_2(\R^d);\R^d)$ and $D\in\mc{M}_p^{\hat{\bm{\zeta}}}(\R^d\times\R^d\times\mc{P}_2(\R^d);\R^{d\times d})$. By assumption all the coefficients which appear in the definition of $\gamma$ and $D$ in Equation \eqref{eq:limitingcoefficients} are already assumed to be in $\mc{M}_p^{\hat{\bm{\zeta}}}(\R^d\times\R^d\times\mc{P}_2(\R^d);\R^k)$ for the appropriate choices of $k$. Thus, to conclude that $\gamma$ and $D$ are in $\mc{M}_p^{\hat{\bm{\zeta}}}(\R^d\times\R^d\times\mc{P}_2(\R^d))$, we need only show that $\Phi,\partial_{y_i}\Phi,\partial_{x_i}\Phi,$ and $\partial_{x_i}\partial_{y_j}\Phi$ are in $\mc{M}_p^{\hat{\bm{\zeta}}}(\R^d\times\R^d\times\mc{P}_2(\R^d);\R^d)$ for $i,j=1,...,d$. Since $\hat{\bm{\zeta}}\subset \hat{\bm{\zeta}}_1$, we already know this for $\Phi$ and $\partial_{y_i}\Phi$. For $\partial_{x_i} \Phi$ and $\partial_{x_i}\partial_{y_j}\Phi$, we note that $\hat{\bm{\zeta}}_1$ is constructed so that $G\in \mc{M}_p^{\hat{\bm{\zeta}}_1}(\R^d\times\R^d\times\mc{P}_2(\R^d);\R^d)$ implies $\partial_{x_i} G\in \mc{M}_p^{\hat{\bm{\zeta}}}(\R^d\times\R^d\times\mc{P}_2(\R^d);\R^d),i=1,...,d$. So indeed $\gamma\in \mc{M}_p^{\hat{\bm{\zeta}}}(\R^d\times\R^d\times\mc{P}_2(\R^d);\R^d)$ and $D\in \mc{M}_p^{\hat{\bm{\zeta}}}(\R^d\times\R^d\times\mc{P}_2(\R^d);\R^{d\times d})$.

Now, we note that applying Lemma \ref{lem:regularityofaveragedcoefficients} to each coordinate function of $\gamma$ and $D$, this implies that $\bar{\gamma}\in \mc{M}_{b,L}^{\hat{\bm{\zeta}}}(\R^d\times\mc{P}_2(\R^d);\R^d)$ and $\bar{D}\in \mc{M}_{b,L}^{\hat{\bm{\zeta}}}(\R^d\times\mc{P}_2(\R^d);\R^{d\times d})$. By triangle inequality, this implies $\gamma-\bar{\gamma}\in \mc{M}_p^{\hat{\bm{\zeta}}}(\R^d\times\R^d\times\mc{P}_2(\R^d);\R^d)$ and $D-\bar{D}\in \mc{M}_p^{\hat{\bm{\zeta}}}(\R^d\times\R^d\times\mc{P}_2(\R^d);\R^{d\times d})$ so that Lemma \ref{lemma:explicitrateofgrowthofderivativesinparameters1D} yields $\Xi,\partial_{y_i}\Xi,\partial_{y_i}\partial_{y_j} \Xi \in \mc{M}_p^{\hat{\bm{\zeta}}}(\R^d\times\R^d\times\mc{P}_2(\R^d);\R^k)$ with $k=d$ for $F=\gamma$ and $k=d\times d$ for $F=D$.

Lastly, to gain $\bar{D}^{1/2}\in \mc{M}_{b,L}^{\hat{\bm{\zeta}}}(\R^d\times\mc{P}_2(\R^d);\R^{d\times d})$ from $\bar{D}\in \mc{M}_{b,L}^{\hat{\bm{\zeta}}}(\R^d\times\mc{P}_2(\R^d);\R^{d\times d})$ under Assumption \ref{assumption:uniformellipticityDbar}, we use that mapping which takes a positive-definite matrix to its unique positive-definite square root is Fr\'echet differentiable up to arbitrary order, with all derivatives being bounded on sets of uniformly bounded, uniformly positive definite matrices (see Equation (6) in \cite{MatrixRoot}). Thus, by chain rule and the fact that $\bar{D}\in \mc{M}_{b,L}^{\hat{\bm{\zeta}}}(\R^d\times\mc{P}_2(\R^d);\R^{d\times d})$, we can see indeed that $\bar{D}^{1/2}\in \mc{M}_{b,L}^{\hat{\bm{\zeta}}}(\R^d\times\mc{P}_2(\R^d);\R^{d\times d})$. See also \cite{CLX} Lemma A.7 for how the growth of  derivatives of $\bar{D}^{1/2}$ can be controlled in terms of $\bar{\lambda}_-$ and the derivatives of $\bar{D}$.
\end{proof}
\end{proposition}

\section{On Differentiation of Functions on Spaces of Measures}\label{Appendix:LionsDifferentiation}
We will need the following two definitions from \cite{CD}:

\begin{defi}
\label{def:lionderivative}
Given a function $u:\mc{P}_2(\R^d)\tto \R$, we may define a lifting of $u$ to $\tilde{u}:L^2(\tilde\W,\tilde\F,\tilde\Prob;\R^d)\tto \R$ via $\tilde u (X) = u(\mc{L}(X))$ for $X\in L^2(\tilde\W,\tilde\F,\tilde\Prob;\R^d)$. Here we assume $\tilde\W$ is a Polish space, $\tilde\F$ its Borel $\sigma$-field, and $\tilde\Prob$ is an atomless probability measure (since $\tilde\W$ is Polish, this is equivalent to every singleton having zero measure).

Here, denoting by $\mu(|\cdot|^r)\coloneqq \int_{\R^d}|x|^r \mu(dx)$ for $r>0$,
\begin{align*}
\mc{P}_2(\R^d) \coloneqq \br{ \mu\in \mc{P}(\R^d):\mu(|\cdot|^2)= \int_{\R^d}|x|^2 \mu(dx)<\infty}.
\end{align*}
$\mc{P}_2(\R^d)$ is a Polish space under the $L^2$-Wasserstein distance
\begin{align*}
\bb{W}_2 (\mu_1,\mu_2)\coloneqq \inf_{\pi \in\mc{C}_{\mu_1,\mu_2}} \biggl[\int_{\R^d\times\R^d} |x-y|^2 \pi(dx,dy)\biggr]^{1/2},
\end{align*}
where $\mc{C}_{\mu_1,\mu_2}$ denotes the set of all couplings of $\mu_1,\mu_2$.

We say $u$ is \textbf{L-differentiable} or \textbf{Lions-differentiable} at $\mu_0\in\mc{P}_2(\R^d)$ if there exists a random variable $X_0$ on some $(\tilde\W,\tilde\F,\tilde\Prob)$ satisfying the above assumptions such that $\mc{L}(X_0)=\mu_0$ and $\tilde u$ is Fr\'echet differentiable at $X_0$.

The Fr\'echet derivative of $\tilde u$ can be viewed as an element of $L^2(\tilde\W,\tilde\F,\tilde\Prob;\R^d)$ by identifying $L^2(\tilde\W,\tilde\F,\tilde\Prob;\R^d)$ and its dual. From this, one can find that if $u$ is L-differentiable at $\mu_0\in\mc{P}_2(\R^d)$, there is a deterministic measurable function $\xi: \R^d\tto \R^d$ such that $D\tilde{u}(X_0)=\xi(X_0)$, and that $\xi$ is uniquely defined $\mu_0$-almost everywhere on $\R^d$. We denote this equivalence class of $\xi\in L^2(\R^d,\mu_0;\R^d)$ by $\partial_\mu u(\mu_0)$ and call $\partial_\mu u(\mu_0)[\cdot]:\R^d\tto \R^d$ the \textbf{Lions derivative} of $u$ at $\mu_0$. Note that this definition is independent of the choice of $X_0$ and $(\tilde\W,\tilde\F,\tilde\Prob)$. See \cite{CD} Section 5.2.

To avoid confusion when $u$ depends on more variables than just $\mu$, if $\partial_\mu u(\mu_0)$ is differentiable at $z_0\in\R^d$, we denote its derivative at $v_0$ by $\partial_z\partial_\mu u(\mu_0)[z_0]$.
\end{defi}
\begin{defi}
\label{def:fullyC2}
(\cite{CD} Definition 5.83) We say $u:\mc{P}_2(\R^d)\tto \R^d$ is \textbf{Fully} $\mathbf{C^2}$ if the following conditions are satisfied:
\begin{enumerate}
\item $u$ is $C^1$ in the sense of L-differentiation, and its first derivative has a jointly continuous version $\mc{P}_2(\R^d)\times \R^d)\ni (\mu,z)\mapsto \partial_\mu u(\mu)[z]\in\R^d$.
\item For each fixed $\mu\in\mc{P}_2(\R^d)$, the version of $\R^d\ni z\mapsto \partial_\mu u(\mu)[z]\in\R^d$ from the first condition is differentiable on $\R^d$ in the classical sense and its derivative is given by a jointly continuous function $\mc{P}_2(\R^d)\times \R^d)\ni (\mu,z)\mapsto \partial_z\partial_\mu u(\mu)[z]\in\R^{d\times d}$.
\item For each fixed $z\in \R^d$, the version of $\mc{P}_2(\R^d)\ni \mu\mapsto \partial_\mu u(\mu)[z]\in \R^d$ in the first condition is continuously L-differentiable component-by-component, with a derivative given by a function $\mc{P}_2(\R^d)\times \R^d\times \R^d\ni(\mu,z,\bar{z})\mapsto \partial^2_\mu u(\mu)[z][\bar{z}]\in\R^{d\times d}$ such that for any $\mu\in\mc{P}_2(\R^d)$ and $X\in L^2(\tilde\W,\tilde\F,\tilde\Prob;\R^d)$ with $\mc{L}(X)=\mu$, $\partial^2_\mu u(\mu)[z][X]
$ gives the Fr\'echet derivative at $X$ of $L^2(\tilde\W,\tilde\F,\tilde\Prob;\R^d)\ni X'\mapsto \partial_\mu u(\mc{L}(X'))[z]$ for every $z\in\R^d$. Denoting $\partial^2_\mu u(\mu)[z][\bar{z}]$ by $\partial^2_\mu u(\mu)[z,\bar{z}]$, the map $\mc{P}_2(\R^d)\times \R^d\times \R^d\ni(\mu,z,\bar{z})\mapsto \partial^2_\mu u(\mu)[z,\bar{z}]$ is also assumed to be continuous in the product topology.
\end{enumerate}
\end{defi}
\begin{remark}\label{remark:sufficientconditionsforItos}
Conditions 1) and 2) from Definition \ref{def:fullyC2} along with local boundedness of $\partial_\mu u$ and $\partial_z\partial_\mu u$ is sufficient to apply It\^o's formula for measure-dependent functions as used in the proofs Section \ref{section:mckeanvlasovergodictheorems} and the proof of Theorem \ref{theo:mckeanvlasovaveraging} - see Section 5.6.4 in \cite{CD}.
\end{remark}
\begin{remark}\label{remark:thirdLionsDerivative}
In this paper we will in fact also look at functions $u:\mc{P}_2(\R^d)\tto \R^d$ which are required to have $3$ Lions Derivatives. We will assume such functions are \textbf{Fully} $\mathbf{C^2}$, and satisfy:
\begin{enumerate}\setcounter{enumi}{3}
\item For each each fixed $\mu\in\mc{P}_2(\R^d)$ the version of  $\R^d\times \R^d \ni (z_1,z_2)\mapsto \partial^2_\mu u(\mu)[z_1,z_2]\in\R^{d\times d}$ in the third condition is differentiable on $\R^{2d}$ in the classical sense and its derivative is given by a jointly continuous function
    $\partial_z\partial^2_\mu u(\mu)[z_1,z_2] = (\partial_{z_1}\partial^2_\mu u(\mu)[z_1,z_2],\partial_{z_2}\partial^2_\mu u(\mu)[z_1,z_2])\in\R^{d\times d\times d}\times\R^{d\times d\times d}$ for $\mc{P}_2(\R^d)\times \R^d\times\R^d\ni (\mu,z_1,z_2).$
\item For each fixed $(z_1,z_2)\in \R^{2d}$, the version of $\mc{P}_2(\R^d)\ni \mu\mapsto \partial^2_\mu u(\mu)[z_1,z_2]\in \R^{d\times d}$ in the third condition is continuously L-differentiable component-by-component, with a derivative given by a function $\mc{P}_2(\R^d)\times \R^d\times \R^d \times \R^d\ni(\mu,z_1,z_2,z_3)\mapsto \partial^3_\mu u(\mu)[z_1,z_2][z_3]\in\R^{d\times d\times d}$ such that for any $\mu\in\mc{P}_2(\R^d)$ and $X\in L^2(\tilde\W,\tilde\F,\tilde\Prob;\R^d)$ with $\mc{L}(X)=\mu$, $\partial^3_\mu u(\mu)[z_1,z_2][X]
$ gives the Fr\'echet derivative at $X$ of $L^2(\tilde\W,\tilde\F,\tilde\Prob;\R^d)\ni X'\mapsto \partial^2_\mu u(\mc{L}(X'))[z_1,z_2]$ for every $(z_1,z_2)\in\R^{2d}$. Denoting $\partial^3_\mu u(\mu)[z_1,z_2][z_3]$ by $\partial^2_\mu u(\mu)[z_1,z_2,z_3]$, the map $\mc{P}_2(\R^d)\times \R^d\times \R^d\times \R^d\ni(\mu,z_1,z_2,z_3)\mapsto \partial^3_\mu u(\mu)[z_1,z_2,z_3]$ is also assumed to be continuous in the product topology.
\end{enumerate}
Though we don't require higher than 3 Lions derivatives in this paper, when we state general results for higher Lions derivatives in terms of the spaces from Definition \ref{def:lionsderivativeclasses}, we assume the analogous higher continuity.
\end{remark}


\begin{bibdiv}
\begin{biblist}

\end{biblist}
\end{bibdiv}


\begingroup
\begin{bibdiv}
\begin{biblist}

\bib{BCCP}{article}{

        title={A Non-Maxwellian Steady Distribution for One-Dimensional Granular Media},

        author={D. Benedetto},
        author={E. Caglioti},
        author={J. A. Carrillo},
        author={M. Pulvirenti},

        journal={Journal of Statistical Physics},

        volume={91},

        date={1998},

        pages={979--990}
}

\bib{Bensoussan}{book}{

        title = {Asymptotic Analysis for Periodic Structures},

        author = {A. Bensoussan},

      author = {J. L. Lions},

      author = {G. Papanicolau},

        date = {1978},

        publisher = { North Holland},

        address = {Amsterdam}

}

\bib{BS}{arxiv}{

        title={Large deviations for interacting multiscale particle systems},

        author={Z. Bezemek},

        author={K. Spiliopoulos},

        date={2020},
      arxiveprint={
            arxivid={2011.03032},
            arxivclass={math.PR},
      }
}

\bib{BS_MDP2022}{arxiv}{

        title={Moderate deviations for fully coupled multiscale weakly interacting particle systems},

        author={Z. Bezemek},

        author={K. Spiliopoulos},

        date={2022},
   arxiveprint={
            arxivid={2202.08403},
            arxivclass={math.PR},
                  }
}
\bib{BorkarGaitsgory}{article}{

        title={Averaging of singularly perturbed controlled stochastic differential equations},

        author={V. Borkar},
        author={V. Gaitsgory},

        journal={Applied Mathematics and Optimization},

        volume={56},

        number={2},

        date={2007},

        pages={169--209}
}

\bib{BLPR}{article}{

        title={Mean-field stochastic differential equations and associated PDEs},

        author={R. Buckdahn},
        author={J. Li},
        author={S. Peng},
        author={C. Rainer},

        journal={Ann. Probab.},

        volume={45},

        number={2},

        date={2017},

        pages={824--878}
}
\bib{NotesMFG}{report}{

        title = { Notes on mean field games (from P. L. Lions’ lectures at Collège de France)},

        author = {P. Cardaliaguet},

        date = {2013},

        status= {unpublished},

        eprint = {https://www.ceremade.dauphine.fr/~cardaliaguet/MFG20130420.pdf},
}

\bib{CD}{book}{

        title = {Probabilistic Theory of Mean Field Games with Applications I},

        author = {R. Carmona},

      author = {F. Delarue},

        date = {2018},

        publisher = { Springer},

        address = {NY}

}

\bib{CerraiBook}{book}{

        title = {Second Order PDE’s in Finite and Infinite Dimension: A Probabilistic Approach},

        author = {S. Cerrai},

        date = {2001},

        publisher = { Springer},

        address = {NY}

}

\bib{CCD}{arxiv}{

        title={A Probabilistic approach to classical solutions of the master equation for large population equilibria},

        author={J.F. Chassagneux},

        author={D. Crisan},

        author={F. Delarue},

        date={2015},
      arxiveprint={
            arxivid={1411.3009},
            arxivclass={math.PR},
      }
}

\bib{CST}{article}{

        title={Weak quantitative propagation of chaos via differential calculus on the space of measures},

        author={J.F. Chassagneux},

        author={L. Szpruch},

        author={A. Tse},

        date={2022},
        journal={Ann. Appl. Probab.},
        volume={32},
        number={3},
        pages={1929-1969}
}

\bib{DF2}{article}{

        title={From the backward Kolmogorov PDE on the Wasserstein space to propagation of chaos for McKean-Vlasov SDEs},

        author={P. E. Chaudru de Raynal},

        author={N. Frikha},

        date={2021},
        journal={Journal de Mathématiques Pures et Appliqués},
        volume={156},
        number={2}

}
\bib{DF1}{arxiv}{

        title={Well-posedness for some non-linear diffusion processes and related PDE on the Wasserstein space},

        author={P. E. Chaudru de Raynal},

        author={N. Frikha},

        date={2018},
      arxiveprint={
            arxivid={1811.06904},
            arxivclass={math.CA},
      }
}

\bib{CLX}{article}{

        title={Approximation to stochastic variance reduced gradient Langevin dynamics by stochastic delay differential equations},

        author={P. Chen},

        author={J. Lu},
        author={L. Xu},
        date={2022},
        journal={Applied Mathematics \& Optimization},
        volume={85},
        number={15}
}

\bib{CM}{article}{

        title={Smoothing properties of McKean-Vlasov SDEs},

        author={D. Crisan},

        author={E. McMurray},

        journal={Probability Theory and Related Fields},

        volume={171},
        number={2},
        date={2018},

        pages={97–-148}

}

\bib{DKZ}{article}{

        title={Regularity of solutions of linear stochastic equations in hilbert spaces},

        author={G. Da prato},

        author={S. Kwapie\v{n}},

        author={J. Zabczyk},

        journal={Stochastics},

        volume={23},

        date={1988},

        pages={1--23}

}

\bib{Dawson}{article}{

        title={Critical dynamics and fluctuations for a mean-field model of cooperative behavior},

        author={D. A. Dawson},

        journal={J. Stat. Phys.},

        volume={31},

        date={1983},

        pages={29--85}

}
\bib{MatrixRoot}{article}{

        title={A Taylor expansion of the square root matrix function},

        author={P. Del Moral},

        author={A. Niclas},

        journal={Journal of Mathematical Analysis and Applications},

        volume={465},

        number={1},


        date={2018},

        pages={259--266}

}
\bib{DLR}{article}{

        title={From the master equation to mean field game limit theory: a central limit theorem},

        author={F. Delarue},

        author={D. Lacker},

        author={K. Ramanan},

        journal={Electron. J. Probab.},

        volume={24},


        date={2019},

        pages={1--54}

}


\bib{delgadino2020}{article}{

        title={On the diffusive-mean field limit for weakly interacting diffusions exhibiting phase transitions},

        author={M. G. Delgadino},
        author={R. S. Gvalani},
        author={G. A. Pavliotis},

        date={2021},
        journal={Archive for Rational Mechanics and Analysis},
        volume={241},
        pages={91--148}
}

\bib{DM}{book}{

        title = {Methods of Nonlinear Analysis: Applications to Differential Equations},

        author = {P. Dr\'abek},
        author = {J. Milota},

        date = {2015},

        publisher = {Birkh\"auser},

        address = {Basel},
        edition = {2},

}
\bib{DP}{arxiv}{

        title={Brownian Motion in an N-scale periodic Potential},

        author={A. B. Duncan},

        author={G. A. Pavliotis},

        date={2016},
      arxiveprint={
            arxivid={1605.05854},
            arxivclass={math.ph},
      }
}

\bib{EidelmanBook}{book}{

        title = {Parabolic Systems},

        author = {S. D. Eidel'man},
        publisher = {Nauka},
        address={Moscow},
        date={1965},
        language={Russian},

        translation={
        title = {Parabolic Systems},
        date = {1969},
        publisher = {North-Holland Publishing},
        address = {Amsterdam},
        language={English}
            }
}

\bib{FM}{article}{

        title={A Hilbertian approach for fluctuations on the McKean-Vlasov model},
        author={B. Fernandez},
        author={S. M\'el\'eard},

        journal={Stochastic Processes and their Applications},
        volume={71},
        date={1997},
        pages={33--53}
}
\bib{Friedman}{book}{

        title = {Partial Differential Equations of Parabolic Type},

        author = {A. Friedman},

        date = {1983},

        publisher = { R. E. Krieger Publishing},

        address = {Malabar, Florida}

}

\bib{GS}{article}{

        title={Inhomogeneous functionals and approximations of invariant distributions of ergodic diffusions: Central limit theorem and moderate deviation asymptotics},

        author={A. Ganguly},
        author={P. Sundar},

        journal={Stochastic Processes and their Applications},

        volume={133},

        date={2021},

        pages={74--110}

}

\bib{Garnier1}{article}{

        title={Large deviations for a mean field model of systemic risk},

        author={J. Garnier},
        author={G. Papanicolaou},
        author={T. W. Yang},

        journal={SIAM Journal of financial mathematics},

        volume={4},
        number={1},

        date={2013},

        pages={151--184}

}
\bib{Garnier2}{article}{

        title={Consensus convergence with stochastic effects},

        author={J. Garnier},
        author={G. Papanicolaou},
        author={T. W. Yang},

        journal={Vietnam Journal of mathematics},

        volume={45},
        number={1-2},

        date={2017},

        pages={51--75}

}
\bib{HSS}{article}{

        title={McKean–Vlasov SDEs under measure dependent Lyapunov conditions},

        author={W. R.P. Hammersley},
        author={D. \v{S}i\v{s}ka},
        author={\L. Szpruch},

        journal={Ann. Inst. H. Poincaré Probab. Statist.},

        volume={57},
        number={2},

        date={2021},

        pages={1032--1057}

}

\bib{HM}{article}{

        title={Tightness problem and stochastic evolution equation arising from fluctuation phenomena for interacting diffusions},

        author={M. Hitsuda},
        author={I. Mitoma},

        journal={Journal of Multivariate Analysis},

        volume={19},
        number={2},

        date={1986},

        pages={311--328}

}

\bib{HLL}{arxiv}{

        title={Strong Convergence Rates in Averaging Principle for Slow-Fast McKean-Vlasov SPDEs},

        author={W. Hong},

        author={S. Li},
        author={W. Liu},

        date={2021},
      arxiveprint={
            arxivid={2107.14401},
            arxivclass={math.PR},
      }
}
\bib{HLS}{arxiv}{

        title={Diffusion Approximation for Multi-Scale McKean-Vlasov SDEs Through Different Methods},


        author={W. Hong},

        author={S. Li},
        author={X. Sun},
        date={2022},
      arxiveprint={
            arxivid={2206.01928},
            arxivclass={math.PR},
      }
}
\bib{HLLS}{arxiv}{

        title={Central Limit Type Theorem and Large Deviations for Multi-Scale McKean-Vlasov SDEs},

        author={W. Hong},

        author={S. Li},
        author={W. Liu},
        author={X. Sun},

        date={2021},
      arxiveprint={
            arxivid={2112.08203},
            arxivclass={math.PR},
      }
}

\bib{HW}{arxiv}{

        title={Derivative Estimates on Distributions of McKean-Vlasov SDEs},

        author={X. Huang},

        author={F.Y. Wang},

        date={2020},
      arxiveprint={
            arxivid={2006.16731},
            arxivclass={math.PR},
      }
}

\bib{HWSingular}{article}{

        title={Distribution dependent SDEs with singular coefficients},

        author={X. Huang},

        author={F.Y. Wang},


        journal={Stochastic Processes and their Applications},

        volume={129},
        number={11},

        date={2019}
}
\bib{HWY}{arxiv}{

        title={Weak Solution and Invariant Probability Measure for McKean-Vlasov SDEs with Integrable Drifts},

        author={X. Huang},

        author={S. Wang},
        author={F.F. Yang},
        date={2021},
      arxiveprint={
            arxivid={2108.05802},
            arxivclass={math.PR},
      }
}

\bib{KS}{book}{

        title = {Brownian Motion and Stochastic Calculus},

        author = {I. Karatzas},

      author = {S. Shreve},

        date = {1998},

        publisher = { Springer},

        address = {NY}

}
\bib{KCBFL}{article}{

        title={Emergent Behaviour in Multi-particle Systems with Non-local Interactions},

        author={T. Kolokolnikov},

        author={A. Bertozzi},

        author={R. Fetecau},

        author={M. Lewis},

        journal={Physica D},
        volume={260},
        date={2013},
        pages={1-4}
}
\bib{KSS}{arxiv}{

        title={Well-posedness and averaging principle of McKean-Vlasov SPDEs driven by cylindrical $\alpha$-stable process},

        author={M. Kong},

        author={Y. Shi},
        author={X. Sun},

        date={2021},
      arxiveprint={
            arxivid={2106.05561},
            arxivclass={math.PR},
      }
}


\bib{KX}{article}{

        title={A stochastic evolution equation arising from the fluctuations of a class of interacting particle systems},

        author={T. G. Kurtz},

        author={J. Xiong},

        journal={Communications in Mathematical Sciences},
        volume={2},
        number={3},
        date={2004},
        pages={325--358}
}
\bib{Lacker}{article}{

        title={Limit theory for controlled McKean-Vlasov dynamics},

        author={D. Lacker},

        journal={SIAM Journal on Control and Optimization},
        volume={55},
        number={3},
        date={2017},
        pages={1641--1672}
}

\bib{LSU}{book}{

        title = {Linear and Quasi-linear Equations of Parabolic Type},

        author = {O. A.  Lady\v{z}enskaja},

        author = {V. A.  Solonnikov},

        author = {N. N.  Ural’ceva},

        date = {1968},

        publisher = {AMS},

        address = {NY}

}
\bib{LWX}{arxiv}{

        title={Poisson equation on Wasserstein space and diffusion approximations for McKean-Vlasov equation},

        author={Y. Li},
        author={F. Wu},
        author={L. Xie},
        date={2022},
      arxiveprint={
            arxivid={2203.12796},
            arxivclass={math.PR},
      }
}

\bib{Mehri}{article}{

        title={Weak solutions to Vlasov–McKean equations under Lyapunov-type conditions},

        author={S. Mehri},
        author={W. Stannat},

        journal={Stochastics and Dynamics},

        volume={19},
        number={9},

        date={2019}
}

\bib{MT}{book}{

        title = {Collective dynamics from bacteria to crowds: An excursion through modeling, analysis and simulation},
        volume={533},
        series={CISM International Centre for Mechanical Sciences. Courses and Lectures},
        editor = {A. Muntean},
        editor = {F. Toschi},

        date = {2014},

        publisher = {Springer},

        address = {Vienna}

}
\bib{MS}{article}{

        title={Moderate deviations principle for systems of slow-fast diffusions},

        author={M. R. Morse},

        author={K. Spiliopoulos},

        journal={Asymptotic Analysis},
        volume={105},
        number={3--4},
        date={2017},
        pages={97--135}
}

\bib{PV1}{article}{

        title={On Poisson equation and diffusion approximation I},

        author={E. Pardoux},
        author={A. Y. Veretennikov},
        journal={Annals of Probability},

        volume={29},
        number={3},

        date={2001},

        pages={1061--1085}

}

\bib{PV2}{article}{

        title={On Poisson equation and diffusion approximation II},

        author={E. Pardoux},
        author={A. Y. Veretennikov},
        journal={Annals of Probability},

        volume={31},
        number={3},

        date={2003},

        pages={1166--1192}

}

\bib{PavliotisSPA}{book}{

        title = {Stochastic Processes and Applications: Diffusion Processes, the Fokker-Planck and Langevin Equations},

        author = {G. Pavliotis},

        date = {2014},

        publisher = { Springer},

        address = {NY}

}

\bib{PS}{book}{

        title = {Multiscale Methods},

        author = {G. Pavliotis},

      author = {G. A. Stuart},

        date = {2008},

        publisher = { Springer},

        address = {NY}

}

\bib{QW}{arxiv}{

        title={Efficient filtering for multiscale McKean-Vlasov Stochastic differential equations},

        author={H. Qiao},
        author={W. Wei},
        date={2022},
      arxiveprint={
            arxivid={2206.05037},
            arxivclass={math.PR},
      }
}
\bib{Ren}{arxiv}{

        title={Singular McKean-Vlasov SDEs: Well-Posedness, Regularities and Wangs Harnack Inequality},

        author={P. Ren},

        date={2021},
      arxiveprint={
            arxivid={2110.08846},
            arxivclass={math.PR},
      }
}
\bib{RocknerFullyCoupled}{article}{

        title={Diffusion approximation for fully coupled stochastic differential equations},

        author={M. R\"ockner},

        author={L. Xie},

        date={2021},
        volume = {49},
        number = {3},

        journal = {The Annals of Probability},
        pages= {101--122},
}
\bib{RocknerMcKeanVlasov}{article}{

        title={Strong convergence order for slow-fast McKean-Vlasov stochastic differential equations},

        author={M. R\"ockner},

        author={X. Sun},
        author={Y. Xie},

        date={2021},
        journal={Ann. Inst. H. Poincaré Probab. Statist.},
        volume={57},
        number={1},
        pages={547--576}

}

\bib{RocknerHolderContinuous}{arxiv}{

        title={Strong and weak convergence in the averaging principle for SDEs with Hölder coefficients},

        author={M. R\"ockner},

        author={X. Sun},
        author={Y. Xie},

        date={2019},
      arxiveprint={
            arxivid={1907.09256},
            arxivclass={math.PR},
      }
}


\bib{RocknerSPDE}{arxiv}{

        title={Asymptotic behavior of multiscale stochastic partial differential equations},

        author={M. R\"ockner},

        author={L. Xie},
        author={L. Yang},

        date={2020},
      arxiveprint={
            arxivid={2010.14897},
            arxivclass={math.PR},
      }
}
\bib{Spiliopoulos2014Fluctuations}{article}{

        title={Fluctuation analysis and short time asymptotics for multiple scales diffusion processes},

        author={K. Spiliopoulos},
        journal={Stochastics and Dynamics},

        volume={14},
        number={3},

        date={2014},

        pages={1350026}

}
\bib{LossFromDefault}{article}{

        title={Fluctuation Analysis for the Loss From Default},

        author={K. Spiliopoulos},
        author={J. A. Sirignano},
        author={K. Giesecke},
        journal={Stochastic Processes and their Applications},

        volume={124},
        number={7},

        date={2014},

        pages={2322--2362}

}

\bib{Tse}{article}{

        title={Higher order regularity of nonlinear Fokker-Planck PDEs with respect to the measure component},

        author={A. Tse},
        journal={Journal de Math\'ematiques Pures et Appliqu\'ees},

        volume={150},


        date={2021},

        pages={134--180}

}

\bib{Veretennikov1987}{article}{

        title={Bounds for the mixing rates in the theory of stochastic equations},

        author={A. Yu. Veretennikov},
        journal={Theory Probab. Appl.},

        volume={32},

        date={1987},

        pages={273--281}

}

\bib{Wang}{article}{

        title={Distribution dependent SDEs for Landau type equations},

        author={F. Y. Wang},
        journal={Stochastic Processes and their Applications},

        volume={128},
        number={2},

        date={2017},

        pages={595-–621}

}

\bib{XLLM}{article}{

        title={Strong Averaging Principle for Two-Time-Scale Stochastic McKean-Vlasov Equations},

        author={J. Xu},

        author={J. Liu},
        author={J. Liu},
        author={Y. Miao},
        journal={Applied Mathematics and Optimization},

        volume={84},
        date={2021},

        pages={837-–867}

}
\bib{Zwanzig}{article}{

        title={Diffusion in a rough potential},

        author={R. Zwanzig},

        journal={Proc. Natl. Acad. Sci.},

        volume={85},

        date={1988},

        pages={2029--2030},
        address = {USA}
}


\end{biblist}
\end{bibdiv}
\endgroup
\end{document}